\documentclass[a4paper, 11pt]{amsart}

\usepackage{enumitem, amsmath, amssymb, amsthm, color, tikz, caption, tikz-cd, booktabs, array, multirow, hyperref, mathrsfs, calc, mathtools, bigints, verbatim, stmaryrd, bm}
\usepackage[capitalize]{cleveref}
\usepackage[savepos]{zref}
\usepackage[toc]{appendix}
\usetikzlibrary{calc, intersections, decorations.markings, decorations.pathmorphing, shapes, arrows.meta}

\newcommand{\lb}{\left(}
\newcommand{\rb}{\right)}
\newcommand{\eps}{\varepsilon}
\newcommand{\Q}{\mathbb{Q}}
\newcommand{\R}{\mathbb{R}}
\newcommand{\C}{\mathbb{C}}
\newcommand{\Z}{\mathbb{Z}}
\renewcommand{\P}{\mathbb{P}}
\newcommand{\pd}{\partial}
\newcommand{\diff}{\mathrm{d}}
\renewcommand{\phi}{\varphi}
\newcommand{\ip}[2]{\langle #1, #2 \rangle}
\newcommand{\lspan}[1]{\left\langle #1 \right\rangle}
\let\oldref\ref
\renewcommand{\eqref}[1]{\textup{(\oldref{#1})}}
\renewcommand{\tilde}{\widetilde}
\newcommand{\conj}[1]{\overline{#1}}
\newcommand{\ev}{\operatorname{ev}}
\newcommand{\SL}{\mathrm{SL}}
\newcommand{\PSL}{\mathrm{PSL}}
\newcommand{\SU}{\mathrm{SU}}
\newcommand{\U}{\mathrm{U}}
\newcommand{\PSU}{\mathrm{PSU}}
\newcommand{\crit}{\operatorname{Crit}}
\newcommand{\CO}{\operatorname{\mathcal{CO}}}
\newcommand{\codim}{\operatorname{codim}} % this ensures the correct spacing

\renewcommand{\ker}{\operatorname{ker}}

\newcommand{\End}{\operatorname{End}}
\newcommand{\rank}{\operatorname{rank}}
\newcommand{\gr}{\operatorname{gr}}
\newcommand{\qcup}{\mathbin{*}}
\newcommand{\Char}{\operatorname{char}}
\newcommand{\PD}{\operatorname{PD}}
\newcommand{\id}{\operatorname{id}}
\newcommand{\begen}{\begin{enumerate}[label=(\roman*), align=right]}
\newcommand{\ssE}[1]{E_{#1}} % for spectral sequence pages
\newcommand{\symp}{\mathrm{Symp}}
\newcommand{\hol}{\operatorname{hol}}

\newcommand{\flow}[7][]{\begin{scope}[#1]
\pgfmathsetseed{#4}
\draw[decorate, decoration={random steps, segment length=0.17cm, amplitude=0.08cm, pre=lineto, pre length=0.1cm, post=lineto, post length=0.1cm}, rounded corners=0.04cm, ->, shorten >=-1pt, line cap=butt] #2 -- ($#2!0.5!#3$) node(midpt){};
\draw ($(midpt)+(#6, #7)$) node{\footnotesize #5};
\draw[decorate, decoration={random steps, segment length=0.17cm, amplitude=0.08cm, pre=lineto, pre length=0.1cm, post=lineto, post length=0.1cm}, rounded corners=0.04cm, line cap=butt] ($#2!0.5!#3$) -- #3;
\end{scope}
} % draws Morse flow line from #2 to #3 with random seed #4, and puts label #5 at midpoint + (#6, #7); the optional first argument is the colour

\newcommand{\disc}[4]{\draw ($(#2, 0)+#3*(0, 1)$) arc (0:180: #2*0.5-#1*0.5 and #2*0.1-#1*0.1);
\draw[fill=white, fill opacity=0.65] ($(#1, 0)+#3*(0, 1)$) arc (180:360: #2*0.5-#1*0.5 and #2*0.1-#1*0.1) -- ($(#2, 0)+#3*(0, 1)$) arc (0:180: #2*0.5-#1*0.5 and #2*0.5-#1*0.5) -- cycle;
\draw ($(#1*0.5+#2*0.5, #2*0.25-#1*0.25)+#3*(0, 1)$) node{#4};
} % draws disc from (#1, #3) to (#2, #3) labelled by #4

\newcommand{\critpt}[7][]{\draw (#2, #3) node[blob, #1](#6){};
\draw (#2+#4, #3+#5) node{$#7$};
} % draws blob at (#2, #3) called (#6) and labels it by $#7$ at offset of (#4, #5); the optional first argument is the colour

\newcommand{\sphere}[2]
{
\def\u{-0.05}
\def\v{0.2}
\begin{scope}[shift={#1}]
\draw[domain=-1:181, samples=100, smooth] plot ({#2*cos(atan(\u)+\x)}, {#2*sin(atan(\u)+\x)}, {0});
\draw[domain=0:359, samples=200, fill=white, fill opacity=0.65, smooth cycle] plot ({cos(\x)*#2/sqrt((1+\u^2)*cos(\x)^2+2*\u*\v*cos(\x)*sin(\x)+(1+\v^2)*sin(\x)^2)}, {(\u*cos(\x)+\v*sin(\x))*#2/sqrt((1+\u^2)*cos(\x)^2+2*\u*\v*cos(\x)*sin(\x)+(1+\v^2)*sin(\x)^2)}, {sin(\x)*#2/sqrt((1+\u^2)*cos(\x)^2+2*\u*\v*cos(\x)*sin(\x)+(1+\v^2)*sin(\x)^2)});
\draw[domain=180:360, samples=100, smooth] plot ({#2*cos(atan(\u)+\x)}, {#2*sin(atan(\u)+\x)}, {0});

\end{scope}
} % arguments (x, y) position of centre, radius; have to set x=1cm, y={(-0.05cm, 0.2cm)}, z={(0cm, 1cm)} before calling this command

\let\originalleft\left
\let\originalright\right
\renewcommand{\left}{\mathopen{}\mathclose\bgroup\originalleft}
\renewcommand{\right}{\aftergroup\egroup\originalright}

\DeclarePairedDelimiter{\norm}{\lVert}{\rVert}

\theoremstyle{plain}
\newtheorem{thm}{Theorem}[subsection]
\newtheorem{lem}[thm]{Lemma}
\newtheorem{prop}[thm]{Proposition}
\newtheorem{cor}[thm]{Corollary}

\newtheorem{mthm}{Theorem}

\theoremstyle{remark}
\newtheorem{rmkx}[thm]{Remark}
\newenvironment{rmk}
  {\pushQED{\qed}\rmkx}
  {\popQED\endrmkx}

\newtheorem{qnx}[thm]{Question}
\newenvironment{qn}
  {\pushQED{\qed}\qnx}
  {\popQED\endqnx}

\theoremstyle{plain}
\theoremstyle{definition}
\newtheorem{defnx}[thm]{Definition}
\newenvironment{defn}
  {\pushQED{\qed}\defnx}
  {\popQED\enddefnx}

\title{Discrete and continuous symmetries in monotone Floer theory}
\author{Jack Smith}
\address{Department of Mathematics\\ University College London\\ Gower Street\\ London\\ WC1E  6BT}
\email{jack.smith@ucl.ac.uk}

\begin{document}

\begin{abstract}
This paper studies the self-Floer theory of a monotone Lagrangian submanifold $L$ of a symplectic manifold $X$ in the presence of various kinds of symmetry.  First we suppose $L$ is $K$-homogeneous and compute the image of low codimension $K$-invariant subvarieties of $X$ under the length-zero closed--open string map.  Next we consider the group $\symp(X, L)$ of symplectomorphisms of $X$ preserving $L$ setwise, and extend its action on the Oh spectral sequence to coefficients of arbitrary characteristic, incorporating its action on the classes of holomorphic discs.  This imposes constraints on the differentials which force them to vanish in certain situations.  These techniques are combined to study a family of homogeneous Lagrangians in products of projective spaces, which exhibit some unusual properties.
\end{abstract}

\maketitle

%------------------------------------

\section{Introduction}
\label{secIntro}

%--------------------------------------

\subsection{Summary}
\label{sscSummary}

Let $(X, \omega)$ be a symplectic manifold and $L \subset X$ a Lagrangian submanifold.  We assume that $L$ is \emph{monotone}, meaning that the Maslov index $\mu$ (relative first Chern class) and area homomorphisms $H_2(X, L; \Z) \rightarrow \R$ are positively proportional on the image $H_2^D$ of $\pi_2(X,L)$.  This ensures that moduli spaces of holomorphic discs have good compactness properties, so that quantum cohomology $QH^*(X)$ and Floer cohomology $HF^*(L, L)$ can be defined using classical methods.

The Floer cohomology $HF^*(L, L)$ is extremely difficult to compute in most cases, and there is no general procedure even to decide whether or not it vanishes.  Broadly speaking, existing non-vanishing proofs rely on the dimension being sufficiently low that the Floer complex can be computed explicitly (as in \cite{EL1}, for example), or exploit the Oh spectral sequence \cite{OhSS}
\begin{equation}
\label{eqOhSS}
\ssE{1} \cong H^*(L) \implies HF^*(L, L).
\end{equation}
The differentials encode `quantum corrections' to the classical cohomology, which count certain configurations of holomorphic discs, and typically one shows that either there are no discs (for example because $L$ is exact---this is essentially Floer's original result \cite{Flo}) or that the differentials must vanish for degree reasons (for example because $L$ is of high minimal Maslov number \cite[Theorem II]{OhSS}, \cite[Proposition 6.1.1]{BCQS}).

Special techniques apply when $L$ is a torus or is the fixed locus of an antisymplectic involution.  For tori the spectral sequence degenerates at the second page, since the classical cohomology ring is generated in degree $1$, and is determined entirely by the sum of the homology classes of the boundaries of index $2$ discs.  This can be calculated when $L$ is a toric fibre \cite{ChoCl, ChoOh} and in some other examples \cite{ChSchTwist}.  For antisymplectic fixed loci one can use the involution to reflect discs and prove that quantum corrections cancel in characteristic $2$ \cite{Ha} (see also \cite{FOOOinv}).

The purpose of the present paper is to develop two new techniques for constraining or computing $HF^*(L, L)$ using symmetries of $L$ inside $X$, which we illustrate by applying to the following family of examples.  Fix an integer $N \geq 3$ and consider the action of $\PSU(N-1)$ on $(\C\P^{N-2}, \omega_\mathrm{FS})$ obtained by projectivising the standard representation of $\SU(N-1)$ on $\C^{N-1}$; here $\omega_\mathrm{FS}$ denotes the Fubini--Study form.  Now let $X$ be the product $(\C\P^{N-2})^N$, with the diagonal $\PSU(N-1)$-action.  This action is Hamiltonian, and the zero set of the moment map is a monotone Lagrangian $L$ which is a free $\PSU(N-1)$-orbit.  The symmetric group $S_N$ acts on $X$ by permuting the factors and $L$ is invariant setwise.  These Lagrangians are high-dimensional but of low minimal Maslov number compared with the degrees of the generators of the cohomology ring, and there is no obvious disc cancellation, so none of the above techniques apply.

Fix a coefficient ring $R$.  If $\Char R \neq 2$ then we assume $L$ is orientable and equipped with a choice of \emph{relative spin structure}.  This is used to orient the moduli spaces of holomorphic curves being counted, and is discussed in detail in \cref{sscOrientations}.  One may also equip $L$ with a flat $R$-line bundle, or more generally a local system as in \cref{sscLocSys}.  Using our two techniques in conjunction, considering both the $\PSU(N-1)$- and $S_N$-symmetries, we compute the following.

\begin{mthm}[\cref{corPrimePow,propWide}]
\label{mthmHF}
For any choice of relative spin structure and local system on $L$, $HF^*(L, L; R)$ vanishes unless $N= p^r$ or $2p^r$ for some prime $p$ and positive integer $r$.

Conversely, if $R$ is a field of prime characteristic $p$ and $N=p^r$ or $2p^r$ then there exists a relative spin structure on $L$ such that with the trivial local system we have an isomorphism of $R$-vector spaces $HF^*(L, L; R) \cong H^*(L; R)$.
\end{mthm}

In particular, from those $N$ which are not prime powers or twice prime powers we obtain an infinite family of monotone Lagrangians whose Floer cohomology vanishes in a robust sense: it isn't just that we chose the wrong coefficient ring or local system.  To the author's knowledge, the only previously known examples with this property, when the ambient manifold $X$ is closed and simply connected, are those constructed by Oakley--Usher in \cite{OakUsh}.

A consequence of non-vanishing of $HF^*(L,L)$ is that $L$ cannot be displaced from itself by Hamiltonian diffeomorphisms---indeed, this was Floer's motivation for defining $HF^*$ in the first place.  Oakley and Usher show that many of their examples are displaceable, and for the remainder the displaceability seems to be currently unknown.  In contrast, we show

\begin{mthm}[\cref{thmNonDisp}]
\label{mthmNonDisp}
For all values of $N$ the Lagrangian $L$ is non-displaceable.
\end{mthm}

The proof uses Floer cohomology over $\C$ deformed by a \emph{$B$-field}.  This is a complex-valued closed $2$-form $B$ on $X$ which is used to weight the contribution of each holomorphic curve $u$ to the Floer differential by a factor of
\[
e^{-2\pi i \int u^* B}.
\]
Our methods give enough information about the curves to enable us to tune $B$ to make most contributions cancel.

%------------------------------------------------

\subsection{The closed--open map}
\label{sscCOintro}

Throughout the paper our standing assumptions are:
\begin{itemize}
\item $(X, \omega)$ is a compact symplectic manifold
\item $L \subset X$ is a closed, connected, monotone Lagrangian of dimension $n$ (so $\dim_\R X = 2n$)
\item The minimal Maslov number $N_L$ of $L$, defined to be the minimal positive generator of $\mu(H_2^D)$ in $\Z_{>0} \cup \{\infty\}$, is at least $2$
\item $R$ is a coefficient ring, and if $\Char R \neq 2$ then $L$ is orientable (which automatically implies $N_L$ is even) and equipped with a relative spin structure (see \cref{sscOrientations})
\item $L$ may also be equipped with a local system (see \cref{sscLocSys}).
\end{itemize}

The first technique we develop concerns the (length-zero) \emph{closed--open string map}, which is a unital $R$-algebra homomorphism
\[
\CO^0 : QH^*(X; R) \rightarrow HF^*(L, L; R).
\]
Given a cohomology class $\alpha$ on $X$, Poincar\'e dual to a cycle $Z$, $\CO^0(\alpha)$ is defined roughly as follows.  Consider the moduli space of holomorphic discs in $X$ with boundary on $L$ which carry one interior marked point and one boundary marked point.  We cut down the moduli space by asking the interior marked point to lie on $Z$, and then evaluate at the boundary marked point to sweep a (Floer) cycle on $L$.  The Poincar\'e dual of this cycle is $\CO^0(\alpha)$.  See \cref{sscCOreview} for a rigorous discussion.

The significance of this map for our purposes is that often $QH^*(X; R)$ is known, so computing parts of $\CO^0$ can give information about $HF^*(L, L; R)$.  For example, it is a famous observation, usually attributed to Auroux \cite[Lemma 3.1]{Au}, Kontsevich and Seidel, that $\CO^0$ sends $c_1(X)$ to $m_0(L) \cdot 1_L$, where $1_L$ is the unit in $HF^*$ and $m_0(L)$ is the count of holomorphic index $2$ discs which map a boundary marked point to a generically chosen point of $L$.  This means that
\[
\CO^0(c_1(X)-m_0(L) \cdot 1_X) = 0,
\]
where $1_X$ is the unit in $QH^*$, so if $HF^* \neq 0$ then $c_1(X)-m_0(X) \cdot 1_X$ cannot be invertible.  In other words, if $HF^*$ is non-zero then $m_0(L)$ is an eigenvalue for quantum multiplication by $c_1(X)$.

In general it is difficult to compute $\CO^0$ but we show that when $L$ has a large group of continuous symmetries then some calculation \emph{is} possible.  More precisely

\begin{defn} \label{labHomLag} A \emph{$K$-homogeneous Lagrangian} is a Lagrangian orbit $L$ of the action of a compact connected Lie group $K$ on a compact K\"ahler manifold $X$ by holomorphic automorphisms.
\end{defn}

The general study of such Lagrangians was initiated by Evans--Lekili \cite{EL1}.  For the remainder of this subsection we fix a $K$-homogeneous $(X, L)$, satisfying the above standing assumptions.  Evans and Lekili introduced a particularly simple type of holomorphic disc in this setting:

\begin{defn} \label{labAxDisc} A holomorphic disc $u : (D, \pd D) \rightarrow (X, L)$ is \emph{axial} if, possibly after reparametrisation, there exists a Lie algebra element $\xi \in \mathfrak{k}$ such that $u$ is of the form
\begin{equation}
\label{eqAxDef}
z \mapsto e^{-i \xi \log z} u(1)
\end{equation}
for $z \in D \setminus \{0\}$, where $e^\cdot$ denotes the exponential map $\mathfrak{k}_\C \rightarrow K_\C$.  Here $\mathfrak{k}$ is the Lie algebra of $K$, and $\mathfrak{k}_\C$, $K_\C$ are their respective complexifications (see \cref{sscHomog}).  We do not require $e^{2\pi \xi}$ to be the identity in $K$, merely that it fixes $u(1)$.
\end{defn}

\begin{rmk}
The definitions given in \cite{EL1} are slightly different.  Evans and Lekili require the $K$-action on $X$ to be Hamiltonian, and the complexified action to be algebraic.  Their definition of axiality is in terms of the existence of a Lie group morphism $R : \R \rightarrow K$ such that $u(e^{i\theta}z)=R(\theta)u(z)$ for all $z \in D$ and $\theta \in \R$, but by \cite[Section 2.3]{Sm} this implies our apparently stronger condition.
\end{rmk}

Let $N_X^+$ in $\Z_{>0} \cup \{\infty\}$ denote the minimal Chern number of rational curves in $X$ which are holomorphic for the standard (integrable, $K$-invariant, K\"ahler) complex structure.  Our main result is the following, assuming for simplicity that $L$ has the trivial local system.

\begin{mthm}[\cref{thmCO}]
\label{mthmCO}
Let $Z \subset X$ be a setwise $K$-invariant analytic subvariety of complex codimension $\leq N_X^+$, disjoint from $L$ and Poincar\'e dual a class $\alpha \in QH^*(X; R)$.  Then
\[
\CO^0(\alpha) = \lambda \cdot 1_L,
\]
where $\lambda$ is a count of axial discs passing through $Z$.  If $L$ is equipped with the \emph{standard} spin structure (see \cref{StdSpin}) then all of these discs contribute with positive sign.
\end{mthm}

The power of this statement is that in examples one can often enumerate the axial discs explicitly, by considering the possible values of $\xi$ in \eqref{eqAxDef} and reducing to a computation in linear algebra (see \cref{lemUniqDisc} for instance).  This is in stark contrast to the problem of enumerating $J$-holomorphic discs for an arbitrary almost complex structure $J$.

Using this knowledge of $\CO^0$ one can prove constraints on $HF^*$.  For instance one could argue as in the Auroux--Kontsevich--Seidel result that if $HF^* \neq 0$ then $\lambda$ must be an eigenvalue for quantum multiplication by $\alpha$.  More directly, if one knows some polynomial relation $f(\alpha_1, \dots, \alpha_r)=0$ holds in $QH^*$ between the classes $\alpha_1, \dots, \alpha_r$ Poincar\'e dual to invariant subvarieties $Z_1, \dots, Z_r$ (of complex codimension $\leq N_X^+$), then the same polynomial relation $f(\lambda_1, \dots, \lambda_r)=0$ must hold between the corresponding axial disc counts in the subring of $HF^*$ generated by $1_L$.  This is essentially how we obtain the first part of \cref{mthmHF}.  In \cite{Sm} we used a special case of \cref{mthmCO}---where $Z$ is the twisted cubic in $\C\P^3$---to explain the fact that the Floer cohomology of the Chiang Lagrangian \cite{Ch} vanishes outside characteristic $5$, as observed by Evans--Lekili \cite{EL1}, and derived similar constraints for other $\SU(2)$-homogeneous `Platonic' Lagrangians.

\begin{rmk}
Toric fibres provide the simplest examples of homogeneous Lagrangians, and in this setting \cref{mthmCO} recovers the well-known fact that (with the standard spin structure and trivial local system say) $\CO^0$ gives $1_L$ on each component of the toric divisor.  Applying this fact with a general local system is the key step in proving that quantum cohomology is isomorphic to the Jacobian of the superpotential (see \cite{SmithToricQH} and references therein).
\end{rmk}

The purpose of the codimension condition is to rule out bad sphere bubbling at the interior marked point constrained to $Z$, when comparing the computation for the standard complex structure with that for a suitably generic almost complex structure.  Without this condition the result can fail, as demonstrated in \cref{rmkCodimCondition}.

%---------------------------------------------

\subsection{The spectral sequence}
\label{sscIntroSS}

Our second technique concerns the Oh spectral sequence \eqref{eqOhSS}.  The group $\symp(X,L)$, of symplectomorphisms of $X$ preserving $L$ setwise, acts on the $\ssE{1}$ page by restricting $\phi \in \symp(X,L)$ to $L$ and acting on cohomology by pullback $(\phi^*)^{-1}$.  Our goal is to show that this induces an action on the spectral sequence, meaning that the action on $\ssE{1}$ commutes with the differential $\diff_1$, so descends to $\ssE{2}$, where it commutes with $\diff_2$, and so on.

Doing this in the most useful way requires an enlargement of the coefficient ring.  Usually in Floer theory one introduces a formal \emph{Novikov variable} $T$ and weights counts of holomorphic curves $u$ by $T^{\omega(u)}$.  By working over the Novikov ring
\[
\Big\{\sum_{j=1}^\infty a_j T^{l_j} : a_j \in R \text{ and } l_j \in \R \text{ with } l_j \rightarrow \infty \text{ as } j \rightarrow \infty\Big\}
\]
this allows one to make sense of potentially infinite counts.  Monotonicity enables us to equivalently weight by $T^{\mu(u)/N_L}$ instead, and ensures that counts are finite, so we may in fact work over the Laurent polynomial ring $R[T^{\pm1}]$ or even take $T=1$.  It is the filtration of the Floer complex by powers of $T$ that gives rise to the spectral sequence \eqref{eqOhSS} after setting $T=1$.

Instead of just weighting curves by $T^{\mu(u)/N_L}$ we can enlarge the coefficients from the group ring $R[T^{\pm 1}]$ to $R[H_2^D]$ and weight by the full homology class $T^{[u]}$.  This is well-known \cite[Section 2.1]{BCEG} in the setting of the \emph{pearl complex} model for Floer theory which we use throughout (see \cref{sscPearlBg}).  The filtration by Maslov index allows us to construct the spectral sequence
\begin{equation}
\label{eqOhSS2}
\ssE{1} \cong H^*(L;R) \otimes_R R[H_2^D] \implies HF^*(L, L;R[H_2^D])
\end{equation}
over this ring just as before.

\begin{mthm}[\cref{labSSrep}]
\label{mthmSS}
There is an action of $\symp(X,L)$ on the $\ssE{1}$ page of \eqref{eqOhSS2}, which involves its actions on both $H_2^D$ and $H^*(L; R)$, that induces an action on the whole spectral sequence.
\end{mthm}

The actions on later pages are all determined by the action on $\ssE{1}$, which is computable entirely in terms of classical topology and factors through $\pi_0 \mathrm{Diff}(X,L)$.  The content is that on each page the action commutes with the differential.  This was shown by Biran--Cornea \cite[Section 5.8]{BCQS} in characteristic $2$ and without the enlarged coefficients.  Outside characteristic $2$ some enlargement of the coefficient ring is in general necessary, to account for the fact that the $\symp(X,L)$-action may not preserve the relative spin structure used to orient moduli spaces, but our observation is that even when not necessary it is often useful.  Ultimately we care about the complex over $R[T^{\pm 1}]$ or $R$, and the enlarged coefficients should be viewed simply as an intermediate step.

In our main family of examples, we use the representation theory of $S_N \subset \symp(X,L)$ to constrain the differentials over $R[H_2^D]$ so that they vanish after reducing to $R[T^{\pm 1}]$.  This is analogous to the disc cancellation technique for antisymplecti fixed loci.  Now the cancellation process is more elaborate, but occurs at the level of homology classes rather than actual discs.

%-------------------------------------------

\subsection{Acknowledgements}

This work was mostly carried out whilst I was an EPSRC-funded PhD student at the University of Cambridge, and I am greatly indebted to my advisor, Ivan Smith, for many helpful discussions and suggestions, and for his patience and encouragement throughout the writing process.  Special thanks are also due to Oscar Randal-Williams, Nick Sheridan, and Frol Zapolsky, and an anonymous referee for valuable feedback.  Mohammed Abouzaid, Paul Biran, Fran\c{c}ois Charette, Jonny Evans, Kenji Fukaya, Sheel Ganatra, Ailsa Keating, Momchil Konstantinov, Yank\i~Lekili, Stuart Martin, Kaoru Ono, Dmitry Tonkonog, Brunella Torricelli and Chris Woodward all provided useful comments.  I am currently supported by EPSRC grant [EP/P02095X/1].

%-------------------------------------------

\section{Floer theory background}
\label{secFloerBg}

We start by reviewing the basic ideas of monotone Floer theory that we will need.

%-------------------------------------------

\subsection{The pearl complex}\label{sscPearlBg}

Recall our standing assumptions from \cref{sscCOintro}, regarding the symplectic manifold $X$, monotone Lagrangian $L$, and coefficient ring $R$.  We denote $R[T^{\pm 1}]$ by $\Lambda$, or $\Lambda_R$, with $T$ given grading $N_L$ (the minimal Maslov number).  Throughout the paper we will use the \emph{pearl complex} model for Floer (co)homology, developed by Biran--Cornea \cite{BCQS}, so we now give a brief summary.  To begin with we assume that $\Char R = 2$ and that $L$ carries the trivial local system.

Fix a Morse function $f$ and metric $g$ on $L$ such that the pair $(f, g)$ is Morse--Smale, and an almost complex structure $J$ on $X$ compatible with the symplectic form.  We abbreviate the triple $(f, g, J)$ to $\mathscr{D}$.  The pearl complex for $L$ over $\Lambda$, defined using auxiliary data $\mathscr{D}$, then has underlying $\Lambda$-module
\begin{equation}
\label{OldComplex}
C^*(L; \mathscr{D}; \Lambda) = \bigoplus_{x \in \crit(f)} \Lambda \cdot x,
\end{equation}
graded by Morse indices of critical points and the grading on $\Lambda$, and the differential counts rigid \emph{pearly trajectories} or \emph{strings of pearls}.  These are Morse flowlines between critical points, which may be interrupted by any number (including zero) of $J$-holomorphic discs in $X$ with boundary on $L$, as shown in \cref{figPearlTraj}.
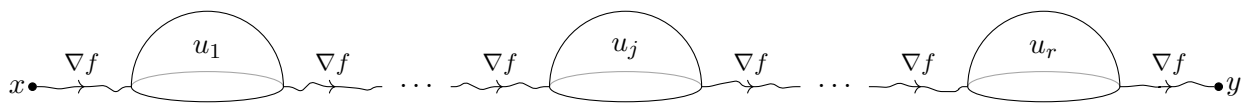
\begin{figure}[ht]
\centering
\begin{tikzpicture}
[blob/.style={circle, draw=black, fill=black, inner sep=0, minimum size=\blobsize}]
\definecolor{green}{rgb}{0, 0.9, 0.2}
\def\blobsize{1mm}

\critpt{-7.8}{0}{-0.2}{0}{x}{x}
\critpt{7.8}{0}{0.2}{0}{y}{y}

\flow{(x)}{(-6.5, 0)}{91}{$\nabla f$}{0}{0.3}
\disc{-6.5}{-4.5}{0}{$u_1$}
\flow{(-4.5, 0)}{(-3.2, 0)}{989}{$\nabla f$}{0}{0.3}
\draw (-2.75, 0) node{$\smash\dots$};
\flow{(-2.3, 0)}{(-1, 0)}{92}{$\nabla f$}{0}{0.3}
\disc{-1}{1}{0}{$u_j$}
\flow{(1, 0)}{(2.3, 0)}{90}{$\nabla f$}{0}{0.3}
\draw (2.75, 0) node{$\smash\dots$};
\flow{(3.2, 0)}{(4.5, 0)}{93}{$\nabla f$}{0}{0.3}
\disc{4.5}{6.5}{0}{$u_r$}
\flow{(6.5, 0)}{(y)}{94}{$\nabla f$}{0}{0.3}

\end{tikzpicture}
\caption{A pearly trajectory from $x$ to $y$.\label{figPearlTraj}}
\end{figure}
Each disc $u_j$ must be non-constant, and carries two marked points---one `incoming', which we assume is at $-1$ in $\pd D$, and one `outgoing' at $1$.

\begin{rmk}
Biran--Cornea work with downward Morse flows and homological notation, whereas we prefer to work with upward flows and \emph{co}homological notation.
\end{rmk}

More precisely, let $\mathcal{P}(x, y; A, \mathscr{D})$ denote the moduli space of pearly trajectories from $x$ to $y$, in which the total homology class of the discs is $A \in H_2(X, L; \Z)$, modulo reparametrisation of discs and flowlines.  This has virtual (expected) dimension
\[
|y| - |x| - 1 + \mu(A),
\]
where $|\cdot|$ denotes Morse index.  Biran--Cornea show \cite[3.1.2, Proposition 3.1.3]{BCQS} that after fixing $(f,g)$ there is a second category set of $J$ such that each $\mathcal{P}(x, y; A, \mathscr{D})$ of virtual dimension $\leq 1$ is transversely cut out, and those of virtual dimension $0$ (i.e.~rigid) are compact.  In the latter case we can therefore count the points in $\mathcal{P}(x, y; A, \mathscr{D})$ modulo $2$, and define the differential $\diff$ by
\[
\diff x = \smashoperator{\sum_{\substack{y, A \\ \mu(A)=|x|-|y|+1}}} \#_{\operatorname{mod} 2} \mathcal{P}(x, y; A, \mathscr{D}) \cdot T^{\mu(A)/N_L}y,
\]
extended $\Lambda$-linearly.  For each $y$ there are only finitely many $A$ which contribute to this sum, by monotonicity and Gromov compactness.

For $J$ in a smaller second category set they show \cite[Proposition 5.1.2]{BCQS} that $\diff^2=0$, and that the resulting (co)homology is independent of the data $\mathscr{D}$, so gives an invariant of $L \subset X$ which they name \emph{Lagrangian quantum (co)homology}.  We call such data $\diff$-regular.  The proof that $\diff^2=0$ involves compactifying the moduli spaces $\mathcal{P}(x, y; A, \mathscr{D})$ of virtual dimension $1$ by adding boundary components in which exactly one of the following degenerations has occurred: a flowline has broken (in the standard Morse sense); a flowline has shrunk to length zero; or a disc has bubbled into two, each carrying one of the marked points.  The first type of degeneration gives terms contributing to the coefficient of $y$ in $\diff^2x$, whilst the second and third cancel.  Since they form the boundary of a $1$-manifold, their count modulo $2$ vanishes. There is a small gap involving bubbling of index $2$ discs in $1$-dimensional families of trajectories with the same input and output, but this is addressed by Zapolsky in \cite[Section 6.2]{ZapPSS}.  The proof of invariance uses a generic Morse cobordism over $L \times [0,1]$ and a regular path $J_t$ of almost complex structures to connect two given choices of auxiliary data, and counts pearly trajectories which cross this cobordism to define a chain map between the two complexes.  The fact that this map is canonical (independent of the choice of Morse cobordism and path $J_t$) up to homotopy, and is a quasi-isomorphism, is proved using a Morse cobordism over $L \times [0, 1]^2$ and a corresponding `paths of paths' $J_{t, \tau}$.  See \cite[Section 5.1.2]{BCQS} for details.

One can also define Y-shaped pearly trajectories, with two input legs and one output, meeting at a $J$-holomorphic disc with three boundary marked points in prescribed cyclic order.  The latter disc is allowed to be constant, and if this is the case and there are no other discs then one recovers the familiar definition of the Morse trees which compute the classical cup product.  In order to make these Y-shaped pearly trajectories transverse one has to choose different Morse functions, $f_1$, $f_2$ and $f_3$, on each leg, such that each $(f_i, g)$ is Morse--Smale and such that the following additional Morse--Smale-type condition is satsified: each intersection of an ascending manifold of $f_1$ with an ascending manifold of $f_2$ and a descending manifold of $f_3$ is transverse.  Once this has been done, there is second category set of $J$ such that for each $i$ the data $(f_i, g, J)$ are $\diff$-regular, and such that the Y-shaped pearly moduli spaces are transversely cut out and give rise to a well-defined associative product $\qcup$ on Lagrangian quantum cohomology \cite[Section 5.2]{BCQS}.  We denote $((f_1, f_2, f_3), g, J)$ by $\mathcal{D}_{\qcup}$, and say they are $\qcup$-regular if they satisfy these transversality conditions.  If $\mathscr{D}_i$ denotes $(f_i, g, J)$, and $C_i$ the corresponding pearl complex, then the count of Y-shaped trajectories really defines a product $C_1 \otimes C_2 \rightarrow C_3$, and in order to obtain a product on a single complex $C$ (constructed using $\diff$-regular $\mathscr{D}$, say) we have to use comparison map quasi-isomorphisms
\[
C \otimes C \xrightarrow{\text{comparison qis}} C_1 \otimes C_2 \xrightarrow{\text{product}} C_3 \xrightarrow{\text{comparison qis}} C.
\]

Biran--Cornea \cite[Section 5.6]{BCQS} show that the Lagrangian quantum cohomology is canonically isomorphic to the self-Floer cohomology $HF^*(L, L; \Lambda)$ via a so-called PSS map.  Zapolsky \cite[Section 5.2.4]{ZapPSS} proves that under this isomorphism the above product agrees with the standard Floer product, counting pseudoholomorphic triangles.  We will therefore only speak of Floer cohomology, not Lagrangian quantum cohomology, but will compute it using the pearl complex.

\begin{rmk}
The differential and product on the pearl complex, as well as the comparison maps between different choices of auxiliary data, all agree with their Morse counterparts to order zero in $T$.  The non-Morse contributions (quantum corrections), carry positive powers of $T$ since non-constant $J$-holomorphic discs have positive area, or equivalently positive index.  For this reason it is often helpful to keep the $T$ variable, however we may also set it equal to $1$ to define the invariant $HF^*(L^\flat, L^\flat; R)$ appearing (without the ${}^\flat$s) in \cref{secIntro}.
\end{rmk}

%--------------------------------------------

\subsection{Orientations}
\label{sscOrientations}

The assumption $\Char R = 2$ can be relaxed by orienting the moduli spaces used above.  Biran--Cornea \cite[Appendix A]{BCEG} show how to do this assuming that $L$ is orientable, using a choice of relative spin structure on $L$ to coherently orient moduli spaces of $J$-holomorphic discs.  Strictly they only consider absolute spin structures, but these only enter in the construction of orientations on disc moduli spaces and for this relative spin structures suffice by \cite[Chapter 8]{FOOObig}.  Since choices of relative spin structure play a significant role in our results, in this subsection we give a brief review of the theory.

\begin{rmk}
Zapolsky \cite{ZapPSS} gives an alternative approach to orientations under somewhat weaker hypotheses.  It would be desirable to understand its compatibility with \cite{BCEG}, and more generally between different orientation schemes in the literature.
\end{rmk}

It is known from work of Vin de Silva \cite{VdS} that moduli spaces of $J$-holomorphic discs are not in general orientable.  More specifically, if $L$ is orientable (this is necessary even for the moduli space of constant discs to be orientable) and $\mathcal{M}$ denotes the moduli space of $J$-holomorphic discs in a fixed homology class, then for any loop $u_t : (D, \pd D) \rightarrow (X, L)$ in $\mathcal{M}$ we have
\[
\ip{w_1(\mathcal{M})}{u_t} = \ip{w_2(L)}{u_t|_{\pd D}}.
\]
Here the right-hand side denotes the evaluation of the second Stiefel--Whitney class of (the tangent bundle of) $L$ on the $2$-torus in $L$ swept by the boundaries of the discs in our loop, whilst the left-hand side denotes the evaluation of the first Stiefel--Whitney class of $\mathcal{M}$ on the loop itself.  If $\mathcal{M}$ is transversely cut out then it has a genuine tangent bundle, given by the kernel of the linearisation $D\overline{\pd}_J$ of the $\overline{\pd}_J$-operator, but in general we should interpret $w_1(\mathcal{M})$ as the first Steifel--Whitney class of the determinant line bundle $\det (D\overline{\pd}_J$).

The obvious condition on $L$ to guarantee orientability of disc moduli spaces is thus orientability (i.e.~vanishing of $w_1(L)$) plus vanishing of $w_2(L)$.  This is equivalent to the existence of a spin structure on $L$, meaning a lift of the positively-oriented frame bundle from a principal $\mathrm{GL}_+$-bundle to a principal $\mathrm{GL}_+^\sim$-bundle, where $\mathrm{GL}$ denotes $\mathrm{GL}(n, \R)$, $\mathrm{GL}_+$ its orientation-preserving component, and $\mathrm{GL}_+^\sim$ the unique connected double cover of this component.  This notion is independent of the choice of orientation on $L$, since the $\mathrm{GL}_+$ frame bundles associated to the two different orientations are isomorphic by acting by an element of $\mathrm{GL} \setminus \mathrm{GL}_+$, and this isomorphism is canonical up to homotopy since the latter space is connected.  The set of spin structures forms a torsor for $H^1(L; \Z/2)$.

In fact we only need $w_2(L)$ to vanish on the boundaries of certain solid tori in $X$, namely those swept out by loops of holomorphic discs.  In particular, if $w_2(L)$ lies in the image of the restriction map $H^2(X; \Z/2) \rightarrow H^2(L; \Z/2)$ then the disc moduli spaces are still orientable.  This motivates

\begin{defn}[{\cite{FOOObig}, as reformulated in \cite[Section 3]{QuiltOr}}]
\label{defRelSpin}
If $L$ is orientable and there exists a class $b \in H^2(X; \Z/2)$ such that $b|_L = w_2(L)$ then $L$ is \emph{relatively spin}.  In this case, a \emph{relative spin structure} on $L$ comprises a choice of $b$ (the \emph{background class}), an open cover $(U_i)$ of $L$, and a lift $\hat{\psi}_{ij}$ of each $\mathrm{GL}_+$ transition map $\psi_{ij} : U_i \cap U_j \rightarrow \mathrm{GL}_+$ to $\mathrm{GL}_+^\sim$ such that on each $U_i \cap U_j \cap U_k$ we have
\[
\hat{\psi}_{jk} \hat{\psi}_{ik}^{-1} \hat{\psi}_{ij} = \delta_{ijk},
\]
where $\delta_{ijk}$ is a \v{C}ech cocycle representing $b|_L$.  We quotient out by an appropriate \v{C}ech-like notion of equivalence.  The set of relative spin structures forms a torsor for $H^2(X, L; \Z/2)$, with a class $\eps \in H^2(X, L; \Z/2)$ acting on background classes by $b \mapsto b + \eps \in H^2(X; \Z/2)$.
\end{defn}

Fukaya--Oh--Ohta--Ono \cite[Chapter 8]{FOOObig} show how an orientation and relative spin structure on $L$ induce coherent orientations on the disc moduli spaces.  The important properties for us are

\begin{lem}[{de Silva \cite[Theorem Q]{VdS}, Cho \cite[Theorem 6.4]{ChoCl}, Fukaya--Oh--Ohta--Ono \cite[Proposition 8.1.16]{FOOObig}}]
\label{SpinChangeProp}
Changing relative spin structure by a class $\eps \in H^2(X, L; \Z/2)$ modifies the orientation on the moduli space of holomorphic discs in class $A \in H_2(X, L; \Z)$ by $(-1)^{\ip{\eps}{A}}$.\hfill$\qed$
\end{lem}

\begin{lem}
\label{OrientationReverse}
Reversing the orientation of $L$, whilst keeping the same relative spin structure, does not affect orientations of pearly trajectory moduli spaces.
\end{lem}
\begin{proof}
First we claim that it reverses the orientations of the disc moduli spaces.  This is because in \cite[Chapter 8]{FOOObig} the orientation and relative spin structure on $L$ enter the construction by inducing a stable trivialisation of $TL$ over the boundary of each disc.  Reversing the orientation of $L$ simply reverses the orientation of this trivialisation, which in turn reverses the moduli space orientation.  This is explained in more detail in \cref{sscOrientingDiscs}.

To show that the pearly orientations are unaffected we now need to examine the conventions from \cite[Appendix A]{BCEG} to see that all of the sign changes cancel out.  Reversing the orientation of $L$ introduces the following minus signs into the orientation computations: one for each disc moduli space (since their orientations are reversed); one for each ascending manifold (\emph{descending} manifolds are the natural ones to orient, and ascending manifolds acquire induced orientations which depend on the orientation of $L$); and one for each fibre product taken over $L$ (this is the number of flowlines in the trajectory).  Representing a pearly trajectory by a tree $\mathcal{T}$, with a flowline for each edge and a disc at each internal vertex, the important quantity is thus the parity of
\[
\# \text{internal vertices} + \# \text{inputs} + \# \text{edges}.
\]
This is simply $\#\text{vertices}+\#\text{edges}-1$, and recognising the first two terms modulo $2$ as the Euler characteristic of $\mathcal{T}$ (which is contractible) we see that the relevant parity is even.
\end{proof}

%-------------------------------

\subsection{Local systems}
\label{sscLocSys}

We now also relax the condition that $L$ carries the trivial local system.  Recall that a local system on $L$ (over the coefficient ring $R$) is a locally constant sheaf $\mathscr{F}$ of $R$-modules.  Since $L$ is connected, a local system can be described relative to an arbitrary base point $p$ by specifying an $R$-module $F$ (the fibre) and a monodromy homomorphism $\overline{m} : \pi_1(L, p)^\mathrm{op} \rightarrow \operatorname{Aut}_R F$.  The ${}^\mathrm{op}$ appears because we write concatenation of paths from left to right, but composition of functions from right to left.  In a slight abuse of notation, we'll say that \emph{$\mathscr{F}$ has rank $1$} if $F=R$.  Otherwise we'll say $F$ has \emph{higher rank}.

\begin{defn}
A \emph{monotone Lagrangian brane} $L^\flat$ is a triple $(L, s, \mathscr{F})$ comprising a monotone Lagragian (satisfying the usual hypotheses), a relative spin structure $s$ (not needed if $\Char R = 2$), and a local system $\mathscr{F}$.
\end{defn}

We now explain how $\mathscr{F}$ modifies the pearl complex.  First, the complex itself becomes
\begin{equation}
\label{eqPearlCx}
C^*(L^\flat; \mathscr{D}; \Lambda) = \bigoplus_{x \in \crit(f)} (\Lambda \cdot x) \otimes_R \End_R \mathscr{F}_x,
\end{equation}
where $\mathscr{F}_x$ denotes the fibre of $\mathscr{F}$ at $x$.  We then change the differential as follows.  For each rigid pearly tranjectory $\gamma$ from $x$ to $y$ we can define two homomorphisms, $\mathrm{P}(\gamma_\mathrm{t})$ and $\mathrm{P}(\gamma_\mathrm{b})$, from $\mathscr{F}_x$ to $\mathscr{F}_y$ by parallel transporting along $\gamma$, traversing each disc around the top and bottom half of the boundary respectively, as shown in \cref{figParTrans}.
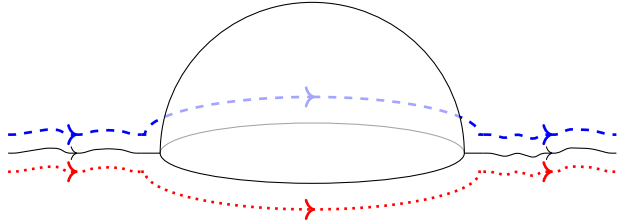
\begin{figure}[ht]
\centering
\begin{tikzpicture}
[blob/.style={circle, draw=black, fill=black, inner sep=0, minimum size=\blobsize}]
\definecolor{green}{rgb}{0, 0.9, 0.2}
\def\blobsize{1mm}

\begin{scope}[blue, dashed, yshift=7, line width=1]
\flow{(-4, 0)}{(-2.25, 0)}{41}{}{}{}
\draw[->, shorten >=-1] (-2.25, 0) -- (-2.2, 0)  arc (180:270: 2.2 and -0.5);
\draw (0, 0.5) arc (270:360: 2.2 and -0.5) -- (2.25, 0);
\flow{(2.25, 0)}{(4, 0)}{42}{}{}{}
\end{scope}

\flow{(-4, 0)}{(-2.25, 0)}{41}{}{}{}
\draw (-2.25, 0) -- (-2, 0);
\disc {-2}{2}{0}{}
\draw (2, 0) -- (2.25, 0);
\flow{(2.25, 0)}{(4, 0)}{42}{}{}{}

\begin{scope}[red, dotted, yshift=-7, line width=1]
\flow{(-4, 0)}{(-2.25, 0)}{41}{}{}{}
\draw[->, shorten >=-1] (-2.25, 0) -- (-2.2, 0)  arc (180:270: 2.2 and 0.5);
\draw (0, -0.5) arc (270:360: 2.2 and 0.5) -- (2.25, 0);
\flow{(2.25, 0)}{(4, 0)}{42}{}{}{}
\end{scope}

\end{tikzpicture}
\caption{The parallel transport paths $\gamma_\mathrm{t}$ (dashed) and $\gamma_\mathrm{b}$ (dotted) defining the maps $\mathrm{P}(\gamma_\mathrm{t})$ and $\mathrm{P}(\gamma_\mathrm{b})$.\label{figParTrans}}
\end{figure}
The contribution of $\gamma$ to the differential of $x \otimes \theta$, for $\theta \in \End \mathscr{F}_x$, is then defined to be $T^{\mu(A)/N_L}y \otimes \mathrm{P}(\gamma_\mathrm{t}) \circ \theta \circ \mathrm{P}(\gamma_\mathrm{b})^{-1}$, where $A$ is the total disc class as before.

The construction all goes through again except that if $\mathscr{F}$ has higher rank then the Floer complex may be obstructed, meaning that $\diff^2$ is not equal to zero.  Konstantinov considers this in \cite{MomchilLocSys}, where he defines certain subcomplexes which are always unobstructed.  Whenever we mention higher rank local systems we assume that either $\diff^2$ really is zero, or that we have passed to one of his subcomplexes.  Since these subcomplexes are closed under the Floer product and contain the image of $\CO^0$, our arguments still apply to them.

If $\mathscr{F}$ has rank $1$ then each $\End \mathscr{F}_x$ is canonically identified with $R$, so the complex reduces to \eqref{OldComplex}.  In this case the group of fibre automorphisms, $R^\times$, is abelian so the monodromy representation $\overline{m} : \pi_1(L, p)^\mathrm{op} \rightarrow \operatorname{Aut}_R \mathscr{F}_p$ factors through $H_1(L; \Z)$.  Then in the map $\theta \mapsto \mathrm{P}(\gamma_\mathrm{t}) \circ \theta \circ \mathrm{P}(\gamma_\mathrm{b})^{-1}$ the contributions of the parallel transports along the Morse flowlines all cancel away, and we are left with $\theta \mapsto \overline{m}(-\pd A) \theta$, where $\pd$ denotes the boundary map $H_2(X, L) \rightarrow H_1(L; \Z)$.  Letting $m$ denote the composition of $\overline{m}$ with the inverse map on $\pi_1(L, p)$, this simplifies to $\theta \mapsto m(\pd A) \theta$.

\begin{rmk}\label{SpinLocSys}
The choices of relative spin structure and local system are coupled together in the following sense.  The group $H^1(L; \Z/2)$ acts on local systems by twisting the signs of the monodromy maps, and acting by a class $\eps \in H^1(L; \Z/2)$ is equivalent (by \cref{SpinChangeProp}) to acting on the relative spin structure by the image of $\eps$ under the connecting map $H^1(L; \Z/2) \rightarrow H^2(X, L; \Z/2)$
\end{rmk}

%--------------------------------------

\subsection{The closed--open map}
\label{sscCOreview}

We now describe the pearl model for the closed--open map
\begin{equation}
\label{eqCOdefn}
\CO^0 : QH^*(X; \Lambda) \rightarrow HF^*(L^\flat, L^\flat; \Lambda)
\end{equation}
mentioned (without ${}^\flat$s and with $T$ set to $1$) in \cref{sscCOintro}.  This is a quantum version of the classical pullback map $H^*(X; \Lambda) \rightarrow H^*(L; \Lambda)$ under the inclusion of $L$ into $X$.  Biran--Cornea \cite[Section 5.3]{BCQS} actually define a more general operation, the \emph{quantum module action}
\begin{equation}
\label{eqQuantumModule}
\qcup : QH^* \otimes HF^* \rightarrow HF^*,
\end{equation}
which combines $\CO^0$ with the product on $HF^*$, and we begin by discussing this.

First fix a Morse--Smale pair $(f, g)$ on $L$ and another $(h, g_X)$ on $X$ with the property that each ascending manifold of $h$ is transverse to every ascending and descending manifold of $f$.  Given $a$ in $\crit(h)$, and $x$ and $y$ in $\crit(f)$, the coefficient of $y$ in $a \qcup x$ counts configurations illustrated in \cref{figQMTraj}.
\begin{figure}[ht]
\centering
\begin{tikzpicture}
[blob/.style={circle, draw=black, fill=black, inner sep=0, minimum size=\blobsize}]
\definecolor{green}{rgb}{0, 0.9, 0.2}
\def\blobsize{1mm}

\critpt{0}{2.5}{-0.2}{0}{a}{a}
\flow{(a)}{(0, 1)}{92}{$\nabla h$}{-0.4}{0}
\draw (0, 1) node[blob]{};

\critpt{-7.8}{0}{-0.2}{0}{x}{x}
\critpt{7.8}{0}{0.2}{0}{y}{y}

\flow{(x)}{(-6.5, 0)}{91}{$\nabla f$}{0}{0.3}
\disc{-6.5}{-4.5}{0}{$u_1$}
\flow{(-4.5, 0)}{(-3.2, 0)}{989}{$\nabla f$}{0}{0.3}
\draw (-2.75, 0) node{$\smash\dots$};
\flow{(-2.3, 0)}{(-1, 0)}{92}{$\nabla f$}{0}{0.3}
\disc{-1}{1}{0}{$u_k$}
\flow{(1, 0)}{(2.3, 0)}{90}{$\nabla f$}{0}{0.3}
\draw (2.75, 0) node{$\smash\dots$};
\flow{(3.2, 0)}{(4.5, 0)}{93}{$\nabla f$}{0}{0.3}
\disc{4.5}{6.5}{0}{$u_r$}
\flow{(6.5, 0)}{(y)}{94}{$\nabla f$}{0}{0.3}

\end{tikzpicture}
\caption{A pearly trajectory contributing to the coefficient of $y$ in $a\qcup x$.\label{figQMTraj}}
\end{figure}
Precisely, one considers ordinary pearly trajectories from $x$ to $y$, but where one of the discs, $u_k$, is distinguished.  This disc may be constant, and carries an interior marked point at $0$ which is required to lie on the ascending manifold $W^\mathrm{asc}(a)$ of $a$.  It also carries a Hamiltonian perturbation $H$, which deforms the holomorphic curve equation and gives extra flexibility for attaining transversality; we describe such perturbations in \cref{sscHamPerts}, and in particular explain how Gromov compactness adapts.  The virtual dimension of the moduli space of such configurations is
\[
|y| - |x| - |a| + \mu(A),
\]
where as usual $A$ is the total homology class of the discs, and we weight the count by $T^{\mu(A)/N_L}$.

\begin{rmk}
Although $u_k$ may in principle be constant, constant discs are unlikely to satisfy the $H$-perturbed $J$-holomorphic curve equation.
\end{rmk}

Biran--Cornea show \cite[Sections 5.3.9--5.3.11]{BCQS} that given $(f, g)$ and $(h, g_X)$ as above, there is a second category set of $J$, and for each such $J$ a second category set of $H$, which achieve transversality for the moduli spaces required to define $\qcup$ and show that it induces a product \eqref{eqQuantumModule} which is associative (with respect to the products on both $QH^*$ and $HF^*$), unital, and independent of choices.  The relevant signs are defined in \cite[Section A.2.3]{BCEG}, and twisting by local systems is done exactly as for ordinary pearly trajectories (the extra marked point and incoming flowline from $a$ have no effect).

The map $\CO^0$ is defined by taking the quantum module action of $QH^*$ on $1_L \in HF^*$.  This gives a unital $\Lambda$-algebra homomorphism \eqref{eqCOdefn}, and it lands in the centre of $HF^*$ by \cite[Theorem 2.1.1(ii)]{BCQS} and the subsequent paragraph.  Concretely, $\CO^0(a)$ counts trajectories as in \cref{figQMTraj} but where $x$ is the minimum $m$ of $f$ (or sum over the  local minima $m_i$ if there's more than one), and $|y|=|a|-\mu(A)$.  In this situation, assuming our data are $\CO^0$-regular, we always have $k=1$ for rigid trajectories, i.e.~the disc carrying the interior marked point is the first one in the trajectory.  This because otherwise we may delete the discs $u_1, \dots, u_{k-1}$, and replace $m$ by the unique critical point $x'$ whose ascending manifold contains $u_k(-1)$, in order to obtain a transversely cut out trajectory of virtual dimension
\[
-\big(\mu(u_1)+\dots+\mu(u_{k-1})\big)-|x'| < 0.
\]
Moreover, we may drop the condition that $u(-1)$ lies in the ascending manifold of $m$ (or the union of the ascending manifolds of the $m_i$) since if we have a quantum module action trajectory satisfying $|y|=|a|-\mu(A)$ but $|x| > 0$ then again it is transversely cut out but of negative virtual dimension.  Thus $\CO^0(a)$ counts trajectories as shown in \cref{figCOTraj}.
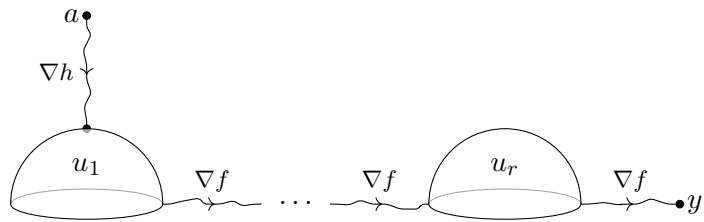
\begin{figure}[ht]
\centering
\begin{tikzpicture}
[blob/.style={circle, draw=black, fill=black, inner sep=0, minimum size=\blobsize}]
\definecolor{green}{rgb}{0, 0.9, 0.2}
\def\blobsize{1mm}

\critpt{0}{2.5}{-0.2}{0}{a}{a}
\flow{(a)}{(0, 1)}{92}{$\nabla h$}{-0.4}{0}
\draw (0, 1) node[blob]{};

\critpt{7.8}{0}{0.2}{0}{y}{y}

\disc{-1}{1}{0}{$u_1$}
\flow{(1, 0)}{(2.3, 0)}{90}{$\nabla f$}{0}{0.3}
\draw (2.75, 0) node{$\smash\dots$};
\flow{(3.2, 0)}{(4.5, 0)}{93}{$\nabla f$}{0}{0.3}
\disc{4.5}{6.5}{0}{$u_r$}
\flow{(6.5, 0)}{(y)}{94}{$\nabla f$}{0}{0.3}

\end{tikzpicture}
\caption{A pearly trajectory contributing to the coefficient of $y$ in $\CO^0(a)$.\label{figCOTraj}}
\end{figure}

We will need to compute $\CO^0$ on classes defined as the Poincar\'e dual $\PD[F]$ of a pseudocycle $F : M \rightarrow X$.  Recall that $F$ being a pseudocycle means that $M$ is a smooth manifold and $F$ is a smooth map whose image is pre-compact, such that $\overline{F(M)}\setminus F(M)$ is covered by a countable collection of smooth maps from manifolds of dimension $\leq \dim M - 2$.  If $M$ is oriented over the coefficient ring then this gives rise to a well-defined homology class in $X$ which we denote by $[F]$.  We restrict to the case where $M$ has even dimension, $2l$, which is all we need in applications.  Otherwise the sign computations depend on specific choices of conventions for duality \cite[Section A.2.6]{BCEG}.

We shall assume that $\overline{F(M)}$ is disjoint from $L$, so in particular $F$ and its boundary strata (the maps covering $\overline{F(M)} \setminus F(M)$) are vacuously transverse to ascending and descending manifolds in $L$.  The same arguments as in \cite[Sections 5.3.9--5.3.11]{BCQS} show that a second category set of $J$, and a second category set of $H$, achieve transversality for the moduli spaces of trajectories used for $\CO^0$ but with $u_1(0)$ now constrained to lie in $F(M)$.  More precisely, the latter means that instead of taking the fibre product over `evaluation of $u_1$ at $0$' and `inclusion of $W^\mathrm{asc}(a)$' in the definition of the moduli space, we take the fibre product over `evaluation of $u_1$ at $0$' and `$F$'.  Counting rigid trajectories of this type defines a pearl cocycle which we denote by $\CO^0_\mathrm{pseudo}(\PD[F])$.

\begin{lem}
\label{lemCOpseudo}
We have $\CO^0_\mathrm{pseudo}(\PD[F]) = \CO^0(\PD[F])$.
\end{lem}
\begin{proof}
Consider the one-dimensional moduli spaces of trajectories of the same shape but where $u_1(0)$ is required to lie at the end of a $\nabla h$ flowline which begins on $F(M)$ (again, this is properly described in terms of a fibre product).  For fixed $(f, g)$ and generic choices of $(h, g_X)$, $J$ and $H$ these are transversely cut out and are compactified by: shrinking of the $\nabla f$ flowlines and bubbling of the discs at their boundary marked points, which cancel each other; breaking of the $\nabla f$ flowlines, which contributes pearl coboundaries; shrinking of the $\nabla h$ flowline, which contributes $-\CO^0_\mathrm{pseudo}(\PD[F])$; and breaking of the $\nabla h$ flowline, which contributes $\CO^0$ of the Morse cocycle $\alpha$ defined by
\begin{equation}
\label{eqMorsePD}
\alpha = \smashoperator{\sum_{\substack{a \in \crit(h)\\ |a| = 2(n-l)}}} \#\big(M \times_X W^\mathrm{desc}(a)\big) \cdot a.
\end{equation}
Here the fibre product is taken over $F$ and the inclusion of the descending manifold $W^\mathrm{desc}(a)$, and the fact that $\CO^0_\mathrm{pseudo}$ and $\CO^0$ occur with opposite signs comes from \cite[(94) and (95)]{BCEG} (in the cited formulae set $X=M$ and $z=a$, and take $Y$ to be the fibre product defining the part of the trajectory not involving the input at the interior marked point).  Since these terms together describe the oriented boundary of a compact $1$-manifold, we get $\CO^0_\mathrm{pseudo}(\PD[F]) = \CO^0(\alpha)$ in $HF^*$.

We claim that $\alpha$ is the Morse-theoretic Poincar\'e dual of $[F]$, which then gives the desired equality.  To prove this, take a Morse--Smale pair $(h', g_X')$ on $X$ such that all ascending manifolds are transverse to $F$ and its boundary strata.  We orient the descending manifolds, which co-orients the ascending manifolds, and then the coefficient of each index $2l$ critical point $a'$ of $h'$ in $[F]$ is $\# (M \times_X W^\mathrm{asc}(a'))$.  Here the signs come directly from the orientation of $M$ and co-orientation of $W^\mathrm{asc}(a')$, i.e.~from the orientation sign of
\begin{equation}
\label{eqMorseSign}
DF : TM \xrightarrow{\sim} TX/TW^\mathrm{asc}(a'),
\end{equation}
independent of the conventions of \cite{BCEG}.

Now recall that Morse-theoretic Poincar\'e duality is realised by reversing the sign of the Morse function, so take $h'$ to be $-h$ and $g_X'$ to be $g_X$, where $(h, g_X)$ is the Morse--Smale pair used above (in particular, the one in the definition of $\alpha$).  There is a natural bijection $a' \leftrightarrow a$ between critical points of $h'$ and $h$; as a bijection between subsets of $X$ this is just the identity map, but we will keep separate notation $a'$ and $a$ so we know which Morse function we have in mind.  We see that $\PD[F]$ coincides with $\alpha$ as in \eqref{eqMorsePD}, as long as we can show that the orientation attached to $M \times_X W^\mathrm{desc}(a)$ by \cite[Section A.1.8]{BCEG}, after orienting $W^\mathrm{desc}(a)$ so that
\begin{equation}
\label{eqMorseSplitting}
T_aX \cong T_aW^\mathrm{desc}(a') \oplus T_aW^\mathrm{desc}(a)
\end{equation}
is orientation-preserving, agrees with the corresponding signs coming from \eqref{eqMorseSign}.  Note that the orientation on $W^\mathrm{desc}(a)$ induced by \eqref{eqMorseSplitting} depends on the orientation of $X$ (which is needed to define $\PD$; we always use the symplectic orientation) and the orientation we chose on $W^\mathrm{desc}(a')$.  In principle it is also sensitive to the ordering of the summands on the right-hand side, but this doesn't matter because $\dim M$ is even.  If this were not the case then would also have to worry about the duality conventions discussed in \cite[Section A.2.6]{BCEG}.

We now check that the two signs do indeed agree, so suppose that $(m, x)$ is a point in
\[
M \times_L W^\mathrm{desc}(a) = M \times_L W^\mathrm{asc}(a').
\]
 Let $e_1, \dots, e_{2l}$ be a positively-oriented basis for $T_mM$, and let $f_1, \dots, f_{2(n-l)}$ be a basis for $T_xW^\mathrm{desc}(a)$ which is positively-oriented according to \eqref{eqMorseSplitting}.  The orientation sign of \eqref{eqMorseSign} is positive if and only if $D_xF(e_1), \dots, D_xF(e_{2l}), f_1, \dots, f_{2(n-l)}$ is a positively-oriented basis for $T_xX$.  We must compare this with the sign from $M \times_X W^\mathrm{desc}(a)$ given by \cite[Section A.1.8]{BCEG}.  This is the orientation sign of the isomorphism
\[
T_mM \oplus T_xX \oplus T_xW^\mathrm{desc}(a) \xrightarrow{\theta} T_xX \oplus T_xX
\]
defined by $\theta(v_M, v_X, v_W) = (D_mF(v_M)-v_X, v_X-v_W)$.  This map can be homotoped through isomorphisms to $(D_mF(v_M)-v_W, v_X)$, and after swapping the $T_xX$ and $T_xW^\mathrm{desc}(a)$ summands on the left-hand side (this makes no difference since $\dim X \cdot \dim W^\mathrm{desc}(a)$ is even) we can cancel the rightmost summands from both sides, and we are left with the orientation sign of the basis $e_1, \dots, e_{2l}, -f_1, \dots, -f_{2(n-l)}$ for $T_xX$.  Reversing the sign of the $f_j$ makes no difference (there's an even number of them), so we conclude that both of the relevant signs are positive if and only if $e_1, \dots, e_{2l}, f_1, \dots, f_{2(n-l)}$ is positively-oriented, completing the proof.
\end{proof}

\begin{defn}
\label{defFregular}
The auxiliary data $\mathscr{D} = (f, g, J, H)$ are \emph{$F$-regular} if they can be used to compute $\CO^0(\PD[F])$ via $\CO^0_\mathrm{pseudo}(\PD[F])$.
\end{defn}

%---------------------------

\subsection{Hamiltonian perturbations}
\label{sscHamPerts}

We end our Floer theory review with a brief summary of Hamiltonian perturbations following \cite[Section 5.3.7]{BCQS} and \cite[(8f)]{SeidelBook} (see also \cite[Chapter 8]{BigMcS}).  The perturbation $H$ is a $1$-form on the disc $D$, valued in smooth functions on $X$.  There is an induced $1$-form $Y$ on $D$, valued in Hamiltonian vector fields on $X$, and the perturbed holomorphic curve equation is $(D u - Y)^{0,1} = 0$.  We call solutions (with boundary on $L$) \emph{$(J,H)$-holomorphic} discs.  Incorporating $H$ gives extra freedom for attaining transversality of moduli spaces, and we now explain how Gromov compactness is modified.

It was observed by Gromov \cite[1.4.C']{Gro} that $(J,H)$-holomorphic discs can be viewed as $J_H$-holomorphic sections of the trivial fibration $D \times X \rightarrow D$, where $J_H$ is the almost complex structure
\[
\begin{pmatrix} i & 0 \\ JY-Yi & J\end{pmatrix}
\]
on $D \times X$.  Here we are using the obvious splitting $T(D\times X) = TD \oplus TX$, and $i$ is the standard complex structure on $D$.  The boundaries of these sections lie on the Lagrangian $\pd D \times L \subset D \times X$.

Write $\omega_D$ for the standard area form on $D$ and $\pi_D$ and $\pi_X$ for the projections from $D \times X$ to $D$ and $X$ respectively.  We claim that $J_H$ is tamed by the symplectic form $\omega_\kappa \coloneqq \kappa \pi_D^* \omega_D +\pi_X^* \omega$ as long as $\kappa > |JY-Yi|^2/4$---this will then allow us to apply Gromov compactness in $D \times X$.  We need to show that $\omega_\kappa(v, J_Hv) > 0$ for all non-zero vectors $v$, and to do this simply write $v$ as $v_D+v_X$, under the splitting of $T(D \times X)$, and compute:
\[
\omega_\kappa(v, J_Hv) = \kappa |v_D|^2 + |v_X|^2 + \omega(v_X, (JY-Yi)v_D) \geq \kappa |v_D|^2 + |v_X|^2 - |JY-Yi| |v_D| |v_X|.
\]
This quadratic form is positive if $\kappa > |JY-Yi|^2/4$.  The norm on $TD$ is that induced by $\omega_D$ and $i$, the one on $TX$ is that induced by $\omega$ and $J$, and the one applied to $JY-Yi$ is the corresponding Frobenius norm (sum of the squares of matrix entries with respect to orthonormal bases).

So fix $\kappa$, and restrict attention to $J$ and $H$ satisfying $|JY-Yi|^2 < 4\kappa$.  Given a sequence $(J_\nu, H_\nu)$ which $C^\infty$-converges to $(J, H)$, and a sequence $u_\nu$ of $(J_\nu, H_\nu)$-holomorphic discs of bounded area in $X$ (hence also bounded area in $D \times X$), Gromov compactness in $D \times X$ shows that some subsequence of the $u_\nu$ Gromov-converges to a limit $u$ which is a stable $J_H$-holomorphic disc with boundary on $\pd D \times L$.  Since the projection $\pi_D$ is holomorphic, any sphere component of $u$ must lie in a fibre.  Similarly, any disc component must lie in a fibre (necessarily over $\pd D$) or its projection to $D$ must have strictly positive degree.  Since each $\pi_D \circ u_\nu$ has degree $1$, we conclude that $u$ has a single disc component which is a section of $\pi_D$ (contributing $1$ to $\deg \pi_D \circ u$), and all other components lie in the fibres.  Since $J_H$ restricted to the fibres coincides with $J$, these fibre components must $J$-holomorphic.  Translating this into $X$, the limit curve $u$ has a `main' disc component, which is $(J,H)$-holomorphic, and possibly some bubble components, which are just $J$-holomorphic.

The approach taken in \cite{BCQS} is slightly different, in that the symplectic form $\omega_\kappa$ is deformed using $H$.  This makes $J_H$ compatible with it, rather than just tamed by it, but this is not necessary for Gromov compactness.  In \cite{BigMcS,SeidelBook} a deformed energy identity is used to prove compactness, and whilst this is more appropriate for non-compact curves it is not needed in our setting.

%--------------------------------

\section{The closed--open map for homogeneous Lagrangians}\label{secCO}

In this section we introduce a technique for computing part of the closed--open map in the presence of continuous symmetries, in the form of $K$-homogeneity.

%---------------------------------

\subsection{Homogeneity}
\label{sscHomog}

Recall from \cref{labHomLag} that a monotone Lagrangian $L \subset X$ is called $K$-homogeneous if $X$ is compact and K\"ahler, $K$ is a compact connected Lie group acting on $X$ by holomorphic automorphisms, and $L$ is a $K$-orbit.  We say $L$ is \emph{sharply} homogeneous if $\dim L = \dim K$.  For the remainder of this section fix a $K$-homogeneous monotone Lagrangian $L$, satisying the usual conditions from \cref{sscCOintro}.  In this subsection we make some preliminary observations.

Following \cite[Section 15.1]{HiNe}, a \emph{universal complexification} of a (real) Lie group $H$ is a complex Lie group $H_\C$ together with a Lie group morphism $\eta_H : H \rightarrow H_\C$ such that any Lie group morphism $H \rightarrow H'$, with $H'$ a complex Lie group, factors uniquely through $\eta_H$.  Universal complexifications always exist \cite[Theorem 15.1.4]{HiNe}, and are unique up to isomorphism in the obvious sense, although the map $\eta_H$ need not be an injection or even an immersion.  For compact groups, however, the universal complexification is well-behaved \cite[Theorem 15.2.1]{HiNe}: $\eta_K : K \rightarrow K_\C$ is injective, so we can view $K$ as a subgroup of $K_\C$, the Lie algebra of $K_\C$ is canonically isomorphic to $\mathfrak{k} \otimes \C$ (where $\mathfrak{k}$ is the Lie algebra of $K$) and the polar map
\[
\Phi : \mathfrak{k} \times K \rightarrow K_\C \text{\quad given by\quad} (\xi, k) \mapsto \exp_{K_\C}(i\xi) k
\]
(or $k\exp_{K_\C}(i\xi)$) is a diffeomorphism.  The prototypical example is the inclusion of the unit circle $\U(1)$ in $\C^*$, or more generally the inclusion of the torus $T^n$ in $(\C^*)^n$ or of the unitary group $\U(n)$ (or special unitary group $\SU(n)$) in $\mathrm{GL}(n, \C)$ (respectively $\SL(n, \C)$).

In the case of our $K$-homogeneous $(X, L)$ we denote the universal complexification $K_\C$ by $G$.

\begin{lem} \label{labCxAction}  The action of $K$ on $X$ complexifies to a holomorphic action of $G$ on $X$.  The $G$-orbit $W$ containing $L$ is dense, and its complement $Y$ is analytically closed.
\end{lem}
\begin{proof} The first claim is a standard application (see, e.g.~\cite[Section 4]{GuSt}) of the universal property of complexifications, using the fact that since $X$ is compact its group of holomorphic automorphisms has the structure of a complex Lie group \cite[Section III Theorem 1.1]{Ko}.

Now take a basis $\xi_1, \dots, \xi_m$ for $\mathfrak{k}$ and for each subset $I=\{i_1 < \dots < i_n\}$ of $\{1, \dots, m\}$ consider the holomorphic section $\sigma_I$ of $K_X^{-1} = \Lambda_\C^n TX$ (i.e.~the top exterior power of $TX$ over $\C$) defined by
\[
\sigma_I(x) = (\xi_{i_1} \cdot x) \wedge \dots \wedge (\xi_{i_n} \cdot x).
\]
The common zeros of the $\sigma_I$ form a proper analytically closed subset $Y$ of $X$, whose complement $W$ is therefore dense and connected ($X$ itself is connected, and $Y$ has real codimension at least $2$ so cannot disconnect it).  $W$ is partitioned in $G$-orbits, each of which is open since the infinitesimal action on $W$ is surjective, and so by connectedness it is a single orbit.  It clearly contains $L$.
\end{proof}

If $(X, L)$ is sharply $K$-homogeneous then there is only one $n$-tuple $I$, and we denote $\sigma_I$ just by $\sigma$.

\begin{lem}\label{labAntiCanon}  If $(X, L)$ is sharply $K$-homogeneous then:
\begen
\item\label{antitm1} $L$ is parallelisable, so in particular is orientable and spin.
\item\label{antitm2} $Y$ is cut out by the vanishing of the holomorphic section $\sigma$ of $K_X^{-1}$ so is a divisor, meaning an analytic subvariety of complex codimension $1$.
\item \label{antitm2b} $L$ is special Lagrangian in $W$ in the sense of \cite[Definition 2.1]{Au}, meaning that the holomorphic volume form $\Omega = \sigma|_W^{-1}$ on $W$ has constant phase (argument) when restricted to $L$.
\item\label{antitm3} The Maslov index of a holomorphic disc $u : (D, \pd D) \rightarrow (X, L)$ is twice the sum of the vanishing orders of $\sigma \circ u$ at the intersection points of $u$ with $Y$.
\end{enumerate}
\end{lem}
\begin{proof}  \ref{antitm1} The infinitesimal action of $\mathfrak{k}$ on $L$ exhibits an isomorphism between $TL$ and the trivial bundle $\mathfrak{k} \times L$.

\ref{antitm2} We just saw this above.

\ref{antitm2b} The section $\sigma_L$ lies in $\Lambda^n_\R TL \subset \Lambda^n_\C TX|_L$, so $\Omega|_L$ is real.

\ref{antitm3} This is equivalent to Auroux's result \cite[Lemma 3.1]{Au}, which is proved by equating the twisting of $\Lambda^n_\R TL$ around the boundary of $u$ with that of $\R \sigma$, and using the argument principle.
\end{proof}

\begin{defn}\label{StdSpin}
The \emph{standard} spin structure on a sharply $K$-homogeneous Lagrangian $L$ is that induced by the trivialisation $TL = \mathfrak{k} \times L$.  This coincides with the definition for toric fibres given by Cho \cite{ChoCl} and Cho--Oh \cite{ChoOh} using toric charts.
\end{defn}

Next recall that a disc $u : (D, \pd D) \rightarrow (X, L)$ is axial if, up to reparametrisation, it is of the form $u(z) = e^{-i\xi \log z} u(1)$ on $D \setminus \{0\}$, where $e^\cdot$ is the exponential map $\mathfrak{g} = \mathfrak{k} \otimes \C \rightarrow G$ and $\xi \in \mathfrak{k}$ is such that $e^{2\pi \xi}$ fixes $u(1)$.  Given such a $\xi$, the expression $e^{-i\xi \log z}$ always defines a holomorphic map $(D\setminus \{0\}, \pd D) \rightarrow (X, L)$, but it need not patch over $0$.  For example, if $X$ is the torus $\C / (\Z\oplus i\Z)$, $L$ is $\R/\Z \subset X$, $K$ is $\U(1)$ acting by translations in the real direction, and $\xi$ is $i \in i\R = \mathfrak{u}(1)$, then the punctured disc $z \mapsto e^{-i \xi \log z}\cdot 0$ wraps around the torus infinitely many times as $z \rightarrow 0$.  However, for Hamiltonian actions such expressions do always patch over.

\begin{lem}\label{labHamPatch}  Suppose the $K$-action is Hamiltonian with moment map $\mu : X \rightarrow \mathfrak{k}^*$.  If $x$ is a point in $L$ and $\xi \in \mathfrak{k}$ satisfies $e^{2\pi \xi}x=x$, then $u : z \mapsto e^{-i \xi \log z}x$ extends continuously, and hence holomorphically, over $0$.
\end{lem}
\begin{proof}  By removal of singularities \cite[Theorem 4.2.1]{McS} it suffices to show $u$ has finite energy. Taking polar coordinates $r$ and $\theta$ on $D \setminus \{0\}$, the energy density is given by
\[
\omega(\pd_r u, \pd_\theta u) \ \diff r \wedge \diff \theta = \omega(\pd_r u, \xi \cdot u) \ \diff r \wedge \diff \theta = \ip{\pd_r u \, \lrcorner \, \diff \mu}{\xi} \ \diff r \wedge \diff \theta,
\]
where $\ip{\cdot}{\cdot}$ denotes the pairing between $\mathfrak{k}^*$ and $\mathfrak{k}$ (Cho--Oh did this computation in the toric case in \cite[Theorem 8.1]{ChoOh}).  The right-hand side is simply $u^*\diff\ip{\mu}{\xi} \wedge \diff\theta$, so the total energy is at most
\[
2 \pi (\max_X \ip{\mu}{\xi}-\min_X \ip{\mu}{\xi}) < \infty.\qedhere
\]
\end{proof}

%----------------------------

\subsection{Computing the closed--open map}
\label{sscComputingCO}

Equip $L$ with a relative spin structure and local system to give a $K$-homogeneous monotone Lagrangian brane $L^\flat$.  We abbreviate $QH^*(X; \Lambda)$ and $HF^*(L^\flat, L^\flat; \Lambda)$ to $QH^*$ and $HF^*$ respectively.  Now fix a subvariety $Z \subset X$ which is $K$-invariant setwise and disjoint from $L$.  The inclusion of its smooth locus defines a pseudocycle $F : M \rightarrow X$, which carries a canonical complex orientation, and by definition the homology class of $Z$ is $[F]$.

Our goal is to compute $\CO^0(\PD(Z))$.  Recall from \cref{lemCOpseudo} that this counts rigid trajectories as shown in \cref{figCOpseudoTraj}, where $u_1$ carries a Hamiltonian perturbation.
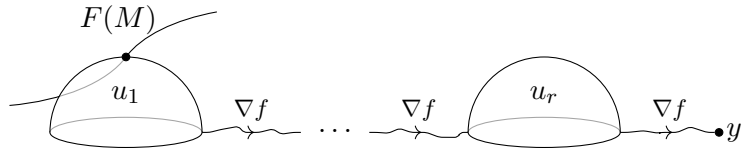
\begin{figure}[ht]
\centering
\begin{tikzpicture}
[blob/.style={circle, draw=black, fill=black, inner sep=0, minimum size=\blobsize}]
\definecolor{green}{rgb}{0, 0.9, 0.2}
\def\blobsize{1mm}

\draw (0, 1) arc (20:80: 2 and -1);
\draw (0, 1) arc (20:70: -2 and 1);
\draw (-0.1, 1.5) node{$F(M)$};
\critpt{7.8}{0}{0.2}{0}{y}{y}

\disc{-1}{1}{0}{$u_1$}
\flow{(1, 0)}{(2.3, 0)}{90}{$\nabla f$}{0}{0.3}
\draw (2.75, 0) node{$\smash\dots$};
\flow{(3.2, 0)}{(4.5, 0)}{93}{$\nabla f$}{0}{0.3}
\disc{4.5}{6.5}{0}{$u_r$}
\flow{(6.5, 0)}{(y)}{94}{$\nabla f$}{0}{0.3}

\draw (0, 1) node[blob]{};

\end{tikzpicture}
\caption{A pearly trajectory contributing to the coefficient of $y$ in $\CO^0(\PD(Z))$.\label{figCOpseudoTraj}}
\end{figure}
Letting $k$ denote the complex codimension of $Z$, the rigidity condition amounts to $|y| = 2k-\mu(A)$, where $A$ is the total disc class.

The strategy is roughly to consider such trajectories for the standard complex structure $J_\mathrm{std}$, without a Hamiltonian perturbation, and show by a disc classification result that the only possible trajectories have a single disc (i.e.~$r=1$) and output at the Morse minimum (we assume for simplicity that there is a unique local min).  Moreover, any such disc meets $Z$ in its smooth locus and is axial, and these trajectories are transversely cut out.

We then show that for sufficiently nearby $F$-regular auxiliary data the trajectories contributing to $\CO^0(\PD(Z))$ have the same shape, namely a single disc and output at the minimum.  Using a cobordism we show that their count coincides with the count for $J_\mathrm{std}$, and verify that for the standard spin structure all trajectories contribute positively.

In order to make the comparison between the $F$-regular data and $J_\mathrm{std}$, we need to rule out sphere bubbling at the interior marked point for $J_\mathrm{std}$.  This is where the assumption that $k \leq N_X^+$ enters, where $N_X^+ \in \Z_{>0} \cup \{\infty\}$ is the minimal Chern number of $J_\mathrm{std}$-holomorphic spheres.

%-------------------------------

\subsection{Axial discs}\label{sscAxDiscs}

The first step is to analyse $J_\mathrm{std}$-holomorphic discs.  In this subsection we recall some ideas from \cite{EL1} which we shall need.

\begin{defn}
For an $\omega$-compatible almost complex structure $J$ and a class $A$ in $H_2(X, L; \Z)$, let
\[
\mathcal{M}(A, J) = \{u: (D, \pd D) \rightarrow (X, L) : \conj{\pd}_Ju=0 \text{ and } [u] = A\}
\]
be the moduli space of $J$-holomorphic discs in class $A$.  This has virtual dimension $n+\mu(A)$.  For a point $p$ in $D$ we write $\ev_p$ for the evaluation map at $p$.
\end{defn}

Using $K$-homogeneity, Evans--Lekili made the following important observation

\begin{prop}[{\cite[Lemma 2.11, Lemma 3.2]{EL1}}]
\label{ModSmooth}
For all classes $A$ the space $\mathcal{M}(A, J_\mathrm{std})$ is a smooth manifold of dimension $n+\mu(A)$.
\end{prop}
\begin{proof}[Sketch proof]
Given a disc $u$ in $\mathcal{M}(A, J_\mathrm{std})$, we need to show that the $\conj{\pd}$-operator is surjective when acting on sections of $u^*TX$ that lie in $u^*TL$ over the boundary.  By Oh's splitting theorem \cite[Theorem I]{Oh}, the bundle pair $(u^*TX, u|_{\pd D}^*TL)$ decomposes into summands of the form $(\underline{\C}, z^{\kappa/2}\underline{\R})$, where $\underline{\C}$ and $\underline{\R}$ are the trivial bundles over the disc and $z$ is the standard complex coordinate, and it suffices to show that each `partial index' $\kappa \in \Z$ is at least $-1$.

We claim that each $\kappa$ is in fact non-negative, and for this it suffices to show that the infinitesimal evluation map $D\ev_p : \ker \conj{\pd} \rightarrow T_{u(p)}L$ at a point $p$ in $\pd D$ is surjective.  To prove that this is indeed the case, note that for any $\xi$ in $\mathfrak{k}$ the map $z \mapsto \xi \cdot u(z)$ defines an element of $\ker \conj{\pd}$ (i.e.~a holomorphic section of $(u^*TX, u|_{\pd D}^*TL)$) which evaluates to $\xi \cdot u(p)$ at $p$.  Since the infinitesimal $K$-action is surjective at each point of $L$, this proves the result.
\end{proof}

\begin{rmk}
\label{SphereModSmooth}
The reader may be more familiar with the corresponding argument for holomorphic spheres $u : \C\P^1 \rightarrow X$, which is completely analogous: after decomposing $u^*TX$ into a sum of line bundles using Grothendieck's splitting theorem, one must check that each summand has Chern number $\geq -1$, and if one can show that $D\ev_p (\ker \conj{\pd})$ spans $T_{u(p)}X$ for some $p$ then this proves that in fact all summands have Chern number $\geq 0$.  This holds if $u(p)$ lies in the open $G$-orbit $W$.
\end{rmk}

Each space $\mathcal{M}(A, J_{\mathrm{std}})$ carries an action of $K$, by translation, and $\mathrm{PSL}(2, \R)$ by reparametrisation.  By combining these actions, Evans and Lekili showed

\begin{prop}[{\cite[Corollary 3.10]{EL1}}]
\label{Ind2Ax}
Any $J_\mathrm{std}$-holomorphic disc of index $2$ is axial.
\end{prop}
\begin{proof}[Sketch proof]
Let $u$ be an element of $\mathcal{M} = \mathcal{M}(A, J_\mathrm{std})$, where $\mu(A)=2$.  By \cref{ModSmooth}, $\mathcal{M}$ is a smooth $(n+2)$-manifold.  The infinitesimal actions of $K$ and $\mathrm{PSL}(2, \R)$ at $u$ have ranks $\geq n$ and $3$ respectively, so there exist non-zero $\xi \in \mathfrak{k}$ and $\eta \in \mathfrak{psl}(2, \R)$ such that $\xi \cdot u = \eta \cdot u$ in $T_u\mathcal{M}$.  A little work (see \cite[Lemma 2.5]{Sm}) shows that $\eta$ is conjugate to a generator of a rotation, so after reparametrising $u$ and rescaling $\xi$ we obtain $e^{\theta \xi}u(z) = u(e^{i\theta}z)$ for all $z$ and $\theta$.  This shows that along its boundary $u$ agrees with the axial disc generated by $\xi$, so by a Schwarz reflection argument the two discs coincide.
\end{proof}

They also showed \cite[Corollary 3.11]{EL1} that an index $4$ disc cleanly intersecting a $K$-invariant subvariety of complex codimension $2$ is axial.  We wish to extend this to any index $2k$ disc meeting a $K$-invariant subvariety of complex codimension $2k$, and show that the corresponding moduli space is transversely cut out.

%---------------------------------------

\subsection{Discs meeting invariant subvarieties}

We next prove our main disc classification result, \cref{DiscClassification}, which is based on some preliminary transversality statements.  Before getting into these we introduce some terminology.  Given a smooth manifold carrying an action of a Lie group $H$, and a vector subspace $\mathfrak{h}'$ of the Lie algebra $\mathfrak{h}$, say that a submanifold $S$ is \emph{$\mathfrak{h}'$-invariant} if for all $p \in S$ we have $\mathfrak{h}'\cdot p \subset T_pS$, where the left-hand side is the infinitesimal action of $\mathfrak{h}'$ at $p$.

The key technical ingredient is

\begin{lem}\label{labInvtTrans}  If $S \subset X$ is a submanifold of which is $i \mathfrak{k}$-invariant, using the complexified action from \cref{labCxAction}, then for all classes $A$ the map $\ev_0 : \mathcal{M}(A, J_\mathrm{std}) \rightarrow X$ is transverse to $S$.
\end{lem}
\begin{rmk}
$\mathcal{M}(A, J_\mathrm{std})$ is smooth by \cref{ModSmooth} so transversality makes sense.
\end{rmk}
\begin{proof}  Fix an arbitrary $u$ in $\mathcal{M} = \mathcal{M}(A, J_\mathrm{std})$ with $\ev_0(u)$ in $S$, and let $p$ denote this point $u(0) \in S$.  Now take a vector $w \in T_pX$.  We need to find a tangent vector $\mathcal{M}$ at $u$, i.e.~a holomorphic section $v$ of $u^*TX$ which restricts to a section of $u|_{\pd D}^* TL$ on the boundary, such that $w-v(0)$ lies in $T_pS$.

First extend $w$ arbitrarily to a holomorphic local section $\tilde{w}$ of $u^*TX$.  We aim to write this section in terms of the $\mathfrak{g}$-action on a punctured neighbourhood of $0$, so pick $\xi_1, \dots, \xi_n \in \mathfrak{k}$ such that the $\xi_j \cdot u(1)$ form a basis of $T_{u(1)} L$.  Now consider the holomorphic sections $v_1, \dots, v_n$ of $u^*TX$ defined by $v_j(z) = \xi_j \cdot u(z)$, and the corresponding section $V \coloneqq v_1 \wedge \dots \wedge v_n$ of $u^* \Lambda^n_\C TX$.  By construction $V$ is holomorphic and not identically zero, so if it vanishes at $0$ then this zero must be isolated.  We conclude that the $v_j$ are fibrewise linearly independent on a punctured neighbourhood of $0$.  On such a neighbourhood there thus exist holomorphic functions $f_1, \dots, f_n$ such that
\[
\tilde{w} = \sum_j f_j v_j.
\]
If the $f_j$ were defined on the whole punctured disc and were real-valued over the boundary then $\sum_j f_j v_j$ would give the desired tangent vector to $\mathcal{M}$.  In general this will not be the case, but we will modify them to produce functions $F_j$ which do have these properties.

For each $j$ we have
\[
v_1 \wedge \dots v_{j-1} \wedge \tilde{w} \wedge v_{j+1} \wedge \dots \wedge v_n = f_j V
\]
and the left-hand side is holomorphic over $0$, so if $d$ denotes the vanishing order of $V$ at $0$ then $f_j$ has at worst a pole of order $d$.  Let the principal part of the Laurent series of $f_j$ (including the constant term) be
\[
\sum_{l=0}^d a_{jl}z^{-l},
\]
and let $a_{j0}$ have real and imaginary parts $b_j$ and $c_j$ respectively.  Now define a function $F_j$ by
\[
F_j =b_j + \sum_{l=1}^d \lb a_{jl}z^{-l}+\conj{a}_{jl}z^l\rb.
\]
Note that this is holomorphic on the punctured disc, and is real over the boundary.  Moreover, on a punctured neighbourhood of $0$ we have $f_j-F_j = ic_j + r_j$, where $r_j$ is a holomorphic function which vanishes at $0$.

Finally, let $v$ be the the holomorphic section of $u^*TX$ over $D \setminus \{0\}$ given by
\[
v = \sum_j F_jv_j.
\]
We claim this is the desired tangent vector to $\mathcal{M}$.  Since the $F_j|_{\pd D}$ are real, $v$ restricts to a section of $u|_{\pd D}^* TL$, so we just need to check that it extends over $0$ and satisfies $u(0) = w$ in $T_pX/T_pZ$.  To prove this, note that on a punctured neighbourhood of $0$ we have
\[
\tilde{w}-v = \sum_j (f_j - F_j)v_j = \sum_j (ic_j+r_j)\xi_j \cdot u.
\]
The right-hand side is holomorphic over $0$, so $v$ itself must be, and we have
\[
w-v(0) = \sum_j ic_j \xi_j \cdot p \in i \mathfrak{k} \cdot p \subset T_pZ,
\]
which gives the desired equality in $T_pX/T_pZ$.
\end{proof}

We also need an analgous result for spheres.  For $A \in H_2(X; \Z)$ write $\mathcal{M}^\mathrm{sph}_W(A, J_\mathrm{std})$ for the space of $J_\mathrm{std}$-holomorphic maps $\C\P^1 \rightarrow X$ which meet the open $G$-orbit $W$.  By \cref{SphereModSmooth} this is a smooth manifold of dimension $2n+\mu(A)$.

\begin{lem}\label{labSphereInvtTrans}  If $S \subset X$ is an $\mathfrak{g}$-invariant submanifold and $q$ is a point of $L$ then for all $A$ the map $\ev_0 \times \ev_\infty : \mathcal{M}^\mathrm{sph}_W(A, J_\mathrm{std}) \rightarrow X \times L$ is transverse to $S \times \{q\}$.
\end{lem}
\begin{proof}
We closely follow \cref{labInvtTrans}.  Take a sphere $u$ in $\mathcal{M} = \mathcal{M}^\mathrm{sph}_W(A, J_\mathrm{std})$ with $u(0)=p \in S$ and $u(\infty)=q$, and vectors $w_X \in T_pX$ and $w_L \in T_qL$.  We seek a holomorphic section $v$ of $u^*TX$ such that $w_X-v(0)$ lies in $T_pS$ and $w_L-v(\infty)=0$.

Extend $w$ to a holomorphic local section $\tilde{w}$ of $u^*TX$, and pick $\xi_1, \dots, \xi_n$ such that the sections $v_j : z \mapsto \xi_j \cdot u(z)$ span $T_qL$ at $\infty$.  Write $\tilde{w}=\sum_j f_j v_j$ on a punctured neighbourhood of $0$, and let $F_j$ be the principal part of $f_j$ (without the constant term).  Finally let $\lambda_1, \dots, \lambda_n \in \C$ be such that $\sum_j \lambda_j v_j(\infty)=w_L$.  Then
\[
v \coloneqq \sum_j (F_j+\lambda_j) v_j
\]
is a holomorphic section of $u^*TX$ with the required properties.
\end{proof}

Note that we require $S$ to be $\mathfrak{g}$-invariant, not just $i\mathfrak{k}$-invariant as in \cref{labInvtTrans}.  This is because the constant term in $v$ differs from that in $\tilde{w}$ by a $\C$-linear combination of the $v_j$, not just an $i\R$-linear combination.

\begin{defn}
For a subset $S \subset X$ and an integer $j$, we write $\mathcal{M}^S_{\mu=2j}$ for the space of index $2j$ discs mapping $0$ to $S$.
\end{defn}

We can now state and prove our main disc classification result.

\begin{thm}
\label{DiscClassification}
Let $(X, L)$ be $K$-homogeneous and $Z \subset X$ a subvariety of complex codimension $k$, which is $K$-invariant setwise.  Suppose that $u$ is a $J_\mathrm{std}$-holomorphic stable disc of index $\leq 2k$ meeting $Z$, and that $k \leq N_X^+$.  Then $u$ has index exactly $2k$ and comprises a single component, which is an axial disc meeting $Z$ in its smooth locus.
\end{thm}
\begin{proof}
First let $S_1$ denote the smooth locus $M$ of $Z$---this is a $K$-invariant complex submanifold of $X$.  The complement $Z \setminus S_1$ is a proper subvariety of $Z$, and is also $K$-invariant, so the smooth loci of its components are again $K$-invariant complex submanifolds.  Repeating, we partition $Z$ into $K$-invariant complex submanifolds $S_1, \dots, S_l$ such that $\codim_\C S_j > k$ for all $j>1$.

Now let $\hat{u}$ denote a non-constant component of $u$ meeting $Z$.  We claim that $\hat{u}$ must be a disc.  Given this, after reparametrising we may assume that $\hat{u}(0)$ lies in $Z$.  Then $\hat{u}$ is an element of $\mathcal{M}^{S_j}_{\mu=2m}$ for some $j$ and some $m \leq k$.  By \cref{labInvtTrans} this moduli space is transversely cut out and is thus a smooth manifold of dimension $n+2m-2\codim_\C S_j$.  Moreover, it carries a $K$-action by translation, which commutes with $\ev_1$, so
\[
\ev_1 : \mathcal{M}^{S_j}_{\mu=2m} \rightarrow L
\]
is a submersion.  In particular, $\mathcal{M}^{S_j}_{\mu=2m}$ has dimension at least $n$, so we must have $m \geq \codim_\C S_j$.  The only possibility is that $j=1$ and $m=k$, i.e.~$\hat{u}$ has index $2k$ so is the unique component of $u$, and it meets $Z$ in its smooth locus.

To show axiality we follow \cref{Ind2Ax}.  We know that $\mathcal{M}^{S_1}_{\mu=2k}$ is a smooth $n$-manifold and carries a $K$-action by translation and a $\mathrm{U}(1)$-action by rotational reparametrisation.  By counting dimensions, for any disc $u$ in this space there exists $\xi$ in $\mathfrak{k}$ such that $\xi \cdot u$ coincides with the infinitesimal rotation of $u$, i.e.~$Du(\pd_\theta)$.  As in \cref{Ind2Ax} we conclude that $u$ is axial.

We are left to prove the claim that $\hat{u}$ is a disc, so suppose for contradiction that it is a sphere.  In this case, it must have index at least $2N_X^+$, so we have equalities in all of
\[
\mu(\hat{u}) \leq \mu(u) \leq 2k \leq 2N_X^+.
\]
In particular, all other components of $u$ are constant, so $\hat{u}$ must meet $L$ as well as $Z$ (which it does in $S_j$ say).  After reparametrising we may assume that $\hat{u}(0) \in S_j$ and $\hat{u}(\infty) \in L$.  By \cref{labSphereInvtTrans} the space of holomorphic spheres $\tilde{u}$ of index $2k$ with $\tilde{u}(0) \in S_j$ and $\tilde{u}(\infty) = \hat{u}(\infty)$ is a smooth manifold of dimension $2k - 2\codim_\C S_j \leq 0$.  However, it also contains $\hat{u}$ and carries a non-trivial action of $\C^*$ by reparametrisation, so must have dimension at least $2$.  This gives the desired contradiction, showing that $\hat{u}$ must have been a disc, and completing the proof.
\end{proof}

%--------------------------------

\subsection{Proof of \cref{mthmCO}}

We are now ready to return to our objective of computing $\CO^0(\PD(Z))$ as outlined in \cref{sscComputingCO}.  Recall that $L^\flat$ is a $K$-homogeneous monotone Lagrangian brane, and $Z$ a $K$-invariant subvariety of complex codimension $k \leq N_X^+$, disjoint from $L$ and represented by the pseudocycle $F:M \rightarrow X$ given by the inclusion of its smooth locus.

Fix a Morse--Smale pair $(f,g)$ such that $f$ has a unique local minimum $y_\mathrm{min}$, and consider auxiliary data $\mathscr{D}=(f, g, J, H)$.  We do not yet assume they are $F$-regular in the sense of \cref{defFregular}.  If $(J, H)=(J_\mathrm{std}, 0)$ then \cref{DiscClassification} tells us that that only trajectories of shape \cref{figCOpseudoTraj} of virtual dimension $\leq 0$ (which amounts to $|y|+\mu(A) \leq 2k$) have a single disc, which is axial, and have output $y_\mathrm{min}$.  We would like to transfer this property, minus the axiality statement, to nearby $(J, H)$, and this is the subject of the next result.

In light of \cref{sscHamPerts}, by a $(J, H)$-holomorphic stable disc we mean a stable disc with one distinguished disc component that is $(J, H)$-holomorphic, and all other components $J$-holomorphic.  When talking about such discs meeting $Z$, we require that the corresponding interior marked point is attached to the main component (as opposed to some other disc component), possibly via a tree of sphere components.

\begin{lem}
\label{DiscNearbyJH}
There exists a $C^\infty$-open neighbourhood $U$ of $(J_\mathrm{std}, 0)$ such that for all $(J, H)$ in $U$ the following holds: any $(J, H)$-holomorphic stable disc $u$ of index $\leq 2k$ and meeting $Z$ has index exactly $2k$, comprises just the main component, and meets $Z$ in its smooth locus.
\end{lem}
\begin{proof}
Suppose for contradiction that no such $U$ exists.  Then there exists a sequence $(J_\nu, H_\nu)$ converging to $(J_\mathrm{std}, 0)$, and a sequence of $(J_\nu, H_\nu)$-holomorphic stable discs $u_\nu$ meeting $Z$, such that for each $\nu$ either: the main component of $u_\nu$ has index $<2k$; or $u_\nu$ has a sphere component; or $u_\nu$ meets the singular locus of $Z$.

By \cref{sscHamPerts} there exists a subsequence which Gromov converges to a $J_\mathrm{std}$-holomorphic stable disc $u$ which hits $Z$ and either: has main component of index $< 2k$; or has a sphere component; or meets the singular locus of $Z$.  In the first case we may assume that there are no sphere components (otherwise we're the second case), so that the main component itself meets $Z$.  Then in any of the three cases no such $u$ can exist, by \cref{DiscClassification}, so we obtain the desired contradiction.
\end{proof}

Using this we deduce

\begin{prop}
\label{TrajNearbyJH}
Suppose $(J,H)$ lies in the set $U$ provided by \cref{DiscNearbyJH}, and consider the contribution to $\CO^0_\mathrm{pseudo}(\PD[F])$ of total class $A$ using data $\mathscr{D}=(f,g,J,H)$, which need not be $F$-regular.  The only contributions have $\mu(A)=2k$ and output $y_\mathrm{min}$, and the corresponding count of trajectories is the count of points in the (oriented) fibre product
\begin{equation}
\label{eqCOfibre}
M \times_X \mathcal{M}(A, (J, H)) \times_L \{y_\mathrm{min}\},
\end{equation}
where $M$ is the smooth locus of $Z$ and $\mathcal{M}(A, (J, H))$ maps to its left by $\ev_0$ and its right by $\ev_1$.  Moreover, the union of these spaces over any compact family of $(J, H)$'s in $U$ (so in particular over any path in $U$) is compact.
\end{prop}
\begin{proof}
By \cref{DiscNearbyJH} any trajectory contributing to $\CO^0_\mathrm{pseudo}(\PD[F])$ for $(J, H)$ in $U$ has a single disc, of index $2k$, and outputs $y_\mathrm{min}$.  The description of the oriented moduli spaces from \cite[Section A.2.3]{BCEG}, with the Morse input replaced by the pseudocycle, then reduces to the claimed fibre product.

We are left to show compactness, so suppose we have a sequence $(f, g, J_\nu, H_\nu)$ of data and a corresponding sequence of trajectories.  Each trajectory is described by its unique disc $u_\nu$, and by passing to a subsequence we may assume that $(J_\nu, H_\nu)$ converges to some $(J, H)$ in $U$ and $u_\nu$ converges to some $(J,H)$-holomorphic stable disc $u$ meeting both $Z$ (through its main component) and $y_\mathrm{min}$.  By \cref{DiscNearbyJH} this limit $u$ has only the main component, which is an index $2k$ disc meeting $Z$ in its smooth locus, so the corresponding trajectory lies in \eqref{eqCOfibre}.
\end{proof}

We are now in a position to precisely formulate and prove \cref{mthmCO}.

\begin{thm}[\cref{mthmCO}]
\label{thmCO}
Let $L^\flat \subset X$ be a $K$-homogeneous monotone Lagrangian brane, and $Z \subset X$ a setwise $K$-invariant subvariety of complex codimension $\leq N_X^+$, Poincar\'e dual to a class $\alpha$ in $QH^*(X; \Lambda)$.  Then
\[
\CO^0(\alpha) = T^{|\alpha|/N_L} \Big(\sum_u \pm \hol_{\mathscr{F}} (-\pd u)\Big) \cdot 1_L \in HF^*(L^\flat, L^\flat; \Lambda),
\]
where the sum is over axial index $2k$ discs $u$ satisfying $u(0) \in Z$ and $u(1) = y_\mathrm{min}$, and where $\hol_\mathscr{F} (-\pd u)$ denotes the holonomy of the local system $\mathscr{F}$ around the boundary of $u$, as an element of $\End_R \mathscr{F}_{y_\mathrm{min}}$ (as appearing in \eqref{eqPearlCx}).  The signs depend on the choice of relative spin structure, and are all $+$ for the standard spin structure (this is only defined if $L$ is \emph{sharply} $K$-homogeneous).
\end{thm}
\begin{rmk}
The appearance of $-\pd u$ rather than $\pd u$ is purely a result of our conventions regarding the direction of parallel transport.
\end{rmk}
\begin{proof}
By \cref{lemCOpseudo} we may compute $\CO^0(\alpha)$ via $\CO^0_\mathrm{pseudo}(\PD[F])$, for $F$-regular data $\mathscr{D}=(f, g, J, H)$, and by \cref{TrajNearbyJH} the latter is given by
\[
T^{|\alpha|/N_L} \Big(\sum_{\substack{A\\ \mu(A)=2k}} \#\big(M \times_X \mathcal{M}(A, (J, H)) \times_L \{y_\mathrm{min}\}\big) \cdot \hol_\mathscr{F} (-\pd u) \Big) \cdot 1_L.
\]
We claim that the $(J,H)$ can be replaced by $(J_\mathrm{std}, 0)$.  \Cref{DiscClassification} then tells us that the discs appearing are all axial, which gives the result, up to the final statement about signs.

In order to prove the claim note that for each $A$ the fibre product \eqref{eqCOfibre} is transversely cut out for both $(J,H)$ and $(J_\mathrm{std}, 0)$.  The former is by $F$-regularity, whilst the latter is by the fact that $\mathcal{M}^Z_{\mu=2k}$ is transversely cut out and the evaluation map $\ev_1$ from this space to $L$ commutes with the $K$-actions so is a submersion.  We may therefore choose a regular path $(J_t, H_t)$ in $U$ from $(J, H)$ to $(J_\mathrm{std}, 0)$, so that the corresponding $1$-parameter family of fibre products
\[
\bigcup_{t \in [0, 1]} M \times_X \mathcal{M}(A, (J_t, H_t)) \times_L \{y_\mathrm{min}\}
\]
is transversely cut out.  It is compact by the final part of \cref{TrajNearbyJH}, so gives an oriented cobordism between the $(J,H)$ and $(J_\mathrm{std},0)$ spaces, proving the claim.

Finally we need to check the signs, and we defer this to \cref{appOri}.
\end{proof}

\begin{rmk}
\label{rmkCodimCondition}
The codimension condition cannot straightforwardly be removed, as the following example shows.  Take $X$ to be the toric monotone blowup of $\C\P^2$ at three points, with hexagonal moment polytope, and $L$ to be the monotone fibre.  This has $N_X^+=1$.  The superpotential is
\[
W = x + xy + y + \frac{1}{x} + \frac{1}{xy} + \frac{1}{y},
\]
and the terms correspond, in cyclic order, to the monodromies of the local system around the index $2$ basis classes $[u_1], \dots, [u_6]$ in $H_2(X;L; \Z)$ which are dual to the components of the toric divisor.

Suppose for contradiction that \cref{thmCO} applies to the six toric fixed points, so we can compute $\CO^0(\PD(\text{point}))$ by counting index $4$ axial discs through any one of them, weighted by monodromy.  Each of the fixed points is hit by a unique such disc (after imposing $u(1)=y_\mathrm{min}$), in classes $[u_1]+[u_2], \dots, [u_5]+[u_6], [u_6]+[u_1]$, so the corresponding monodromies are
\[
x^2y \text{, } xy^2 \text{, } \frac{y}{x} \text{, } \frac{1}{x^2y} \text{, } \frac{1}{xy^2} \text{, and } \frac{x}{y}.
\]
We deduce that if $HF^* \neq 0$ then these six numbers must coincide, but taking $(x,y)=(-1,-1)$ gives a counter-example (this has $HF^*\neq 0$ since $(-1,-1)$ is a critical point of $W$).  The bad bubbled configurations comprise an index $2$ disc attached to the dual divisor, which is a rational curve of Chern number $1$.
\end{rmk}

%-------------------------------------

\section{The Oh spectral sequence and $\symp(X, L)$}\label{secOhSS}

We now study the Oh spectral sequence for (not necessarily homogeneous) monotone Lagrangians, and the action of discrete symmetries in the form of the image of $\symp(X, L)$ in $\pi_0 \mathrm{Diff}(X, L)$.

%---------------------------------

\subsection{The spectral sequence}
\label{SpecSeq}

Fix a monotone Lagrangian brane $L^\flat=(L, s, \mathscr{F}) \subset X$ over the ring $R$, and $\diff$-regular auxiliary data $\mathscr{D} = (f, g, J)$.  We abbreviate the complex $C^*(L^\flat; \mathscr{D}; \Lambda)$ to $C$, and assume that $\mathscr{F}$ has rank $1$.  This allows us to work with \eqref{OldComplex} rather than \eqref{eqPearlCx}, and we denote $HF^*(L^\flat, L^\flat; \Lambda)$ simply by $HF^*$ as usual.  In this subsection we explain the construction of the spectral sequence $\ssE{1} \cong H^*(L; \Lambda) \implies HF^*$ and its basic properties in some detail, as we will make heavy use of it.  It originates in \cite{OhSS} but the pearl complex approach is taken from \cite{BCQS}.

First note that the grading on the chain complex $C$ can be refined to a bigrading $C^{*, *}$, by powers of $T$ and by Morse index.  Explicitly, for an integer $p$ and a critical point $x$ of index $|x|=q$ we say that the term $T^px$ (of degree $N_Lp+q$) has \emph{polynomial degree} $p$ and \emph{Morse degree} $q$, and hence lies in $C^{p, q}$.  The differential $\diff$ decomposes as $\pd_0 + \pd_1 + \pd_2 + \dots$, where each $\pd_j$ is of bidegree $(j, 1-N_Lj)$ (so $\pd_j = 0$ for $j > (n+1)/N_L$), and we obtain a decreasing filtration
\[
\cdots \supset F^{-1}C \supset F^{0}C \supset F^1C \supset \cdots
\]
of $C$ by subcomplexes $F^pC=C^{\geq p, *}$ comprising elements of polynomial degree at least $p$.  This is a filtration of $R$- or $R[T]$-modules, but not of $\Lambda$-modules.  The standard construction for a filtered complex, as described in \cite[Section A.3.13]{Eis} for example, then gives the desired spectral sequence.

Explicitly, the spectral sequence is a sequence $(\ssE{r}, \diff_r)_{r=1}^\infty$ of differential $R$-modules (pages) along with an identification of the cohomology $H(\ssE{r}, \diff_r)$ of the $r$th page with the $(r+1)$th page $\ssE{r+1}$.  There is a concrete description of $\ssE{r}$, given by
\begin{equation}
\label{eqSSPage}
\ssE{r} = \bigoplus_{p \in \Z} \frac{\{z \in F^pC : \diff z \in F^{p+r}C\} + F^{p+1}C}{\lb F^pC \cap \diff F^{p-r+1}C \rb + F^{p+1}C},
\end{equation}
and $\diff_r$ is the differential mapping the $p$th summand (`column') to the $(p+r)$th by
\[
x+y \in \{z \in F^pC : \diff z \in F^{p+r}C\} + F^{p+1}C \mapsto \diff x.
\]
In particular, $\ssE{1}$ is the homology of the associated graded complex $\gr C = \oplus_p F^pC/F^{p+1}C$ coming from our filtration.  Since $\pd_0$ is the Morse differential for $(f, g)$ on $L$, this complex is simply the Morse complex over the ring $\gr \Lambda$, which is canonically identified with $\Lambda$ itself.  In other words,
\[
\ssE{1} \cong H^*(L; \Lambda) \cong H^*(L; R) \otimes_R \Lambda \cong \bigoplus_{p \in \Z} T^p \cdot H^*(L; R).
\]

Since $\pd_j = 0$ for $j\gg 0$ the spectral sequence degenerates after finitely many steps, meaning that $\diff_r$ vanishes for $r$ sufficiently large and hence that the pages stabilise to a limit $\ssE{\infty}=\ssE{\gg 0}$.  There is a filtration on $HF^* = H(C)$, given by the images of the $H(F^pC)$, which we denote by $F^p HF^*$.  The limit page $\ssE{\infty}$ is equal to $\gr HF^*$: the associated graded module of this filtration.

The pages of the spectral sequence naturally inherit a $\Lambda$-module structure from the filtration, with multiplication by $T^j$ acting from the $p$th column of \eqref{eqSSPage} to the $(p+j)$th.  On $\ssE{1}$ this coincides with the $\Lambda$-module structure on $H^*(L; \Lambda)$.  The differentials $\diff_r$ are all manifestly $\Lambda$-linear.  Similarly, $\gr HF^*$ has the structure of a $\Lambda$-module, with $T^j$ shifting up $j$ levels of the filtration, and the isomorphism with $\ssE{\infty}$ is $\Lambda$-linear.

\begin{rmk}
This construction is canonical in the following sense.  Suppose we built the spectral sequence $(\ssE{r}', \diff_r')_{r=1}^\infty$ from different auxiliary data $\mathscr{D}'$.  The comparison map $c : C \rightarrow C'$ respects the filtrations so induces pagewise chain maps $\ssE{r} \rightarrow \ssE{r}'$ between the spectral sequences.  Since $c$ agrees with the classical Morse comparison map on the associated graded complex, the map between the first pages corresponds to the natural identification $\ssE{1} = H^*(L; R) \otimes \Lambda = \ssE{1}'$, and the maps between later pages are all induced by this, and so are all quasi-isomorphisms.
\end{rmk}

The product on $C$ respects the filtration so gives a multiplicative structure on our spectral sequence.  In other words, there is a product $*_r$ on each $\ssE{r}$ (for $r=1, 2 \dots, \infty$), satisfying a Leibniz rule with respect to $\diff_r$, and such that $*_{r+1}$ is the product induced on $\ssE{r+1}=H(\ssE{r})$ by $*_r$.  Note that $*_1$ is precisely the (Morse-theoretic) classical cup product $\smile$, whilst $*_\infty$ is the product induced on $\ssE{\infty} \cong \gr HF^*$ by the Floer product.  Since the chain-level multiplication is $\Lambda$-bilinear, all of the $*_r$ also have this property.

A typical application of this spectral sequence, which will be a key step in some of our later arguments, is the following well-known result due to Biran--Cornea \cite[Proposition 6.1.1]{BCQS}.

\begin{defn}\label{WideNarrow}
$L^\flat$ is \emph{wide} (over $R$) if its Floer cohomology $HF^*(L^\flat, L^\flat; \Lambda)$ is isomorphic to its singular cohomology $H^*(L; \Lambda \otimes_R \End \mathscr{F})$ with local coefficients as a graded $\Lambda$-module (not necessarily canonically), and \emph{narrow} (over $R$) if $HF^*(L^\flat, L^\flat; \Lambda)=0$.
\end{defn}

\begin{prop}\label{labBCWide}  Suppose that $R$ is a field, $\mathscr{F}$ has rank $1$, and that $H^*(L; R)$ is generated as an $R$-algebra by elements of degree at most $m$, with $m \leq N_L-1$.  Then $L^\flat$ is either wide or narrow over $R$, and the latter is only possible if we have equality.
\end{prop}
\begin{proof}[Sketch proof]
If $m < N_L-1$ then all differentials vanish on $H^{\leq m}(L; R)$ for degree reasons, and hence on the whole of $H^*(L; R)$ by the Leibniz rule, so $L^\flat$ is wide.  The same argument applies when $m = N_L-1$ except that on the first page the differential
\[
\diff_1 : H^m(L; R) \rightarrow T \cdot H^0(L; R)
\]
may be non-zero.  In this case, $\diff_1$ hits the unit (as $R$ is a field), so $\ssE{2}=0$ and $L^\flat$ is narrow.
\end{proof}

%-----------------------------------------

\subsection{The action of $\symp(X, L)$}\label{sscMoCo}

In this subsection we prove \cref{mthmSS}.  Recall that $H_2^D$ denotes the image of $\pi_2(X, L)$ in $H_2(X, L; \Z)$, i.e.~the group of homology classes of topological discs on $L$.  The pearl complex can be defined over this ring by weighting counts of pearly trajectories by $T^A$ rather than $T^{\mu(A)/N_L}$, and we refer to this as the \emph{enriched} pearl complex.  This is bigraded by declaring that $T^Ax$ has bidegree $(\mu(A)/N_L, |x|)$, and the construction of the spectral sequence from the previous subsection all carries over to give
\[
\ssE{1} \cong H^*(L;R) \otimes_R R[H_2^D] \implies HF^*(L^\flat, L^\flat;R[H_2^D]).
\]
We shall show that $\symp(X,L)$ can be made to act on the $\ssE{1}$ page in a way that induces an action on the whole spectral sequence.  The precise statement is \cref{labSSrep}.  For this we must restrict to the case of rank $1$ local systems, but in practice this covers many situations of interest.

With this assumption, recall from \cref{sscLocSys} that each trajectory $\gamma$ from $x$ to $y$ contributes ($T^A$ times) $y \otimes m(\pd A)$ to $\diff x$, where $m$ encodes the monodromy and as usual $A \in H_2^D$ is the total homology class of the discs in $\gamma$.  The claimed action on $\ssE{1}$ is
\begin{equation}
\label{eqSSrep}
\phi \cdot (a \otimes T^A) = (-1)^{\ip{\phi_*s-s}{A}} m(\pd \phi^*A) m(\pd A)^{-1} \phi^*(a) \otimes T^{\phi^*A}
\end{equation}
for $\phi \in \symp(X,L)$, $a \in H^*(L; R)$ and $A \in H_2^D$, extended $R$-linearly.  Here $s$ denotes the relative spin structure on $L$, $\phi_*s$ its pushforward (we need not assume $\phi$ preserves the orientation of $L$), and $\phi_*s - s$ the unique element of $H^2(X, L; \Z/2)$ whose action on relative spin structures sends $s$ to $\phi_* s$.  The motivation for this formula will become apparent in the proof of \cref{labSSrep}.

\begin{lem}
This is a right action and respects the product.
\end{lem}
\begin{proof}
For elements $\phi$ and $\psi$ of $\symp(X,L)$ we have
\begin{align*}
\psi \cdot \big(\phi \cdot (a \otimes T^A)\big) &= (-1)^{\ip{\phi_*s-s}{A}} m(\pd \phi^*A) m(\pd A)^{-1} \big(\psi \cdot (\phi^*(a) \otimes  T^{\phi^*A})\big)
\\ &=(-1)^{\ip{\phi_*s-s}{A}+\ip{\psi_*s-s}{\phi^*A}} m(\pd \psi^*\phi^*A) m(\pd A)^{-1} \psi^*\phi^*(a) \otimes  T^{\psi^*\phi^*A}
\end{align*}
for all $a$ and $A$.  We can rewrite $\ip{\psi_*s-s}{\phi^*A}$ as $\ip{\phi_*\psi_*s-\phi_*s}{A}$, so the exponent of $-1$ becomes $\ip{\phi_*\psi_*s-s}{A}$, and we obtain precisely $(\phi\circ\psi)\cdot (a \otimes T^A)$.  This shows we have a right action.

To check compatibility with the product note that it is given by
\[
(a_1 \otimes T^{A_1}) *_1 (a_2 \otimes T^{A_2}) = (a_1 \smile a_2) \otimes T^{A_1+A_2}.
\]
Acting by $\phi$ on each factor on the left-hand side gives
\begin{multline*}
(-1)^{\ip{\phi_*s-s}{A_1+A_2}}m(\pd \phi^*A_1) m(\pd A_1)^{-1}m(\pd \phi^*A_2) m(\pd A_2)^{-1}\big(\phi^*(a_1)\smile \phi^*(a_2)\big) \otimes T^{\phi^*A_1+\phi^*A_2}
 \\ = (-1)^{\ip{\phi_*s-s}{A_1+A_2}}m\big(\pd \phi^*(A_1+A_2)\big) m\big(\pd (A_1+A_2)\big)^{-1}\phi^*(a_1\smile a_2) \otimes T^{\phi^*(A_1+A_2)},
\end{multline*}
which is exactly what we get by acting on the right-hand side.
\end{proof}

This action is manifestly topological, in the sense that it factors through $\pi_0 \mathrm{Diff}(X,L)$.

\begin{thm}[\cref{mthmSS}]
\label{labSSrep}
For any monotone Lagrangian brane $L^\flat \subset X$ with rank $1$ local system over $R$, the differentials in the Oh spectral sequence for $L^\flat$ over $R[H_2^D]$ commute with the $R$-linear $\symp(X,L)$-action on the pages induced by \eqref{eqSSrep} on $\ssE{1}$.
\end{thm}
\begin{proof}
Take $\diff$-regular auxiliary data $\mathscr{D} = (f, g, J)$ so that our enriched pearl complex is
\[
C(L^\flat; \mathscr{D}; R[H_2^D]) = \bigoplus_{x \in \crit (f)} R[H_2^D] \cdot x.
\]
We claim that for each $\phi \in \symp(X, L)$ there exists a filtered endomorphism $\hbar(\phi)$ of this complex which induces \eqref{eqSSrep} on $\ssE{1}$.  Given this we obtain pagewise endomorphisms $\hbar(\phi)_r^*$ which commute with the differentials, and which are themselves induced by the expression \eqref{eqSSrep} on $\ssE{1}$, from which the result follows.

The steps in the construction of $\hbar(\phi)$ are shown in the following diagram of $R$-linear chain maps:
\begin{multline}
\label{MapSequence}
C((L, s, \mathscr{F}); \mathscr{D}; R[H_2^D]) \xrightarrow{\ \delta_s\ } C((L, \phi_*s, \mathscr{F}); \mathscr{D}; R[H_2^D]) \xrightarrow{\ \delta_\mathscr{F}\ } C((L, \phi_*s, \phi_*\mathscr{F}); \mathscr{D}; R[H_2^D])
\\ \xrightarrow{\ h^\phi\ } C((L, s, \mathscr{F}); \phi^*\mathscr{D}; R[H_2^D]) \xrightarrow{\ c\ } C((L, s, \mathscr{F}); \mathscr{D}; R[H_2^D]),
\end{multline}
and we build each map in turn.  In characteristic $2$, without local systems, and with non-enriched coefficients, the last two steps appear in \cite[Section 5.8]{BCQS}.  $\phi^*\mathscr{D}$ denotes the auxilary data
\[
(\phi^*f = f \circ \phi, \phi^*g = g(\diff \phi (\cdot), \diff \phi (\cdot)), \phi^*J = \diff \phi^{-1} \circ J \circ \diff \phi).
\]

\textbf{Step 1:}  We define $\delta_s$ by
\[
T^Ax \mapsto (-1)^{\ip{\phi_*s - s}{A}} T^Ax,
\]
extended $R$-linearly.  By \cref{SpinChangeProp} this is a chain map.

\textbf{Step 2:}  The monodromy factor attached to a pearly trajectory with total disc class $A$ is $m(\pd A)$ for $\mathscr{F}$, and $\phi_*m(\pd A) = m(\pd \phi^* A)$ for the pushforward $\phi_* \mathscr{F}$.  This means that $\delta_{\mathscr{F}}$ defined by
\[
T^Ax \mapsto m(\pd \phi^*A) m(\pd A)^{-1} T^Ax,
\]
and extended $R$-linearly, is a chain map between the second and third complexes in \eqref{MapSequence}.

\textbf{Step 3:}  There is an obvious bijection $\crit(f) \rightarrow \crit(\phi^*f)$ given by $x \mapsto \phi^{-1}(x)$ in $\crit (\phi^*f)$.  Similarly, trajectories contributing to the differential using auxiliary data $\mathscr{D}$ are carried bijectively by $\phi^{-1}$ to trajectories contributing to the differential using $\phi^*\mathscr{D}$, and this preserves $\diff$-regularity (the moduli spaces are identical up to translating everything by the action of $\phi^{-1}$).  To get a chain map we can extend the map on critical points $R$-linearly, but \emph{not $R[H_2^D]$-linearly}.  The problem is that trajectories of class $A$, which are weighted by $T^A$, are carried to trajectories of class $\phi^*A$, which are weighted by $T^{\phi^* A}$.  The solution is to define the map $h^\phi$ to act non-trivially on $R[H_2^D]$, by $T^A \mapsto T^{\phi^* A}$.  In other words, $h^\phi$ is defined by
\[
T^A x \mapsto T^{\phi^* A} \phi^{-1}(x),
\]
extended $R$-linearly.  We assume that the orientations chosen on the ascending and descending manifolds of critical points of $\phi^* f$ are those carried from $f$ by $\phi^{-1}$.

\textbf{Step 4:}
Finally, $c$ is the pearl complex comparison map described in \cref{sscPearlBg}, constructed using a generic Morse cobordism and path of almost complex structures.  This \emph{is} $R[H_2^D]$-linear.  It depends on the choice of auxiliary data for the cobordism, but respects the filtration and induces the classical Morse pullback map on the associated graded complex.

Putting everything together, we define $\hbar(\phi)$ by
\[
\hbar(\phi) : T^Ax \mapsto (-1)^{\ip{\phi_*s-s}{A}} m(\pd \phi^*A) m(\pd A)^{-1} T^{\phi^*A} c(\phi^{-1}(x)),
\]
extended $R$-linearly.  By construction this is a chain map, respects the filtration, and induces \eqref{eqSSrep} on $\ssE{1}$, which is exactly what we need.
\end{proof}

\begin{rmk}\label{rmkPostHoc}
Instead of building rank $1$ local systems into the definition of the complex, one can view them as a post hoc modification of the map reducing $R[H_2^D]$ to $\Lambda$, under which the monomial $T^A$ is sent to $m(\pd A) T^{\mu(A)/N_L}$ instead of $T^{\mu(A)/N_L}$.  We may as well then take the local system to be trivial in \cref{labSSrep}.
\end{rmk}

\begin{rmk}
\Cref{labSSrep} can be extended to diffeomorphisms $\phi$ which are \emph{antisymplectic}, meaning that they pull back the symplectic form $\omega$ to $-\omega$.  In this case we have to replace $\phi^*\mathscr{D} = (\phi^*f, \phi^*g, \phi^*J)$ with $(\phi^*f, \phi^*g, -\phi^*J)$, since $\phi^*J$ is not compatible with $\omega$, and instead of $h^\phi$ simply carrying pearly trajectories by $\phi^{-1}$ it now has to replace each disc $u$ in the trajectory with $z \mapsto \phi^{-1} \circ u(\conj{z})$ in order to preserve holomorphicity.  This leads to two changes in \eqref{eqSSrep}.

Firstly, introducing the complex conjugation flips the signs of the homology classes of the discs, so $T^{\phi^*A}$ becomes $T^{-\phi^*A}$ and $m(\pd \phi^*A)$ becomes $m(\pd \phi^*A)^{-1}$.  And secondly it changes the orientations on the disc moduli spaces, so introduces a factor of $(-1)^{\mu(A)/2}$.  This was computed by Fukaya--Oh--Ohta--Ono in \cite[Theorem 1.3]{FOOOinv}.  Recall that if the Maslov indices of discs are not all even then $L$ is non-orientable and we must be working in characteristic $2$, so this factor can be ignored.  The resulting expression is
\[
a \otimes T^A \mapsto (-1)^{\mu(A)/2 + \ip{\phi_*s-s}{A}} m(\pd \phi^*A)^{-1} m(\pd A)^{-1} \phi^*(a) \otimes T^{-\phi^*A}.\qedhere
\]
\end{rmk}

%-----------------------------

\subsection{Warm-up example: the Clifford torus}

We now illustrate \cref{labSSrep} with a simple example.  In \cite{ChoCl} Cho computed the Floer cohomology of the monotone Clifford torus
\[
L = \{[z_0: \dots : z_n] : |z_j| = 1 \text{ for all } j\} \subset X = \C\P^n,
\]
equipped with an arbitrary spin structure $s$ and rank $1$ local system $\mathscr{F}$ over $\C$.  By \cref{SpinLocSys} we may assume $s$ is the \emph{standard} spin structure $s$, which for the present purpose is best viewed as the unique $\mathrm{Diff}(L)$-invariant spin structure, and summarise Cho's result as follows.

\begin{prop}
$L^\flat$ is wide if the monodromies of $\mathscr{F}$ around the basic loops
\[
\gamma_j(\theta) \coloneqq \{z_j = e^{i\theta} \text{, } z_k=1 \text{ for }k\neq j\}
\]
are all equal.
\end{prop}

His proof is by explicit calculation of the discs which contribute to the Floer differential.  We now show how to recover the result without any such analysis.

\begin{proof}
Let $\sigma$ denote the symplectomorphism of $\C\P^n$ which cycles the homogeneous coordinates.  This generates a subgroup of $\symp(X,L)$ isomorphic to $\Z/(n+1)$, and we need to show that if $\mathscr{F}$ is $\sigma$-invariant then $L^\flat$ is wide.  Assume then that $\mathscr{F}$ is indeed invariant.

By the proof of \cref{labBCWide} it suffices to show that the differential
\begin{equation}
\label{eqDiffToKill0}
\diff_1 : H^1(L; \C) \rightarrow T \cdot H^0(L; \C)
\end{equation}
in the spectral sequence over $\Lambda=\Lambda_\C$ is zero.  It is therefore enough to show that the image of the corresponding differential
\begin{equation}
\label{eqDiffToKill}
\diff_1 : H^1(L; \C) \rightarrow H^0(L; \C) \otimes_\C \C[H_2^D]_{\text{degree $2$}}
\end{equation}
in the spectral sequence over $\C[H_2^D]$ is annihilated by the reduction map $\pi: \C[H_2^D] \rightarrow \Lambda$ sending $T^A$ to $T^{\mu(A)/N_L}$.  Since both $s$ and $\mathscr{F}$ are invariant, \cref{labSSrep} tells us that \eqref{eqDiffToKill} is equivariant with respect to the obvious actions of $\sigma$ on $H^1(L; \C)$ and $\C[H_2^D]_{\text{degree $2$}}$.  The symmetrisation operator $\overline{\sigma} \coloneqq 1+\sigma+\dots+\sigma^n$ acts as zero on the former, so we will be done if the kernel of $\pi$ contains the kernel of $\overline{\sigma}$ on the latter.

Suppose then that $z = \sum_{A:\mu(A)=2} z_A T^A \in \C[H_2^D]_{\text{degree $2$}}$ is annihilated by $\overline{\sigma}$.  We need to show $\pi(z)=0$.  Since $\overline{\sigma}(z)=0$ we have
\[
\sum_A (z_A+z_{\sigma_*A}+\cdots+z_{\sigma^n_*A})T^A = 0
\]
so $z_A+z_{\sigma_*A}+\cdots+z_{\sigma^n_*A} = 0$ for each $A$.  We therefore have
\[
\pi(z) = \sum_A z_A T^{2/N_L} = \frac{1}{n+1} \sum_A (z_A+z_{\sigma_*A}+\cdots+z_{\sigma^n_*A})T^{2/N_L} = 0,
\]
giving the result.
\end{proof}

More generally, this argument shows the following.

\begin{prop}
\label{MonodromyGroup}
Suppose $L^\flat$ is a monotone Lagrangian torus equipped with the standard (meaning $\mathrm{Diff}(L)$-invariant) spin structure and a rank $1$ local system over a field $R$.  If $\phi_1, \dots, \phi_k$ are elements of $\symp(X, L)$ such that each $\phi_j$ preserves $\mathscr{F}$, and $\sum_j r_j\phi_j^*$ annihilates $H^1(L; R)$ for some $r_j$ in $R$ with non-zero sum, then $L^\flat$ is wide.\hfill$\qed$
\end{prop}

\begin{rmk}
If $L$ is narrow---for example if it's displaceable---then \cref{MonodromyGroup} gives restrictions on the group of linear automorphisms of $H^1(L; \Z)$ obtained as pullbacks by elements of $\symp(X, L)$.  The corresponding group for Hamiltonian diffeomorphisms has been studied by Mei-Lin Yau \cite{MLYau09,MLYau12} and Hu--Lalonde--Leclercq \cite{HuLalondeLeclercq} under the name \emph{Hamiltonian monodromy group}.
\end{rmk}

\Cref{MonodromyGroup} can be proved more directly as follows.  Let $D_1 \in H_1(L; R)$ denote the sum of the boundaries of the index $2$ discs through a generic point of $L$, weighted by the monodromy of $\mathscr{F}$.  If the $\phi_j$ and $r_j$ satisfy the hyoptheses of \cref{MonodromyGroup} then each $(\phi_j)_*$ preserves $D_1$ and the sum $\sum_j r_j (\phi_j)_*$ vanishes.  We thus have
\[
0 = \sum_j r_j (\phi_j)_* (D_1) = \sum_j r_j D_1,
\]
so $D_1 = 0$.  On the other hand, the differential \eqref{eqDiffToKill0} is given by pairing with $D_1$, so the differential vanishes and $L^\flat$ is wide.  The advantage of \cref{labSSrep} is that it applies even when there is no such concrete interpretation of the differentials, as is the case in main application below.

%---------------------------------

\section{The main family of examples}\label{secApp}

In this section we construct and study a family of monotone Lagrangians in products of projective spaces which provide an interesting testing ground for the tools developed in the preceding sections.

%------------------------------------

\subsection{Constructing the family}\label{sscMyPSU}

First we define the Lagrangians, so fix an integer $N \geq 3$ and consider $\C^{N-1}$ equipped with the standard symplectic form $\omega_0$.  The diagonal $S^1$-action is Hamiltonian with moment map $(\norm{z}^2-1)/2$, and reduction at the zero level set defines a symplectic structure on $\C\P^{N-2}$ which we will use throughout this section.  This corresponds to the Fubini--Study form normalised so that the area of a projective line is $\pi$.

The standard action of $\SU(N-1)$ on $\C^{N-1}$ descends to a $\PSU(N-1)$-action on $\C\P^{N-2}$, where it is Hamiltonian with moment map $\mu_{\C\P^{N-2}} : \C\P^{N-2} \rightarrow \mathfrak{psu}(N-1)^* = \mathfrak{su}(N-1)^*$ given by
\[
\ip{\mu_{\C\P^{N-2}}([z])}{A} = -\frac{i}{2}\frac{z^\dag A z}{z^\dag z}
\]
for all $z \in \C^{N-1}$, representing a point $[z]$ in $\C\P^{N-2}$, and all $A \in \mathfrak{su}(N-1)$.  Now take $X$ to be the product $(\C\P^{N-2})^N$, carrying the diagonal action of $K = \PSU(N-1)$, with moment map $\mu$ defined by
\begin{equation}
\label{eqMomMap}
\ip{\mu([z_1], \dots, [z_N])}{A} = -\frac{i}{2} \sum_{j=1}^N \frac{z_j^\dag A z_j}{z_j^\dag z_j}
\end{equation}
for all $([z_1], \dots, [z_N]) \in X$ and all $A$ as above.  Using the map $\mathfrak{u}(N-1) \rightarrow \mathfrak{su}(N-1)^*$ induced by the inner product $\ip{A}{B} = \operatorname{Tr} (A^\dag B)$ on $\mathfrak{u}(N-1)$, we can express $\mu$ as the map
\[
Z \in (\C\P^{N-2})^N \mapsto \frac{i}{2} Z Z^\dag \in \mathfrak{u}(N-1),
\]
where $Z$ is an $(N-1) \times N$ matrix whose columns are the homogeneous coordinates of the components of the corresponding point in $(\C\P^{N-2})^N$, scaled to have norm $1$ (so we can ignore the denominators in \eqref{eqMomMap}).  The phases of the columns of $Z$ are undetermined but do not affect the quantity $ZZ^\dag$.

Consider the vectors $v_1, \dots, v_N$ in $\C^{N-1}$ defined by
\[
v_j = (\zeta^j, \zeta^{2j}, \dots, \zeta^{(N-1)j}),
\]
where $\zeta = e^{2\pi i/N}$ is a primitive $N$th root of unity, and let $x$ be the point $([v_1], \dots, [v_N])$ in $X$.

\begin{prop}
The point $x$ lies in the zero level set of $\mu$ and its $K$ stabiliser is trivial.  Its orbit $L$ is a sharply $K$-homogeneous Lagrangian.
\end{prop}
\begin{proof}
One choice of matrix $Z$ representing $x$ has components $Z_{jk} = \zeta^{jk}/\sqrt{N-1}$, so
\[
(ZZ^\dag)_{jk} = \frac{1}{N-1} \sum_{l=1}^N \zeta^{(j-k)l} = \frac{N}{N-1} \delta_{jk}.
\]
Thus $iZZ^\dag/2$ is proportional to $i$ times the identity matrix, which spans the orthogonal to $\mathfrak{su}(N-1)$ inside $\mathfrak{u}(N)$, and hence $\mu(x)=0$.  The $K$-orbit of $x$ is therefore isotropic by \cite[Proposition 1.5]{Ch}.

Now suppose that $M$ is a matrix in $\SU(N-1)$ that stabilises $x$, which is equivalent to each $v_j$ being an eigenvector.  We need to show that $M$ is a scalar.  Our strategy is to show that the $v_j$ are pairwise non-orthogonal, so their eigenvalues must coincide, and that they span $\C^{N-1}$ so the corresponding eigenspace is the whole of $\C^{N-1}$.  We prove the two claims by direct computation: for the former we have for any distinct $j$ and $k$ that
\begin{equation}
\label{vjInnerProduct}
\ip{v_j}{v_k} = \sum_{l=1}^{N-1} \zeta^{(k-j)l} = -1+\sum_{l=1}^N \zeta^{(k-j)l} = -1
\end{equation}
(noting the upper limits on the sums); and for the latter we have for $k=1, \dots, N-1$ that
\[
\sum_{j=1}^N \zeta^{-jk}v_j = N \cdot (\text{$k$th standard basis vector}).
\]

The upshot is that $L$ is an isotropic free $K$-orbit, and we have
\[
\dim_\R L = \dim_\R \PSU(N-1) = N^2-1 = \dim_\C X
\]
so it is in fact Lagrangian, and thus a sharply $K$-homogeneous Lagrangian.
\end{proof}

Let $H_j \in H^2(X; \Z)$ denote the pullback of the hyperplane class under the projection of $X$ to the $j$th $\C\P^{N-2}$ factor.

\begin{lem}
\label{PSUmonotone}
$X$ has minimal Chern number $N-1$ and $L$ is monotone.
\end{lem}
\begin{proof}
We have $c_1(X)=(N-1)(H_1+\dots+H_N)$, which proves the first claim.  We also have $[\omega]=\pi(H_1+\dots+H_N)=\pi c_1(X)/(N-1)$ in $H^2(X; \R)$.  Since $L$ has fundamental group $\Z/(N-1)$, the group $H_1(L; \R)$ vanishes, and the long exact sequence of the pair in real homology then shows that the Maslov index homomorphism is $2(N-1)/\pi$ times area.  Hence $L$ is monotone.
\end{proof}

This means we can equip $L$ with a relative spin structure and local system to give a monotone Lagrangian brane $L^\flat$, to which we can apply the Floer theory outlined in \cref{secFloerBg}.

%------------------------------------

\subsection{Invariant subvarieties and axial discs}\label{sscCOConstr}

Our next task is to identify a collection of $K$-invariant subvarieties, compute their Poincar\'e duals, and classify their axial discs with a view towards applying \cref{thmCO}.  We will then use these constraints to prove the vanishing of $HF^*$ claimed in \cref{mthmHF}.

For a proper subset $I \subset \{1, \dots, N\}$ of size at least $2$ let $Z_I \subset X$ denote the subvariety
\[
Z_I = \Big\{\lb[z_1], \dots, [z_N]\rb \in X : \lb z_j\rb_{j \in I} \text{ is linearly dependent}\Big\},
\]
of complex codimension $N-|I|$.  This is invariant under the complexification $G=\PSL(n, \C)$ of $K$, and contains
\begin{multline*}
\mathcal{O}_I = \Big\{\lb[z_1], \dots, [z_N]\rb \in X : \lb z_j\rb_{j \in I} \text{ is linearly dependent but all proper subsets are} \\ \text{linearly independent, and } z_1, \dots, z_N \text{ span } \C^{N-1}\Big\}
\end{multline*}
as a dense open $G$-orbit.  By permuting the $\C\P^{N-2}$ factors, it suffices to understand $Z_I$ when $I$ is of the form $\{1, 2, \dots, N-k\}$ for $k=1, \dots, N-2$, and we denote the corresponding $Z_I$ (respectively $\mathcal{O}_I$) by $Z_k$ (respectively $\mathcal{O}_k$).  Recalling that $H_j$ is the hyperplane class on the $j$th factor we compute the following.

\begin{lem}\label{lemPDClass} For each $k = 1, \dots, N-2$ we have
\[
\PD(Z_k) = h_k(H_1, \dots, H_{N-k}),
\]
where $h_k$ denotes the complete homogeneous symmetric polynomial of degree $k$, i.e.
\[
\sum_{\substack{r_1,  \dots, r_{N-k} \geq 0 \\ r_1+\dots+r_{N-k} = k}} H_1^{r_1}\cdots H_{N-k}^{r_{N-k}}.
\]
\end{lem}
\begin{proof}  Since $\PD(Z_k)$ lies in $H^{2k}(X; \Z)$, it is a linear combination of monomials $H_1^{r_1} \cdots H_N^{r_N}$ with total degree $r_1+\dots+r_N$ equal to $k$.  The coefficient of such a monomial is obtained by counting intersection points of $Z_k$ with
\[
\prod_{j=1}^N \Pi_j \subset \prod_{j=1}^N \C\P^{N-2} = X,
\]
where $\Pi_j$ is a generic $r_j$-plane in $\C\P^{N-2}$.  All such intersections count positively since all cycles involved carry complex orientations.  We therefore wish to count $N$-tuples of points in $\C\P^{N-2}$ such that the $j$th point is constrained to a generic $r_j$-plane and the first $N-k$ points are linearly dependent, where $\sum r_j = k$.

Suppose first that $r_1+\dots+r_{N-k} < k$.  In this case we claim that generically there are no linear dependencies between the first $N-k$ points, so the coefficient of $H_1^{r_1}\cdots H_N^{r_N}$ is zero.  Equivalently, we claim that generically
\begin{equation}
\label{eqPiDep}
\Pi_j \cap \lb \Pi_1 + \dots + \widehat{\Pi}_j + \dots + \Pi_{N-k} \rb = \emptyset
\end{equation}
for $j=1, \dots, N-k$, where $\hat{}$ denotes omission as usual (elements of this intersection correspond to linear dependencies $a_1z_1+\dots+a_{N-k}z_{N-k}$ between $[z_1]\in \Pi_1, \dots, [z_{N-k}]\in\Pi_{N-k}$ with $a_j \neq 0$).  To prove this, note that the bracketed term on the left-hand side is an $R_j$-plane for some $R_j$, where
\[
R_j \leq r_1+\dots+\hat{r}_j+\dots+r_{N-k}+N-k-2 < N-2-r_j,
\]
so it is disjoint from a generic $r_j$-plane in $\C\P^{N-2}$.

We are left to deal with the case $r_1+\dots+r_{N-k}=k$ and $r_{N-k+1}=\dots=r_N=0$.  In this situation each intersection \eqref{eqPiDep} is non-empty (it is an intersection of planes of complementary dimensions), and generically it consists of a single point.  We conclude that generically there is a single point in $\prod_j \Pi_j$ such that the first $N-k$ entries are linearly dependent (and moreover---although this is irrelevant for us---the dependency involves all $N-k$ of these entries, i.e.~the $a_j$ above are all non-zero).  Thus the coefficient of $H_1^{r_1}\dots H_N^{r_N}$ is $1$.  This gives the claimed expression for $\PD(Z_k)$.
\end{proof}

\Cref{thmCO} expresses $\CO^0(\PD(Z_k))$ in terms of axial discs $u$ of index $2k$ with $u(0)$ in $Z_k$ and $u(1)$ equal to the unique Morse minimum.  Note that by \cref{DiscClassification} any such disc must necessarily satisfy $u(0) \in \mathcal{O}_k$; otherwise it meets the invariant subvariety $Z_k \setminus \mathcal{O}_k$, of complex codimension $>k$, so must have index $>2k$.  The choice of Morse minimum is irrelevant, so we take it to be the point $x=([v_1], \dots, [v_N])$ from \cref{sscMyPSU}.  The next result classifies these discs.

\begin{prop}\label{lemUniqDisc}
For each $k$ in $\{1, \dots, N-2\}$ there is a unique axial disc $u$ of index $2k$ with $u(0) \in \mathcal{O}_k$ and $u(1)=x$.
\end{prop}
\begin{proof}
Axial discs $u$ with $u(1)=x$ are precisely maps of the form $z \mapsto e^{-i\xi \log z}x$ where $\xi$ is an element of $\mathfrak{su}(N-1)$ satisfying $e^{2\pi\xi}x=x$.  Any $\xi \in \mathfrak{su}(N-1)$ is diagonalisable, with orthogonal eigenspaces and imaginary eigenvalues (which sum to zero when counted with multiplicity, since $\xi$ is trace-free), and our plan is to use the conditions $e^{2\pi\xi}x=x$, $u(0) \in \mathcal{O}_k$, and $\mu(u)=2k$ to determine what the eigenspaces and eigenvalues must be.

Since $x$ has trivial stabiliser in $\PSU(N-1)$, the condition $e^{2\pi\xi}x=x$ forces $e^{2\pi\xi}$ to be the identity in $\PSU(N-1)$, i.e.~of the form $e^{2\pi i \delta/(N-1)}$ times the identity matrix for some $\delta \in \Z$.  Each eigenvalue $i\lambda$ of $\xi$ must therefore satisfy $\lambda = \delta/(N-1) \mod \Z$, so we write the distinct eigenvalues as
\begin{equation}
\label{eqEvalForm}
i\Big(\lambda_j + \frac{\delta}{N-1}\Big)
\end{equation}
for $\lambda_j \in \Z$.  Reordering if necessary, we may assume that $\lambda_1 < \dots < \lambda_m$, and by adding multiples of $\pm(N-1)$ to $\delta$ and $\mp 1$ to each $\lambda_j$ we may assume that $\delta$ lies in $\{-k, \dots, N-2-k\}$.  Letting $V_1, \dots, V_m$ be the corresponding eigenspaces, the trace-free condition gives
\begin{equation}
\label{eqTraceFree}
\delta + \sum_{j=1}^m \lambda_j \dim V_j = 0.
\end{equation}

Now let $\pi_j$ be orthogonal projection onto $V_j$.  For any non-zero vector $v$ in $\C^{N-1}$ we have
\[
[e^{-i\xi\log z}v] = \Big[\sum_{j=1}^m z^{\lambda_j}\pi_j(v)\Big],
\]
so as $z \rightarrow 0$ the left-hand side converges to $[\pi_{j_\mathrm{min}(v)}(v)]$, where $j_\mathrm{min}(v)$ is defined to be the minimal $j$ such that $\pi_j(v)$ is non-zero.  Using this we can write $u(0)$ as $([v^\pi_1], \dots, [v^\pi_N)])$, where $v^\pi_j = \pi_{j_\mathrm{min}(v_j)}(v_j)$.  The condition $u(0) \in \mathcal{O}_k$ then tells us that $v^\pi_1, \dots, v^\pi_{N-k}$ are linearly dependent but there are no other dependencies between $v^\pi_1, \dots, v^\pi_N$.

Next recall from the proof of \cref{PSUmonotone} that the Maslov index of a disc is $2(N-1)/\pi$ times its area.  The computation in the proof of \cref{labHamPatch} expresses the area of our axial disc in terms of the moment map and we obtain
\[
\mu(u) = \frac{2(N-1)}{\pi} \cdot -2\pi \ip{\mu(u(0))}{\xi},
\]
where the $\mu$ on the left-hand side is the Maslov index and that that on the right is the moment map.  This reduces to
\[
\mu(u) = 2i(N-1)\sum_{j=1}^N \frac{(v^\pi_j)^\dag\xi v^\pi_j}{(v^\pi_j)^\dag v^\pi_j},
\]
and the $j$th summand on the right-hand side is exactly the eigenvalue $i(\lambda_{j_\mathrm{min}(v_j)}+\delta/(N-1))$.  The condition $\mu(u)=2k$ then gives
\begin{equation}
\label{eqIndex}
k = -N\delta - (N-1)\sum_{j=1}^N \lambda_{j_\mathrm{min}(v_j)}.
\end{equation}

We have thus reduced to the following situation: the eigenvalues are of the form \eqref{eqEvalForm} and satisfy the constraints \eqref{eqTraceFree} and \eqref{eqIndex}, and the unique linear dependency (up to scaling) between the $v^\pi_1, \dots, v^\pi_N$ is between the first $N-k$.

We need to analyse the latter more closely, so for each $r$ let
\[
J_r = \{j : j_\mathrm{min} (v_j) = r\}
\]
index those $v_j$ with $j_\mathrm{min}=r$.  Linear dependencies between the $v^\pi_j$ decompose over the sets $J_r$, in the sense that
\[
\sum_{j=1}^N a_j v^\pi_j = 0 \quad \text{if and only if} \quad  \sum_{j \in J_r} a_j v^\pi_j = 0 \quad \text{for all } r,
\]
so there must be a unique value of $r$---call it $r_0$---such that $(v^\pi_j)_{j\in J_r}$ is linearly dependent, and this $J_{r_0}$ must contain $1, \dots, N-k$.

We claim that $r_0=1$.  To see this, note that we can express $J_1$ as
\[
J_1 = \{j : v_j \notin V_2 \oplus \dots \oplus V_m\}.
\]
Any proper subset of the $v_j$ is linearly independent, so in particular this holds for the set of $v_j$ contained in $V_2 \oplus \dots \oplus V_m$, and we deduce that
\[
|\{j : v_j \in V_2 \oplus \dots \oplus V_m\}| \leq \dim V_2 + \dots + \dim V_m
\]
and hence
\[
|J_1| \geq N - (\dim V_2 + \dots + \dim V_m) = \dim V_1 + 1.
\]
The projection of the $(v_j)_{j \in J_1}$ to $V_1$ therefore produces a linear dependence between the $(v^\pi_j)_{j \in J_1}$, so we must have $r_0=1$ as claimed.

Since there are no other linear dependencies, we must in fact have $|J_1|=\dim V_1 +1$ and $|J_r| \leq \dim V_r$ for all $r>1$.  Summing over $r$ gives
\[
N = \sum_{r=1}^m |J_r| \leq 1 + \sum_{r=1}^m \dim V_r = N,
\]
so the only possibility is that $|J_r| = \dim V_r$ for $r=2, \dots, m$.  This forces us to have $V_m = \lspan{v_j}_{j \in J_m}$, and more generally
\[
V_r \oplus \dots \oplus V_m = \lspan{v_j}_{j \in J_r \cup \dots \cup J_m},
\]
so the kernel of the orthogonal projection of $\lspan{v_j}_{j \in J_1}$ onto $V_1$ is
\begin{equation}
\label{eqKernel}
\lspan{v_j}_{j \in J_1} \cap \lspan{v_j}_{j \in J_2 \cup \dots \cup J_m}.
\end{equation}
We want this (one-dimensional) kernel to lie in $\lspan{v_1, \dots, v_{N-k}}$, to produce the linear dependence between $v^\pi_1, \dots, v^\pi_{N-k}$.  But since the unique linear dependence between the $v_j$ is $v_1 + \dots + v_N = 0$, we can compute \eqref{eqKernel} exactly: it is the span of $\sum_{j \in J_1} v_j$.  We must therefore have $J_1 = \{1, \dots, N-k\}$.

The upshot of the previous three paragraphs is that $J_1 = \{1, \dots, N-k\}$, and each $V_r$ is given by the orthogonal complement of $\lspan{v_j}_{j \in J_{r+1} \cup \dots \cup J_m}$ in $\lspan{v_j}_{j \in J_r \cup \dots \cup J_m}$.  We now need to combine this information with \eqref{eqEvalForm}, \eqref{eqTraceFree}, and \eqref{eqIndex}, to determine $\delta$, the values $\lambda_1, \dots, \lambda_m$, and the sets $J_2, \dots, J_m$.

Well, we can rewrite \eqref{eqIndex} as
\[
k = -N\delta - (N-1)\sum_{r=1}^m |J_r|\lambda_r = -N\delta - (N-1) \lambda_1 - (N-1)\sum_{r=1}^m \lambda_r\dim V_r,
\]
and eliminate the final term using \eqref{eqTraceFree} to give
\[
k + \delta = -(N-1)\lambda_1.
\]
Since the left-hand side is in $\{0, \dots, N-2\}$ but the right-hand side is divisible by $N-1$, we must have $\delta = -k$ and $\lambda_1=0$.  By induction we then have $\lambda_r \geq r-1$ for each $r$, so returning to \eqref{eqTraceFree} we obtain
\[
k = \sum_{r=2}^m \lambda_r \dim V_r \geq \sum_{r=2} \dim V_r = k.
\]
We thus have equality in the middle inequality, which means that $\lambda_2=1$ and $m=2$ (there's no $V_3, V_4, \dots$).  This completely pins down $\delta$, the $\lambda_j$, and the $J_r$, and hence $\xi$ itself: $\xi$ acts as $i(1-k/(N-1))$ on $\lspan{v_{N-k+1}, \dots, v_N}$ and $-ik/(N-1)$ on the orthogonal complement.  Conversely, this $\xi$ has the required properties: it satisfies $e^{2\pi\xi}x =x$, and generates an axial disc $u$ with index $2k$ and $u(0) \in \mathcal{O}_k$.
\end{proof}

%--------------------------------

\subsection{Constraints from the closed--open map}
\label{sscCOconstraints}

We are almost ready to apply \cref{thmCO}, but first we need to understand the signs appearing in it, which depend on the choice of relative spin structure $s$ on $L$.  Recall from \cref{defRelSpin} that the set of relative spin structures forms a torsor for $H^2(X, L; \Z/2)$, so we can describe $s$ as the translate of the standard spin structure by some class $\eps \in H^2(X, L; \Z/2)$.  To compute the effect of $\eps$ we need to know the classes of the discs in \cref{lemUniqDisc}, and before doing this we need a convenient basis for $H_2(X, L; \Z)$.

To obtain such a basis note that the compactification divisor $Y \subset X$ is the union of the subvarieties
\begin{equation}
\label{eqDivCpts}
Z_{\{1, \dots, \widehat{j}, \dots, N\}}
\end{equation}
By \cref{lemUniqDisc} each of these subvarieties meets a unique axial index $2$ disc.  Let $A_j \in H_2(X, L; \Z)$ be the homology class of the disc meeting the $j$th.

\begin{lem}\label{DiscBasis}  The $A_j$ freely generate $H_2(X, L)$.
\end{lem}
\begin{proof}
All homology groups are over $\Z$.  The Poincar\'e duals of the subvarieties in \eqref{eqDivCpts} define classes in $H^2(X \setminus L)$, and by pairing with these we obtain a map $\theta : H_2(X, L) \rightarrow \Z^N$.  This gives an isomorphism
\begin{equation}
\label{eqKerTheta}
H_2(X, L) \cong \Z^N \oplus \ker \theta,
\end{equation}
with the $A_j$ forming a free basis for the $\Z^N$ summand, and we are done if we can show that $\ker \theta = 0$.

To prove that this is indeed the case, consider the normalised Maslov index homomorphism $\nu = \mu/2 : H_2(X, L) \rightarrow \Z$, which yields an isomorphism
\begin{equation}
\label{HXLsplit}
H_2(X, L) \cong \Z \oplus \ker \nu.
\end{equation}
Since the minimal Chern number is $N-1$, the image of $H_2(X)$ in $H_2(X, L)$ is annihilated by the mod $N-1$ reduction $\overline{\nu}$ of $\nu$.  This means that $\overline{\nu}$ factors as the boundary map $\pd : H_2(X, L) \rightarrow H_1(L)$ followed by a map $\hat{\nu} : H_1(L) \rightarrow \Z/(N-1)$, and because $\nu$ is surjective $\hat{\nu}$ is also surjective and is thus an isomorphism.  In particular, $\hat{\nu}$ is injective, so $\ker \overline{\nu}$ coincides with $\ker \pd$, and hence $\ker \nu$ is contained in the image of $H_2(X)$.  We deduce that $\ker \nu$ is the image of
\begin{equation}
\label{kerc1}
\ker \lb \ip{c_1(X)}{\cdot} : H_2(X) \rightarrow \Z \rb \subset H_2(X)
\end{equation}
in $H_2(X, L)$.  Since \eqref{kerc1} is isomorphic to $\Z^{N-1}$, $\ker \nu$ must be a quotient of $\Z^{N-1}$, so \eqref{HXLsplit} tells us that $H_2(X, L)$ is a quotient of $\Z^N$.  Then \eqref{eqKerTheta} forces $\ker \theta$ to vanish, proving the result.
\end{proof}

We can now compute the classes of the discs appearing in \cref{lemUniqDisc}.

\begin{lem}
\label{DiscClass}
The unique axial disc $u$ of index $2k$, satisfying $u(0) \in \mathcal{O}_k$ and $u(1)=x$, is in class
\[
A_{N-k+1} + \dots + A_N.
\]
\end{lem}
\begin{proof}
We know from \cref{DiscBasis} and its proof that we have
\[
[u] = \sum_{j=1}^N \lambda_j A_j,
\]
where $\lambda_j$ is the intersection number of $u$ with \eqref{eqDivCpts}, and since everything is complex all intersections count positively.  The subvariety $Z_k$ lies in the intersection
\[
\bigcap_{j=N-k+1}^N Z_{\{1, \dots, \widehat{j}, \dots, N\}}
\]
so the fact that $u$ meets $Z_k$ forces $\lambda_{N-k+1}, \dots, \lambda_N$ to be strictly positive.  Since all other $\lambda_j$ are non-negative (by positivity of intersections again), we obtain
\[
\sum_j \lambda_j \geq k
\]
with equality if and only if $[u]$ is as claimed.  But the left-hand side is exactly half the Maslov index of $u$ (since each $A_j$ has index $2$), so we do have equality.
\end{proof}

The final thing we need to do before applying \cref{thmCO} is to introduce some notation.  Let $\eps_j \in \{\pm1\}$ denote the sign $(-1)^{\ip{\eps}{A_j}}$, recalling that $\eps \in H^2(X, L; \Z/2)$ is the class describing the relative spin structure.  From the proof of \cref{DiscBasis}, the boundary map $H_2(X, L; \Z) \rightarrow H_1(L; \Z) \cong \Z/(N-1)$ is the mod $N-1$ reduction of $\mu/2$ so in particular the classes $A_j$ all have the same boundary.  Let $M$ denote the fibre endomorphism of the local system $\mathscr{F}$ which describes the monodromy around this loop.  We do not need to specify the base point since $L$ has abelian fundamental group.  We can now give the main result of this subsection.

\begin{thm}
\label{CalcCO}
For a proper subset $I \subset \{1, \dots, N\}$ of size at least $2$ we have
\begin{equation}
\label{eqCalcCO}
\CO^0(h_{N-|I|}\big((H_j)_{j\in I}\big) = \Big(\prod_{j \notin I} \eps_j\Big) T^{N-|I|}  \cdot [M^{N-|I|} \otimes 1_L] \in HF^*(L^\flat, L^\flat; \Lambda).
\end{equation}
\end{thm}
\begin{proof}
By permuting the $\C\P^{N-2}$ factors we may assume that $I = \{1, \dots, N-k\}$, so the left-hand side becomes $\CO^0(h_k(H_1, \dots, H_{N-k}))$.  Now simply apply \cref{thmCO} to $Z_k$ and plug in \cref{lemPDClass} and \cref{lemUniqDisc}.  The sign comes from \cref{SpinChangeProp} combined with \cref{DiscClass}
\end{proof}

%----------------------------------

\subsection{Analysing the constraints}

In this subsection we manipulate the equalities from \cref{CalcCO} to prove the $HF^*$ vanishing results of \cref{mthmHF}.  To this end, let $\eta_j$ denote $\CO^0(H_j)$ in $HF^*(L^\flat, L^\flat; R)$, let $\bm{\eta}$ denote the $N$-tuple of all $\eta_j$, and for a subset $I \subset \{1, \dots, N\}$ let $\bm{\eta}_I$ denote the tuple $(\eta_j)_{j \in I}$, so that
\[
h_k(\bm{\eta}_I) = \sum_{\substack{j_1, \dots, j_k \in I \\ j_1 \leq \dots \leq j_k}} \eta_{j_1} \dots \eta_{j_k}
\]
for each $k$.  Note that the $\eta_j$ are $\CO^0$-images of the $H_j$, which commute, and hence the $\eta_j$ themselves commute.  As a simple corollary of \cref{CalcCO} we obtain

\begin{prop}\label{COconstr1} For any non-empty subset $I \subset \{1, \dots, N\}$ we have
\[
h_{N-|I|} (\bm{\eta}_I) = \Big(\prod_{j \notin I} \eps_j\Big) T^{N-|I|}  \cdot [M^{N-|I|} \otimes 1_L]
\]
in $HF^*(L^\flat, L^\flat; R)$.  The classes $\eta_j + \eps_jT\cdot [M\otimes 1_L]$ are all equal to $\eta_1 + \dots + \eta_N$.  We denote this common value by $\conj{\eta}$.
\end{prop}
\begin{proof}
The polynomial appearing on the left-hand side of \eqref{eqCalcCO}, expressing the cohomology class Poincar\'e dual to $Z_I$, was computed using the classical cup product on $H^*(X; R)$, but since its degree is less than the minimal Chern number we can equally interpret it in terms of the quantum product on $QH^*(X; \Lambda)$.  Using the fact that $\CO^0$ is a ring map with respect to the quantum product, \cref{CalcCO} immediately gives the claimed equality when $2 \leq |I| \leq N-1$.  Taking $|I|=N-1$ we obtain the statement about the $\eta_j + \eps_jT\cdot[M\otimes 1_L]$.

We are left to deal with the cases $|I|=N$ and $|I|=1$.  The former is a consequence of unitality of $\CO^0$, so restrict attention to the latter, say $I=\{j\}$.  We then have
\[
h_{N-|I|}(\bm{\eta}_I) = \eta_j^{N-1} = \CO^0(H_j^{N-1}),
\]
and we can compute $H_j^{N-1}$ using the relations in $QH^*(X; \Lambda)$.  If the background class (image of $\eps$ in $H^2(X; \Z/2)$) vanishes then we have the familiar relation from the quantum cohomology of $\C\P^{N-2}$, namely $H_j^{N-1} = T^{N-1} \cdot 1_X$.  In general however, we must twist this by $(-1)^{\ip{\eps}{A}}$, where $A$ is the homology class of the rational curve which contributes to this quantum product.  This class is exactly the class of a line on the $j$th $\C\P^{N-2}$ factor, and by arguing as in \cref{DiscClass} this is given by
\[
A_1 + \dots + \widehat{A_j} + \dots + A_N
\]
in terms of our basis for $H_2(X, L; \Z)$.  In the presence of the class $\eps$, the quantum cohomology relation thus becomes
\[
H_j^{N-1} = \Big( \prod_{k \neq j} \eps_k \Big)T^{N-1} \cdot 1_X.
\]
Using unitality of $\CO^0$ again, and the fact that $M^{N-1}$ is the identity (it's the monodromy around a contractible loop), we get
\[
h_{N-|I|}(\bm{\eta}_I) = \Big( \prod_{k \neq j} \eps_k \Big)T^{N-1}\cdot [M^{N-1} \otimes 1_L],
\]
which is what we want.
\end{proof}

We need to transform these equalities into a more usable form, via the yoga of symmetric polynomials.   First abbreviate $\eps_jT \cdot [M \otimes 1_L]$ to $t_j$ and observe that the $t_j$ commute with each other and (by writing $t_j$ as $\eta_1+\dots+\widehat{\eta}_j+\dots+\eta_N$) with the $\eta_k$.  The algebraic transformation is then given by

\begin{lem}\label{yoga}
Treating $\eta_j$, $\overline{\eta}$ and $t_j$ as commuting formal variables over $\Z$, subject only to the relations $\eta_j+t_j=\overline{\eta}$ for all $j$, the following are equivalent:
\begen
\item\label{constr1} $h_{N-|I|}(\bm{\eta}_I)=\prod_{j \notin I} t_j$ for all non-empty subsets $I \subset \{1, \dots, N\}$.
\item\label{constr2} $h_k(\bm{\eta})=\overline{\eta}^k$ for $k=1, \dots, N-1$.
\item\label{constr3} $e_1(\bm{\eta})=\overline{\eta}$, and $e_k(\bm{\eta})=0$ for $k=2, \dots, N-1$, where the $e_j$ are the elementary symmetric polynomials.
\end{enumerate}
\end{lem}
\begin{proof}
For any $k$ in $\{1, \dots, N-1\}$ and any subset $I \subset \{1, \dots, N\}$ of size $N-k$ we have
\begin{equation}
\label{eqYoga}
h_k(\bm{\eta}) = h_k(\bm{\eta}_I) + \sum_{j_1 \notin I} \eta_{j_1} h_{k-1}(\bm{\eta}_{I \cup \{j_1\}}) + \sum_{\substack{j_1, j_2 \notin I \\ j_1 < j_2}} \eta_{j_1}\eta_{j_2} h_{k-2}(\bm{\eta}_{I \cup \{j_1, j_2\}}) + \dots.
\end{equation}
Here the first term on the right-hand side contains precisely the monomials from the left-hand side which involve only those $\eta_j$ with $j \in I$, the second term (the sum over $j_1$) contains those monomials involving exactly one $\eta_j$ from outside this set, the third term those monomials involving exactly two from outside the set, and so on.

Assuming \ref{constr1} holds, \eqref{eqYoga} gives
\begin{align*}
h_k(\bm{\eta}) &= \prod_{j \notin I} t_j + \sum_{j_1 \notin I} \eta_{j_1} \prod_{j \notin I \cup \{j_1\}}t_j + \sum_{\substack{j_1, j_2 \notin I \\ j_1 < j_2}} \eta_{j_1}\eta_{j_2} \prod_{j \notin I \cup \{j_1, j_2\}}t_j + \dots
\\ &= \prod_{j \notin I} (t_j+\eta_j).
\end{align*}
The right-hand side is nothing but $\overline{\eta}^k$, so \ref{constr2} holds.  Coversely, suppose \ref{constr2} holds, and that \ref{constr1} holds for all subsets of size strictly greater than $k$ (the base case---where the subset is all of $\{1, \dots, N\}$---is trivial).  From \eqref{eqYoga} we then get
\[
\overline{\eta}^k = h_{N-|I|}(\bm{\eta}_I) + \sum_{j_1 \notin I} \eta_{j_1} \prod_{j \notin I \cup \{j_1\}}t_j + \sum_{\substack{j_1, j_2 \notin I \\ j_1 < j_2}} \eta_{j_1}\eta_{j_2} \prod_{j \notin I \cup \{j_1, j_2\}}t_j + \dots.
\]
Expanding the left-hand side as $\prod_{j \notin I} (t_j+\eta_j)$, we see that $h_{N-|I|}(\bm{\eta}_I) = \prod_{j \notin I} t_j$, so by downward induction on $|I|$ \ref{constr1} holds.  Thus \ref{constr1} and \ref{constr2} are equivalent.

Now use the fact that
\[
\sum_{r \geq 0} h_r(\bm{\eta})z^r = \prod_j \frac{1}{1-\eta_jz} = \frac{1}{\sum_{r \geq 0} e_r(\bm{\eta})z^r}
\]
as power series in $z$.  Both \ref{constr2} and \ref{constr3} are equivalent to this quantity being $1/(1-\overline{\eta}z) + O(z^N)$.
\end{proof}

Applying the equivalence of \ref{constr1} and \ref{constr3} to \cref{CalcCO} we obtain

\begin{cor}\label{CorCO}
The classes $\eta_j$, $t_j$ and $\overline{\eta}$ in $HF^*$ satisfy $\eta_j+t_j=\overline{\eta}$ for all $j$ and
\begin{flalign*}
&& e_k(\bm{\eta}) &= \begin{cases} \overline{\eta} & \text{if } k=1 \\ 0 & \text{if } k = 2, \dots, N-1.\end{cases} & \qed
\end{flalign*}
\end{cor}

\begin{rmk}
\label{rmkCharCorn}
The closed--open map for Lagrangians invariant under a loop $\gamma$ of Hamiltonian diffeomorphisms has been studied by Charette--Cornea \cite{CharCorn} and more recently by Tonkonog \cite{TonkS1}.  If $S(\gamma)$ denotes the Seidel element in $QH^*(X)$ defined by $\gamma$, they showed that after setting the variable $T$ to $1$, and with the trivial local system,  $\CO^0(S(\gamma))$ is equal to $\pm 1_L$, where the sign depends on the choice of spin structure.  We now illustrate this for our family, and verify its consistency with the above computations.

Let $\gamma$ be the action of the loop $(g_\theta)$ in $\PSU(N-1)$ given by the diagonal matrices with diagonal entries $e^{-i\theta/(N-1)}(e^{i\theta}, 1, \dots, 1)$.  Taking the zero background class in $H^2(X; \Z/2)$, and applying a result of McDuff--Tolman \cite[Theorem 1.10]{McDuffTolman} to each $\C\P^{N-2}$ factor, we obtain $S(\gamma) = H_1 \dots H_N$.  \cite[Theorem 1.7]{TonkS1} then yields $e_N(\bm{\eta}) = \pm 1_L$.  On the other hand, by \cref{CorCO} we have
\[
\prod_{j=1}^N t_j = \prod_{j=1}^N (\overline{\eta} - \eta_j) = \sum_{j=0}^N (-1)^j e_j(\bm{\eta}) \overline{\eta}^{N-j} = (-1)^N e_N(\bm{\eta}).
\]
Taking the trivial local system and standard spin structure, and setting $T$ to $1$, which corresponds to replacing each $t_j$ by $1_L$, we obtain $e_N(\bm{\eta}) = (-1)^N \cdot 1_L$.  It remains to check that the sign from \cite{CharCorn,TonkS1} is also $(-1)^N$.

Let $\overline{\gamma}$ denote an orbit of $\gamma$ on $L$.  A choice of spin structure on $L$ induces a homotopy class of trivialisation of $\overline{\gamma}^*TL$, and by \cite[Theorem 1.7]{TonkS1} the sign of $\CO^0(S(\gamma))$ is $+1$ if and only if this homotopy class agrees with that defined by transport by $\gamma$.  In our case, the latter corresponds to the identification
\[
T_{\overline{\gamma}(\theta)}L = (g_\theta)_* T_{\overline{\gamma}(0)}L = g_\theta \mathfrak{su}(N-1)\cdot \overline{\gamma}(0)
\]
for all $\theta$.  The standard spin structure, meanwhile, uses the identification
\[
T_{\overline{\gamma}(\theta)}L = \mathfrak{su}(N-1)\cdot \overline{\gamma}(\theta) = \mathfrak{su}(N-1)g_\theta\cdot \overline{\gamma}(0).
\]
The two trivialisations therefore differ by the loop $\widetilde{\gamma}$ in $\SL(\mathfrak{su}(N-1))$ defined by the adjoint action of $(g_\theta)$.

Under this action $\mathfrak{su}(N-1)$ decomposes as $N-2$ two-dimensional subspaces which $g_\theta$ rotates by angle $\theta$, and a complement which is fixed.  Thinking of $\mathfrak{su}(N-1)$ as $(N-1)\times(N-1)$ skew-Hermitian matrices, the rotated subspaces correspond to the entries along the top row (or equivalently the first column) apart from the first one.  The loop $\widetilde{\gamma}$ therefore represents $N-2$ times the generator of $\pi_1 (\SL(\mathfrak{su}(N-1))) \cong \Z/2$, so the two homotopy classes of trivialisation coincide if and only if $N$ is even: hence the factor of $(-1)^N$.
\end{rmk}

%\begin{rmk}
%Applying the same ideas of Charette--Cornea, Tonkonog and McDuff--Tolman to circle subgroups of the torus acting on a Fano toric variety, one sees that $\CO^0$ maps each component of the toric divisor to $\pm 1_L$ in the self-Floer cohomology of the monotone fibre.  To compute the sign in the standard spin structure one again has to consider the adjoint action of the circle subgroup, but since the torus is abelian this action is trivial.  This means the sign is always $+1$, in agreement with \cref{labCOCalc}.
%\end{rmk}

We now reach the punchline:

\begin{thm}\label{PrimePowProp}
Assume that the coefficient ring $R$ contains a field $\mathbb{K}$, and let $p$ denote its characteristic (prime or $0$).  Then $L^\flat$ is narrow, i.e.~$HF^*(L^\flat, L^\flat; \Lambda)=0$, unless:
\begen
\item\label{primeitm1} $p$ is prime, $N$ is a power of $p$, and $s$ has signature $(N, 0)$ or $(0, N)$ if $p \neq 2$.
\item\label{primeitm2} $p$ is prime, $N$ is twice a power of $p$, and $s$ has signature $(N/2, N/2)$ if $p \neq 2$.
\item\label{primeitm3} $p=5$, $N=3$, and $s$ has signature $(2, 1)$ or $(1, 2)$.
\end{enumerate}
Here the signature of the relative spin structure $s$ is defined to be
\[
(\text{number of $\eps_j$ equal to $+1$}, \text{number of $\eps_j$ equal to $-1$}).
\]
Note that the relative spin structure, and hence signature, is irrelevant in characteristic $2$.
\end{thm}
\begin{proof}
Assume $HF^* \neq 0$, and let $s$ have signature $(a, b)$ where $a+b=N$.  By reordering the $\C\P^{N-2}$ factors, we may assume that $\eps_1 = \dots = \eps_a = 1$ and $\eps_{a+1} = \dots = \eps_N = -1$.  We write $T\cdot[M \otimes 1]$ as $t$, so that each $t_j$ is given by $\eps_j t$.  Recall that the $\eta_j$ and $t_j$ all commute with each other.  The $e_1$ condition of \cref{CorCO} gives
\begin{equation}
\label{eqe1}
a(\overline{\eta}-t)+b(\overline{\eta}+t)=\overline{\eta},
\end{equation}
whilst the remaining conditions give
\begin{equation}
\label{eqks}
\sum_{j=0}^k \binom{a}{j}\binom{b}{k-j}(\overline{\eta}-t)^j(\overline{\eta}+t)^{k-j} = 0.
\end{equation}
for $k=2, \dots, N-1$.  These are all equalities in $HF^*$.

Suppose first that $N \geq 4$ and $p \neq 2$.  The equations $a+b = N$ and \eqref{eqe1} give
\begin{equation}
\label{eqab}
at = \frac{Nt+(N-1)\overline{\eta}}{2} \quad \text{and} \quad bt = \frac{Nt-(N-1)\overline{\eta}}{2}.
\end{equation}
Substituting into \eqref{eqks} with $k=2$ and $k=3$ (using the fact that $N \geq 4$), and simplifying, we get
\begin{equation}
\label{eqk1}
\overline{\eta}^2(N-1) = Nt^2
\end{equation}
and
\begin{equation}
\label{eqk2}
\overline{\eta}\lb2t^2-\overline{\eta}^2(N-1)+Nt^2\rb = 0.
\end{equation}
Now substituting \eqref{eqk1} into \eqref{eqk2}, and using the fact that $t$ is invertible (its inverse is $T^{-1}\cdot[M^{-1} \otimes 1_L]$), we conclude that $\overline{\eta}=0$.  Plugging this back into \eqref{eqk1} and \eqref{eqab} we see that $N\cdot 1_L$, $a \cdot 1_L$ and $b \cdot 1_L$ vanish in $HF^*$.  The $\mathbb{K}$-subalgebra of $HF^*$ generated by the unit is isomorphic to $\mathbb{K}$ itself, so we deduce that $N \equiv a \equiv b \equiv 0$, where $\equiv$ denotes equality in $\mathbb{K}$.  In particular, $p$ must be prime, not $0$.  We can then express the equations \eqref{eqks} as
\begin{equation}
\label{eqks2}
\sum_{j=0}^k \binom{a}{j}\binom{b}{k-j}(-1)^j \equiv 0
\end{equation}
for $k$ in $\{2, \dots, N-1\}$.

Next assume that $a$ and $b$ are both positive.  We claim that $a$ and $b$ are both equal to the same power of $p$, so we are in case\ref{primeitm2}.  Write $a=p^\alpha A$ and $b=p^\beta B$ where $\alpha$, $\beta$, $A$ and $B$ are positive integers with $A$ and $B$ not divisible by $p$.  Suppose for contradiction that $a$ and $b$ are not both powers of $p$, so $A$ and $B$ are not both $1$.  Then $p^\alpha+p^\beta$ lies in $\{2, \dots, N-1\}$, and hence we can take $k=p^\alpha+p^\beta$ in \eqref{eqks2}.  All terms on the left-hand side are divisible by $p$ except for that with $j=p^\alpha$, which is not, so we get the desired contradiction and conclude that $a$ and $b$ are indeed both powers of $p$.

Now suppose for contradiction that $a \neq b$ and take $k=\min(a, b)$ in \eqref{eqks2}.  All terms on the left-hand side are divisible by $p$ except the first (if $a > b$) or the last (if $b > a$), so again we obtain a contradiction and now conclude that $a=b$.  Putting everything together, we deduce that if $N \geq 4$, $p \neq 2$, and $a,b>0$ then \ref{primeitm2} holds.

Next we deal with the case $N \geq 4$, $p\neq 2$, and $b=0$; the $a=0$ case is analogous.  We need to show that $N$ is a power of $p$.  Now \eqref{eqks2} reads
\begin{equation}
\label{eqaks}
\binom{a}{k} \equiv 0
\end{equation}
for all $k$ in $\{2, \dots, N-1\}$, and if $a=p^\alpha A$ then $k=p^\alpha$ gives a contradiction unless $A=1$.  Therefore $N=a$ is a power of $p$ and \ref{primeitm1} holds.

Now consider the case $p=2$ with $N \geq 3$.  Since here there is no distinction between $1$ and $-1$ we may assume that $b=0$.  From \eqref{eqe1} we see that $(N-1)\overline{\eta} = Nt$, which forces $N$ to be even and hence $\overline{\eta}$ to vanish (otherwise the left-hand side is zero but the right-hand side is invertible).  This turns \eqref{eqks} into \eqref{eqaks}, and we argue as in the previous paragraph to show that \ref{primeitm1} holds.

Finally we turn to the exceptional case $N=3$ and $p \neq 2$.  \eqref{eqe1} and \eqref{eqk1} are still valid, and they give $2\overline{\eta} = (a-b)t$ and $2\overline{\eta}^2 = 3t^2$ respectively.  Combining these we get $(a-b)^2 \equiv 6$, and the left-hand side can only take the values $9$ and $1$, so a solution is only possible if $p$ is $3$ or $5$.  The corresponding values of $|a-b|$ are $3$ and $1$ respectively, so \ref{primeitm1} or \ref{primeitm3} holds.
\end{proof}

\begin{cor}[First part of \cref{mthmHF}]
\label{corPrimePow}
Over any coefficient ring, and with any relative spin structure and local system, $L$ is narrow unless $N$ is a prime power or twice a prime power.
\end{cor}
\begin{proof}
If $HF^*$ is non-zero then the subring generated by the unit is of the form $\Z/(m)$ for some non-negative integer $m$.  Take a prime $p$ dividing $m$, replace $R$ by $R \otimes_\Z \Z/(p)$---which contains the field $\Z/(p)$---and apply \cref{PrimePowProp}.  The new $HF^*$ is still non-zero by universal coefficients.
\end{proof}

%--------------------------

\subsection{Wideness}
\label{sscWideness}

Having seen that the Lagrangians in our family are narrow in many circumstances, we now use the methods of \cref{secOhSS} to show that for the trivial local system they are in fact wide in cases \ref{primeitm1} and \ref{primeitm2} of \cref{PrimePowProp}.

\begin{rmk}
In \cite{Casebook} we show that the same holds in case \ref{primeitm3}, but this uses different methods.
\end{rmk}

The first observation is that $\symp(X, L)$ contains an obvious subgroup isomorphic to $S_N$, the symmetric group on $N$ objects, which acts by permuting the factors of $X$.  This action preserves $L$ as it preserves the moment map \eqref{eqMomMap}, of which $L$ is the zero set.  In a slight abuse of notation we will refer to this subgroup simply as $S_N$.  Our goal is to study the action of $S_N$ on the Oh spectral sequence using \cref{labSSrep}.  First we need to understand its action on the cohomology and spin structure of $L$, and on disc classes.

\begin{lem}\label{HtpyTriv}  $S_N$ acts on $H_2(X, L)$ by permuting the $A_j$.  Its action on $L$ is trivial up to isotopy, and hence is trivial on $H^*(L; \Z)$ and preserves any (absolute) spin structure.
\end{lem}
\begin{proof}
The disc classes $A_j$ are dual to the divisors \eqref{eqDivCpts}, and these are manifestly permuted by $S_N$.  This proves the first statement.

For the second it suffices to show that the transposition of two factors is isotopic to the identity on $L$ and without loss of generality we may assume they are the first two.  Recall that $L$ is the free $\PSU(N-1)$-orbit of $x=([v_1], \dots, [v_N])$, so points of $L$ can be written uniquely in the form $g \cdot x$ for $g \in \PSU(N-1)$.  We claim that there exists an element $h$ in $\PSU(N-1)$ which swaps $[v_1]$ and $[v_2]$ but fixes the other $[v_j]$.  Since $\PSU(N-1)$ is connected, we can pick a path $h_t$ from $h_0 = \id$ to $h_1 = h$, and for each $t$ define a diffeomorphism $\phi_t$ of $L$ by $g \cdot x \mapsto gh_t \cdot x$.  The family $\phi_t$ then provides an isotopy from the identity to the required transposition.

It is left to construct $h$, so let $V \subset \C^{N-1}$ be the span of $v_3, \dots, v_N$ and $V^\perp$ its orthogonal complement.  For any angle $\theta$ there is a unique matrix $M_\theta$ in $\SU(N-1)$ which acts as $e^{i\theta}$ on $V$ and $e^{-(N-2)i\theta}$ on $V^\perp$, and we will be done if we can find $\theta$ such that $M_\theta$ swaps $[v_1]$ and $[v_2]$.  To do this, write $u_j$ and $w_j$ for the orthogonal projections of $v_j$ onto $V$ and $V^\perp$ respectively.  Recall from \eqref{vjInnerProduct} that we have $\ip{v_j}{v_k}=-1$ for distinct $j$ and $k$, so $u_1$ and $u_2$ must coincide ($v_1$ and $v_2$ have equal inner products with a set of vectors that span $V$).  We also have $v_1+\dots+v_N=0$, so projecting onto $V^\perp$ gives $w_1+w_2=0$.  Swapping $[v_1]$ and $[v_2]$ therefore corresponds to reversing the sign of $V^\perp$ relative to $V$, so choosing $\theta$ to satisfy $e^{-i(N-2)\theta}=-e^{i\theta}$ does the job.
\end{proof}

In light of this, \cref{labSSrep} tells us that if we equip $L$ with the standard spin structure and trivial local system then the differentials in the Oh spectral sequence over $R[H_2^D] \cong R[T_1^{\pm 1}, \dots, T_N^{\pm 1}]$ (identifying $T^{A_j}$ with $T_j$) commute with the action of $S_N$ by permuting the $T_j$.  In particular, when restricted to the subset $H^*(L; R)$ of the zeroth column of $\ssE{1}$, the $\diff_1$ differential lands in
\[
H^*(L; R) \otimes (\text{\emph{symmetric} Laurent polynomials in the $T_j$, homogeneous of degree $2$})
\]
(recall that each $T_j$ has degree $2$, so the total exponent in each Laurent monomial must be $1$).  Similarly, when restricted to whatever survives from this subset to the zeroth column of $\ssE{r}$, the Laurent polynomial coefficients output by the $\diff_r$ differential are symmetric and homogeneous of degree $2r$.

Our strategy is rougly as follows (warning: this paragraph involves various imprecise or not quite correct statements).  Assume for simplicity that $R$ is a field of prime characteristic $p$ and $N$ is a power of $p$.  We wish to show that $L^\flat$ is wide, and by the proof of \cref{labBCWide} it suffices to show that the differentials in the spectral sequence over $\Lambda$ vanish on pages $\ssE{1}, \dots, \ssE{N-1}$.  In the previous paragraph we saw that over $R[H_2^D]$ the differentials only output Laurent polynomial coefficients which are symmetric and homogeneous of degree $2, 4, \dots, 2(N-1)$, so we are done if we can show that all such Laurent polynomials vanish upon passing to $\Lambda$ by setting all $T_j$ equal to $T$.  This vanishing follows from the fact that $\Char R=p$ and that $N$ is a power of $p$.

To make this strategy go through, we actually argue slightly differently.  We study the spectral sequence over the quotient $\Lambda_R^H \coloneqq R[H_2^D]/(e_1, \dots, e_{N-1})$, where $e_j$ is the $j$th elementary symmetric polynomial.  Using the representation theory of $S_N$ we show that all differentials vanish in the case $R=\Q$; really we care about fields $R$ of prime characteristic dividing $N$, but in this case the representation theory may be badly behaved.  We then use a universal coefficients argument, passing via $\Lambda_\Z'$, to show that the differentials also vanish over $\Lambda_R^H$ when $R$ is a field of characteristic coprime to $N-1$.  The conditions on $\Char R$, $N$, and $s$ from \cref{PrimePowProp} then enter when we try to pass from $\Lambda_R^H$ to $\Lambda=\Lambda_R$.

%--------------------------------

\subsection{Representation theory}
\label{sscRepThy}

The main ingredient, which is the goal of this subsection, is an understanding of $\Lambda_\Q^H$ as a $\Q[S_N]$-module.  We could replace $\Q$ by any field of characteristic coprime to $|S_N|=N!$ in these results since all we use is the semisimplicity of $\Q[S_N]$.  However, this extra generality could be misleading because the fields we are ultimately interested in do not satisfy this condition.  As sketched above, we will work around this using the universal coefficient theorem.

As a stepping stone we first consider instead the quotient $\Q[T_1, \dots, T_N]/(e_1, \dots, e_N)$; note that here the ideal includes $e_N$, whilst in the definition of $\Lambda_\Q^H$ it does not.  Recall the well-known fact that $e_1, \dots, e_N$ are algebraically independent and that $\Q[T_1, \dots, T_N]/(e_1, \dots, e_N)$ is a free $\Q$-vector space of rank $N!$.  The identification of this module as a representation of $S_N$ is standard (see, e.g.~\cite[Corollary 2.5.8]{Manivel} for the statement over $\C$), and can be done by pure algebra, but there is also a beautiful geometric argument, explained to the author by Oscar Randal-Williams:

\begin{lem}\label{lemSrRep}  The representation
\[
\Q[T_1, \dots, T_N]/(e_1, \dots, e_N)
\]
of $S_N$ is isomorphic to the regular representation $\Q[S_N]$.
\end{lem}
\begin{proof}[Sketch proof]
The ring $\Z[T_1, \dots, T_N]/(e_1, \dots, e_N)$ is isomorphic to the cohomology of the variety $F$ of complete flags in $\C^N$, with $T_j$ giving the first Chern class of the $j$th tautological line bundle.
%Viewing $F$ as the unitary group $\U(N)$ modulo the right-action of the group $T^N$ of diagonal matrices, this can be seen as follows.  For $r=0, \dots, N$ let $F_r$ denote the space of `first $r$ columns of elements of $\U(N)/T^N$'.  We obtain an obvious chain of fibrations
%\[
%F=F_N \twoheadrightarrow F_{N-1} \twoheadrightarrow \dots \twoheadrightarrow F_1 \twoheadrightarrow F_0 = *
%\]
%with fibres $\C\P^1, \dots, \C\P^{N-1}$ respectively, and we apply the Serre spectral sequence to each step.  If $\mathcal{L}_r$ denotes the tautological complex line bundle over $F$ whose fibre over a matrix $Q \in \U(N) / T^N$ is the span of the $r$th column of $Q$ inside $\C^N$, then we see that the first Chern classes $T_1, \dots, T_N$ of the $\mathcal{L}_r$ generate the cohomology as a ring, and in particular that the monomials
%\begin{equation}
%\label{MonomialsSpan}
%\{T_1^{r_1}\cdots T_N^{r_N} : 0 \leq r_j \leq N-j\}
%\end{equation}
%form a free $\Z$-basis.  Thus $H^*(F; \Z)$ is a free $\Z$-module of rank $N!$.  The sum of the bundles $\mathcal{L}_j$ is trivial, so $\prod_j (1+T_j) = 1$ and $e_1, \dots, e_N$ vanish in $H^*(F; \Z)$.  We obtain a surjective ring map
%\[
%\Z[T_1, \dots, T_N]/(e_1, \dots, e_N) \rightarrow H^*(F; \Z),
%\]
%and since both sides are free $\Z$-modules of rank $N!$ it must be an isomorphism.  Tensoring gives the corresponding result over $\Q$.
Viewing $F$ as the quotient of the unitary group $\U(N)$ by the right-action of diagonal matrices $T^N$, the $S_N$-action on $\Q[T_1, \dots, T_N]/(e_1, \dots, e_N)$ corresponds to the action on $H^*(F; \Q)$ induced by permutation of the columns of matrices in $\U(N)/T^N$.  This action on $F$ is clearly free, so there is a smooth quotient manifold $F/S_N$.  Lifting a cell structure from this quotient, we obtain a cellular cochain complex of free $\Q[S_N]$-modules which computes $H^*(F; \Q)$ as a representation of $S_N$.

Since $\Q[S_N]$ is semisimple, it makes sense to talk about the Euler characteristic of a complex of $\Q[S_N]$-modules: this is a tuple of multiplicities indexed by isomorphism classes of simple modules.  For our lifted complex this Euler characteristic is manifestly an integer multiple of the multiplicities occurring in the regular representation.  On the other hand, the cohomology is concentrated in even degrees, so the Euler characteristic is just the multiplicities occurring in
\[
H^*(F; \Q) \cong \Q[T_1, \dots, T_N]/(e_1, \dots, e_N).
\]
Therefore the latter is isomorphic to a direct sum of copies of $\Q[S_N]$, and counting dimensions we see that the number of copies is exactly $1$.
\end{proof}

To transfer this to $\Lambda_\Q^H$, in which the powers of the $T_j$ can be negative, we filter by `the number of negative powers'.  More precisely, let $F^0$ be the image of $\Q[T_1, \dots, T_N]$ in $\Lambda_\Q^H$, and for each integer $r$ let $F^r = e_N^r F^0$.  We obtain a descending $\Q[S_N]$-module filtration of $\Lambda_\Q^H$ which is Hausdorff $(\cap_r F^r=0$) and exhaustive ($\cup_r F^r = \Lambda_\Q^H$).  Moreover, each quotient $F^r/F^{r+1}$ is isomorphic, via multiplication by $e_N^{-r}$, as a graded $\Q[S_N]$-module to $F^0/F^1[-2Nr]$.  Here $[-2Nr]$ denotes the grading shift---all ideals and submodules involved are homogeneous, so everything inherits a natural grading.  Using this we can compute what we want:

\begin{prop}
There is an isomorphism of graded $\Q[S_N]$-modules
\[
\Lambda_\Q^H \cong \bigoplus_{r \in \Z} \Q[S_N][-2Nr].
\]
\end{prop}
\begin{proof}
Write $P$ as shorthand for the polynomial ring $\Q[T_1, \dots, T_N]$.  Since the filtration is Hausdorff and exhaustive, and $\Lambda_\Q^H$ is finite-dimensional in each degree, we have an isomorphism of graded $\Q$-vector spaces between $\Lambda_\Q^H$ and the associated graded module of the above filtration.  Since $\Q[S_N]$ is semisimple, this upgrades to an isomorphism of graded $\Q[S_N]$-modules.  It is therefore left to show that $F^0/F^1$ is isomorphic to $\Q[S_N]$, and this will follow from \cref{lemSrRep} if we can show that the kernel of the map $P \rightarrow F^0/F^1$ is the ideal (of $P$) generated by $e_1, \dots, e_N$.

To show this, note that the kernel is precisely the intersection in $\Q[H_2^D]$ of the subring $P$ with $I + e_NP$, where $I$ is the ideal (of $\Q[H_2^D]$) generated by $e_1, \dots, e_{N-1}$.  It therefore suffices to show that $I \cap P$ is the ideal of $P$ generated by $e_1, \dots, e_{N-1}$, and this is what we shall do.

Let $B$ be a free basis for $P$ as an $\Q[e_1, \dots, e_N]$-module (of rank $N!$).  Since $\Q[H_2^D]$ is obtained from $P$ by adjoining an inverse to $e_N$ (we can express $T_1^{-1}$ as $T_2\cdots T_N e_N^{-1}$ for example) $B$ also forms a free basis for $\Q[H_2^D]$ as an $\Q[e_1, \dots, e_{N-1}, e_N^{\pm 1}]$-module.  Therefore $\Q[H_2^D]$ is a free $\Q$-module with basis
\[
\{e_1^{j_1}\dots e_N^{j_N}b : j_1, \dots, j_{N-1} \in \Z_{\geq 0} \text{, } j_N \in \Z \text{, and } b\in B\}.
\]
The ideal $I$ is the $\Q$-linear span of those elements with $j_l$ strictly positive for some $l$ in $\{1, \dots, N-1\}$, whilst $P$ is the span of those elements with $j_N \geq 0$.  Their intersection is thus the span of those elements for which some $j_l$ is positive \emph{and} $j_N \geq 0$, which is precisely the ideal of $P$ generated by $e_1, \dots, e_{N-1}$.
\end{proof}

For the argument sketched in \cref{sscWideness}, the relevant consequence is

\begin{cor}\label{corSrFix}  The $S_N$-invariant subring of $\Lambda_\Q^H$ is spanned by elements in degrees $2N\Z$.\hfill$\qed$
\end{cor}

%-----------------------

\subsection{Completing the proof of \cref{mthmHF}}

The other ingredient is the cohomology ring of $L \cong \PSU(N-1)$, as computed by Borel \cite{Borel}.  See also \cite{BauBrow}.  The properties we shall use are:

\begin{lem}
\label{HPSU}
$H^*(\PSU(N-1); \Z)$ is generated as a ring by elements of degree $\leq 2N-3$.  For any field $R$ of characteristic coprime to $N-1$ we have
\[
\dim_R H^*(\PSU(N-1); R) = 2^{N-2}.
\]
\end{lem}
\begin{proof}[Sketch proof]
The crucial idea is to replace the quotient $\PSU(N-1) = \U(N-1) / S^1$ by the homotopy quotient $\U(N-1) \times_{S^1} ES^1$, which is a principal $U(N-1)$-bundle over $BS^1 = \C\P^\infty$, so that we can employ the Serre spectral sequence.  The second page
\[
\ssE{2} \cong H^*(\C\P^\infty; H^*(\U(N-1); \Z))
\]
is isomorphic as a ring to $\Z[b] \otimes \Lambda_\Z (a_1, a_3, \dots, a_{2N-3})$, where $b$ has degree $2$ and $a_{2j-1}$ has degree $2j-1$.  The generators $a_{2j-1}$ of $H^*(\U(N-1); \Z)$ can be chosen so that each transgresses to the Chern class $c_j$ of the associated $\C^{N-1}$-bundle over $\C\P^\infty$ (this can be seen by viewing the bundle as a pullback of the universal bundle over $B\U(N-1)$ and using naturality of the Serre spectral sequence), and these Chern classes were computed in \cite[Section 4]{BauBrow} to be
\[
c_j = \binom{N-1}{j} b^j.
\]
Each page is therefore generated as a ring by elements of degree at most $2N-3$ (specifically by $b$ and whichever classes $a_{2j-1}$ that have survived), so the same is true of $H^*(\PSU(N-1); \Z)$.  Passing to a field of characteristic coprime to $N-1$, the class $b$ is hit by the $\ssE{2}$ differential, so $a_1$ and $b$ die but all other classes survive to the limit, giving the dimension equality.
\end{proof}

Combining this with \cref{corSrFix} we obtain the key result towards establishing wideness.

\begin{prop}\label{propSrWide}
Equipping $L$ with the standard spin structure and trivial local system, all differentials in the Oh spectral sequence over $\Lambda_\Q^H$ (starting from $\ssE{1}$) vanish.  The same holds with $\Q$ replaced by any field $R$ of characteristic $p$ (prime or $0$) not dividing $N-1$.
\end{prop}
\begin{proof}  Consider the first page $H^*(L; \Q) \otimes \Lambda_\Q^H$ of the spectral sequence over $\Lambda_\Q^H$.  We saw in \cref{sscWideness} that $S_N$ acts trivially on $H^*(L; \Q)$ and on the spin structure, so its action on this page is purely on $\Lambda_\Q^H$.  Restricting the differential $\diff_1$ to the subset $H^*(L; \Q) \otimes 1$ of the zeroth column, the output lands in
\[
H^*(L; \Q) \otimes (\text{$S_N$-invariant degree $2$ part of $\Lambda_\Q^H$}),
\]
but this is zero by \cref{corSrFix}.  Thus $\diff_1$ vanishes on this subset and hence, by $\Lambda_\Q^H$-linearity, on the whole page.

Now apply the same argument on $\ssE{2}, \dots, \ssE{N-1}$, using the fact that $\Lambda_\Q^H$ also has no invariant part in degrees $4, 6, \dots, 2(N-1)$, to see that $\diff_2, \dots, \diff_{N-1}$ also vanish.  All later differentials vanish by an argument analogous to \cref{labBCWide}, since $H^*(L; \Q)$ is generated in degrees less than $2N-1$ by \cref{HPSU}.

Letting $L^\flat$ denote our Lagrangian brane, we deduce that $HF^*(L^\flat, L^\flat; \Lambda_\Q^H) \cong H^*(L; \Lambda_\Q^H)$ as graded $\Q$-vector spaces.  Since $\Lambda_\Q^H$ is finite-dimensional in each degree (it has rank $N!$ over $\Q[e_N^{\pm1}]$), we can count dimensions and use the universal coefficient theorem to obtain, for any field $R$ and for all $j$,
\[
\dim_R HF^j(L^\flat, L^\flat; \Lambda_R^H) \geq \rank_\Z HF^j(L^\flat, L^\flat; \Lambda_\Z^H) = \dim_\Q HF^j(L^\flat, L^\flat; \Lambda_\Q^H) = \dim_\Q H^j(L; \Lambda_\Q^H).
\]
If $\Char R$ is coprime to $N-1$ then we can replace the right-hand side by $\dim_R H^j(L; \Lambda_R^H)$, using the second part of \cref{HPSU}, and deduce
\begin{equation}
\label{RankIneq}
\dim_R HF^j(L^\flat, L^\flat; \Lambda_R^H) \geq \dim_R H^j(L; \Lambda_R^H).
\end{equation}
Considering the spectral sequence for $L^\flat$ over $\Lambda_R^H$, we see that all differentials must vanish (and for each $j$ \eqref{RankIneq} is in fact an equality).
\end{proof}

Using this we can complete the proof of \cref{mthmHF}.

\begin{thm}\label{propWide}  Suppose $R$ is a field of characteristic $p$ (prime or $0$) not dividing $N-1$, and let $L^\flat$ be $L$ equipped with a relative spin structure $s$ and the trivial local system over $R$.  Let $\eps \in H^2(X, L; \Z/2)$ be the difference between $s$ and the standard spin structure, and let $\eps_j$ denote $(-1)^{\ip{\eps}{A_j}}$ as in \cref{sscCOconstraints}.  If the elementary symmetric polynomials $e_k(\eps_1, \dots, \eps_N)$ are zero in $R$ for $k=1, \dots, N-1$ then $L^\flat$ is wide over $R$.  In particular, this holds in cases \ref{primeitm1} and \ref{primeitm2} of \cref{PrimePowProp}, proving the second part of \cref{mthmHF}.
\end{thm}
\begin{proof}
By \cref{SpinChangeProp} the spectral sequence for $L^\flat$ over $\Lambda_R$ is obtained from the spectral sequence with the standard spin structure over $R[H_2^D]$ by sending $T_j$ to $\eps_j T$.  If the stated symmetric polynomials are zero then this map $R[H_2^D]\rightarrow \Lambda_R$ factors through $\Lambda_R^H$, and by \cref{propSrWide} the differentials in the spectral sequence over the latter all vanish.  In this case we conclude that the differentials for $L^\flat$ over $\Lambda_R$ vanish, so $L^\flat$ is wide.

To prove the last part, note that the symmetric polynomials are zero if and only if we have
\[
\prod_{j=1}^N (z+\eps_j) = z^N + \prod_{j=1}^N \eps_j
\]
in the polynomial ring $R[z]$.  In case \ref{primeitm1} this reduces to
\[
(z\pm1)^N = z^N+(\pm1)^N,
\]
whilst in case \ref{primeitm2} it becomes
\[
(z^2-1)^{N/2} = z^N+(-1)^{N/2}.
\]
Both of these equalities follow from the fact that $(x+y)^r = x^r + y^r$ in characteristic $p$ when $r$ is a power of $p$.
\end{proof}

\begin{qn}
In case \ref{primeitm1} each of the wide relative spin structures is invariant under the action of $S_N$, so there is an induced action on the ring $HF^*(L^\flat, L^\flat; \Lambda_R)$.  Setting $T=1$ we obtain a $2^{N-2}$-dimensional representation of $S_N$ over $R$.  What is this representation?  From the spectral sequence we see that it carries a filtration such that the associated graded representation is trivial, but since the characteristic of our field divides the order of the group the representation need not be semisimple.

When $N$ is twice a prime power each wide relative spin structure is no longer invariant under the whole of $S_N$.  Instead its stabiliser is a subgroup $S_{N/2} \times S_{N/2}$, and similarly one can ask whether the corresponding representation on Floer cohomology is trivial.
\end{qn}

%---------------------------------

\subsection{$B$-fileds and non-displaceability}

We end by proving \cref{mthmNonDisp}, namely that for all values of $N$ the Lagrangian $L$ is non-displaceable.  To do this we turn on a $B$-field---a closed complex-valued $2$-form $B$ on $X$.  Cho \cite[Section 2]{ChoNonU}, following Fukaya \cite{FukFam}, describes how the Lagrangian intersection Floer complex for a pair of Lagrangians $L_0$ and $L_1$ can be deformed by such a $B$, if its restriction to each $L_j$ defines an integral cohomology class (i.e.~a class in the image of $H^2(L_j; \Z)$ in $H^2(L; \C)$), and we briefly recap this construction.  See also recent work of Siegel \cite{Siegel}.

Since $[B]$ restricts to an integral class on $L_j$ there exists a complex line bundle $\mathcal{L}^j$ on $L_j$ with $-[B]|_{L_j}$ as its first Chern class, and hence we can put a connection $\nabla^j$ on $\mathcal{L}^j$ whose curvature is $2 \pi i B$.  Just as in the case $B=0$, where the connections $\nabla^j$ are flat and hence define locally constant structures on the $\mathcal{L}^j$, the Floer complex $CF^*(L_0, L_1)$ is generated by the direct sum over intersection points $x \in L_0 \cap L_1$ of the spaces $\operatorname{Hom} (\mathcal{L}^0_x, \mathcal{L}^1_x)$, assuming these intersections are transverse.  The differential counts pseudoholomorphic strips $u$ which contribute according to the parallel transport maps of $(\mathcal{L}^j, \nabla^j)$ along $u|_{t=j}$, but now there is an extra factor of $e^{-2\pi i \int u^* B}$.  By Stokes's theorem this combination of the parallel transports and the exponential factor depends only on the homotopy class of the strip.  The usual continuation map argument shows that if this deformed Floer cohomology is non-zero then $L_1$ is not displaceable from $L_0$ by a Hamiltonian isotopy.

\begin{rmk}
This can be extended to larger collections of Lagrangians and higher $A_\infty$-operations in the obvious way, to give a deformation of the full subcategory of the (monotone) Fukaya category of $X$ comprising those (monotone) Lagrangians on which the class $[B]$ is integral.  The deformation only depends on the image of $[B]$ in $H^2(X; \C) / H^2(X; \Z)$.
\end{rmk}

The analogous construction in the pearl model is to take a single Lagrangian $L$ and equip it with a line bundle $\mathcal{L}$ and connection $\nabla$ of curvature $2 \pi i B$.  A trajectory is then weighted by the combination of boundary monodromy and exponential integral of $-2\pi i B$ for each disc.  This has a more topological description as follows.  The map which sends a smooth $2$-cycle $\sigma$ in $X$ to $e^{-2\pi i \int \sigma^* B}$ defines a class in $\operatorname{Hom}(H_2(X; \Z), \C^*) = H^2(X; \C^*)$, which we denote by $[e^{-2\pi i B}]$.  The condition that $[B]|_L$ is integral can be phrased more cleanly as $[e^{-2 \pi i B}]|_L = 1$ (the identity in $\C^*$), and the choice of $(\mathcal{L}, \nabla)$ (up to isomorphism) corresponds to a lift of $[e^{-2\pi i B}]$ to $H^2(X, L; \C^*)$, which we denote by $[e^{-2\pi i B}]_L$.  The weight attached to a trajectory of class $A$ is then given simply by evaluating $[e^{-2\pi i B}]_L$ on $A$.

Just as changes of spin structure can be realised by changes of local system (as in \cref{SpinLocSys}), changes of relative spin structure can be realised by changes of $B$-field.  Specifically, changing relative spin structure by $\eps \in H^2(X, L; \Z/2)$ is equivalent to multiplying $[e^{-2\pi i B}]_L$ by the class
\[
(-1)^\eps \in H^2(X, L; \{\pm 1\}) \rightarrow H^2(X, L; \C^*).
\]
It is precisely this analogy that will allow us to prove non-displaceability.

\begin{thm}[\cref{mthmNonDisp}]\label{thmNonDisp}
For any $N \geq 3$ our $\PSU(N-1)$-homogeneous Lagrangian $L$ can be made wide over $\C$ by an appropriate choice of $B$-field.  In particular, $L$ is non-displaceable.
\end{thm}
\begin{proof}
Suppose we turn on a $B$-field and choose a lift $[e^{-2\pi i B}]_L$ as above (we have $H^2(L; \C)=0$ so the integrality condition on $[B]|_L$ is vacuous).  Let $b_j \in \C^*$ denote the evaluation of $[e^{-2\pi i B}]_L$ on the basic disc class $A_j$.  In \cref{propWide} we obtained the spectral sequence over $\Lambda_R$ with the non-standard relative spin structure from the spectral sequence over $R[H_2^D]$ with the standard spin structure by sending $T_j$ to $\eps_jT$.  Now, in exactly the same way, we obtain the spectral sequence over $\Lambda_\C$ with the $B$-field deformation from that over $\C[H_2^D]$ by sending $T_j$ to $b_jT$.  The wideness criterion now becomes $e_k(b_1, \dots, b_k)=0$ for $k=1, \dots, N-1$, or equivalently
\[
\prod_{j=1}^N (z+b_j) = z^N + \prod_{j=1}^N b_j
\]
in $\C[z]$, and our task is to show that there exists a choice of the $b_j$ for which this holds.  Taking the $b_j$ to be the $N$th roots of an arbitrary $b \in \C^*$ will do (conversely any solution has this form).
\end{proof}

\begin{rmk}
\label{WideVar}
In \cite[Section 3]{BCEG}, Biran and Cornea define the \emph{wide variety} $\mathcal{W}_2$ of a monotone Lagrangian $L \subset X$ (closed, connected, orientable, and equipped with a choice of spin structure).  In our setting, this is precisely the set
\[
\{(b_1, \dots, b_N) \in (\C^*)^N : L \text{ is wide with the corresponding $B$-field}\}.
\]
The above argument shows that $\mathcal{W}_2$ contains the variety defined by the equations
\begin{equation}
\label{W2eqn1}
e_1(b_1, \dots, b_N) = \dots = e_{N-1}(b_1, \dots, b_N) = 0.
\end{equation}
On the other hand, by \cref{CorCO}, for any point $(b_1, \dots, b_N) \in \mathcal{W}_2$ we have that
\[
e_1(\overline{\eta}-b_1 \cdot 1_L, \dots, \overline{\eta}-b_N \cdot 1_L)=\overline{\eta}
\]
and
\[
e_k(\overline{\eta}-b_1 \cdot 1_L, \dots, \overline{\eta}-b_N \cdot 1_L)=0 \quad \text{for} \quad k=2, \dots, N-1
\]
in $HF^*$.
The former implies that
\[
\overline{\eta} = \frac{e_1(t_1, \dots, t_N)}{N-1} \cdot 1_L,
\]
and substituting into the remaining equations we deduce that
\begin{equation}
\label{W2eqn2}
e_k\lb t_1 - \frac{e_1(t_1, \dots, t_N)}{N-1}, \dots, t_N - \frac{e_1(t_1, \dots, t_N)}{N-1} \rb = 0 \quad \text{for} \quad k=2, \dots, N-1.
\end{equation}
Therefore $\mathcal{W}_2$ lies between the varieties defined by equations \eqref{W2eqn1} and \eqref{W2eqn2}.  We show in \cite{VirtualPearl} that it is in fact given by the latter.
\end{rmk}

\begin{rmk}\label{NonHomogIdeal}
If $e_1(t_1, \dots, t_N)$ is equal to some $\lambda \in \C^*$ then we cannot just apply our earlier wideness argument with the ideal $I_\lambda = (e_1-\lambda, e_2, \dots, e_{N-1})$ in place of $I = (e_1, \dots, e_{N-1})$.  The reason is that $I_\lambda$ is not homogeneous, and the filtration induced on $\Lambda_\lambda \coloneqq \Q[H_2^D]/I_\lambda$ by the grading filtration on $\Q[H_2^D]$ is trivial.  Thus the spectral sequence arising from the induced filtration on the pearl complex over $\Lambda_\lambda$ is zero, and contains no information.
\end{rmk}

%---------------------------------------

\appendix

%---------------------------------------------

\section{Orientation computations}\label{appOri}

%--------------------------------------------

\subsection{Setup}
\label{sscSetup}

In this appendix we verify the signs in \cref{thmCO}.  So assume throughout that: $L \subset X$ is sharply $K$-homogeneous and equipped with an arbitrary orientation and the standard spin structure; $M \subset X$ is a $K$-invariant complex submanifold of complex codimension $k$ (in practice this is the smooth locus of the invariant subvariety $Z$); $A \in H_2(X, L; \Z)$ is a class of index $2k$; and $y_\mathrm{min}$ is an arbitrary point in $L$.  We need to show that each axial disc $u$ in
\[
M \times_X \mathcal{M}(A, J_\mathrm{std}) \times_L \{y_\mathrm{min}\}
\]
carries a positive sign, where $\mathcal{M} = \mathcal{M}(A, J_\mathrm{std})$ maps to the $L$ on the left by $\ev_0$ and the $L$ on the right by $\ev_1$.

From \cite[Section A.1.8]{BCEG} the sign of $u$ is positive if and only if the map
\[
T_{u(0)}M \oplus T_{u(0)}X \oplus T_u\mathcal{M} \oplus T_{y_\mathrm{min}}L \rightarrow T_{u(0)}X \oplus T_{u(0)}X \oplus T_{y_\mathrm{min}}L \oplus T_{y_\mathrm{min}}L
\]
given by
\[
(v_1, v_2, v_3, v_4) \mapsto (v_1-v_2, v_2-D_u\ev_0 (v_3), D_u\ev_1(v_3)-v_4, v_4)
\]
is orientation-preserving.  Here $TX$ and $TM$ carry their complex orientations, whilst $TL$ carries the orientation we chose on $L$.  Using row and column operations this can be reduced to the orientation sign of
\begin{equation}
\label{TargetSign}
D_u\ev_0 \oplus D_u\ev_1 : T_u\mathcal{M} \rightarrow T_{u(0)}X/T_{u(0)}M \oplus T_{y_\mathrm{min}}L.
\end{equation}
We will explicitly construct a positively-oriented basis of the left-hand side and show that this map is orientation-preserving.

%------------------------------------------------

\subsection{Orienting disc moduli spaces}
\label{sscOrientingDiscs}

First we need to understand some properties of the orientations of disc moduli spaces constructed in \cite[Chapter 8]{FOOObig}, so take an arbitrary disc $u$ in $\mathcal{M}$.  Recall from \cref{ModSmooth} that we have a decomposition of bundle pairs
\[
(u^*TX, u|_{\pd D}^*TL) \cong \bigoplus_j (\underline{\C}, z^{\kappa_j/2}\underline{\R}),
\]
where the $\kappa_j \in \Z$ are the partial indices of $u$, all of which are in fact non-negative.  We abbreviate the left-hand side to $(E, F)$.  The tangent space $T_u\mathcal{M}$ is the kernel of the Cauchy--Riemann operator
\[
\conj\pd_{(E, F)} : \Gamma \big( (D, \pd D), (E, F) \big) \rightarrow \Gamma(D, \Omega^{0, 1}(E)).
\]
A choice of orientation and relative spin structure on $L$ induces a homotopy class of trivialisation of $F$, and it is this trivialisation which is used to define the orientation on $T_u \mathcal{M}$.

Roughly speaking, the construction proceeds by degenerating the disc $D$ into a nodal curve $D \cup \C\P^1$ (joined at the point $0$ in each component), itself carrying a bundle pair $(E', F')$ and Cauchy--Riemann operator $\conj\pd_{(E', F')}$, such that $(E, F)$ and $(E', F')$ are identified over a collar neighbourhood of $\pd D$.  Gluing results give a bijection between holomorphic sections of $E'$ over each component of the nodal curve separately, which agree at the joining point, and holomorphic sections of the original bundle $E$.  From this we obtain an isomorphism
\begin{equation}
\label{eqorker}
\ker \conj\pd_{(E, F)} \cong \ker \big( \ker \conj\pd_{(E'|_D, F')} \oplus \ker \conj\pd_{E'|_{\C\P^1}} \rightarrow E'_0 \big),
\end{equation}
where $E_0'$ is the fibre over $0$ and the map on the right-hand side sends a pair of sections $(s_1, s_2)$ to $s_1(0)-s_2(0)$.  Since the disc $u$ is regular (\cref{ModSmooth}) this map is surjective, so we have a short exact sequence
\begin{equation}
\label{eqorSES}
0 \rightarrow \ker \conj\pd_{(E, F)} \rightarrow \ker \conj\pd_{(E'|_D, F')} \oplus \ker \conj\pd_{E'|_{\C\P^1}} \rightarrow E'_0 \rightarrow 0.
\end{equation}

The degeneration is done in such a way that $(E'|_D, F')$ is trivialised by our choice of trivialisation of $F$, giving an identification of $\ker \conj\pd_{(E'|_D, F')}$ with $\R^n$ by evaluating solutions at a boundary point (and reversing the orientation on $L$ switches the orientation of this identification, as claimed in \cref{OrientationReverse}).  The other two spaces appearing on the right-hand side of \eqref{eqorker} carry complex structures and hence canonical orientations.  Putting these three orientations together with the short exact sequence \eqref{eqorSES}, we see that there is an induced orientation on $\ker \conj\pd_{(E, F)}$, which is the space we are really interested in.  The auxiliary choices made do not affect this orientation.

The details of this procedure are technical and unimportant for our purposes, but we note two key properties of the construction:
\begen
\item\label{ori1} If the trivialisation of $F$ extends to a holomorphic trivialisation of $E$ then the bundle pair $(E', F')$ is trivial and $\ker \conj\pd_{(E, F)}$ is oriented directly by its identification with a fibre of $F$ by evaluation at a boundary point.
\item\label{ori2} If $(E, F)$ splits as a direct sum $(E^1, F^1) \oplus (E^2, F^2)$, compatible with the trivialisation of $F$ (i.e.~so that fibrewise, over each point $z$, $F_z^1$ corresponds to the first $k$ components of $F_z \cong \R^n$ and $F_z^2$ to the last $n-k$) then we can make the construction respect this splitting and observe that the identification
\[
\ker \conj\pd_{(E, F)} = \ker \conj\pd_{(E^1, F^1)} \oplus \ker \conj\pd_{(E^2, F^2)}
\]
is orientation-preserving.
\end{enumerate}

All we shall need are these two properties and the following explicit calculation:

\begin{lem}\label{ind2sign} Suppose $(E, F)$ is a rank $1$ Riemann--Hilbert pair of index $2$---which can be identified with the tangent bundle of the pair $(D, \pd D)$---with $F$ oriented.  Evaluation at $0$ and $1$ defines an isomorphism
\[
f : \ker \conj{\pd}_{(E, F)} \xrightarrow{\sim} E_0 \oplus F_1,
\]
and the codomain is oriented by the complex structure on $E_0$ and the choice of orientation on $F$.  This isomorphism is orientation-preserving.
\end{lem}
\begin{proof}
Let $(E^{(t)}, F^{(t)}) \rightarrow (D^{(t)}, \pd D^{(t)})$ be the family of bundle pairs, parametrised by $t \in [0, 1]$, realising the above degeneration of the domain from $D^{(0)} = D$ to $D^{(1)} = D \cup \C\P^1$.  Note that for all $t < 1$ the domain is biholomorphic to the disc $D$, and we may choose a continuous family $p^{(t)}$ of interior marked points in the domain which converge to the point $\infty$ in $\C\P^1 \subset D^{(1)}$ as $t \rightarrow 1$, as shown in \cref{figDiscDegen} (recall that $0$ in $\C\P^1$ is the point which is glued to $0$ in $D$ to construct $D^{(1)}$).  Similarly we can choose a continuous family $q^{(t)}$ of boundary marked points.
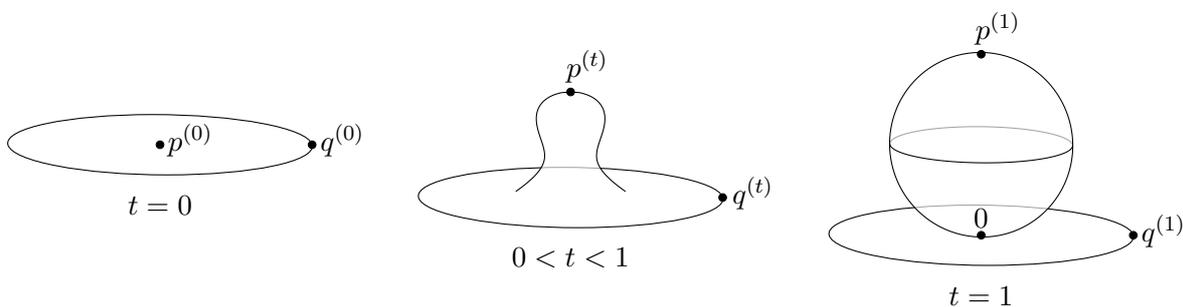
\begin{figure}[ht]
\centering
\begin{tikzpicture}[x=1cm, y={(-0.05cm, 0.2cm)}, z={(0cm, 1cm)}, blob/.style={circle, draw=black, fill=black, inner sep=0, minimum size=\blobsize}]
\def\blobsize{1mm}
\def\r{2}
\def\R{1.2}

\begin{scope}[shift={(-2.7*\r, 0, 0)}]
\draw (0, 0) circle[radius=\r];
\draw (0, 0, 0) node[blob]{};
\draw (0.4, 0, 0.06) node{$p^{(0)}$};
\draw (\r, 0, 0) node[blob]{};
\draw (\r+0.4, 0, 0.06) node{$q^{(0)}$};
\draw (0, 0, -0.4*\r) node{$t=0$};
\end{scope}

\begin{scope}[shift={(0, 0, -0.7)}]
\draw (0, 0) circle[radius=\r];
\draw[domain=-2.2:2.2, samples=200, fill=white, fill opacity=0.65, smooth] plot ({\x<0 ? -0.2*sqrt(8*abs(\x)^3-24*abs(\x)^2+20*abs(\x)) : 0.2*sqrt(8*abs(\x)^3-24*abs(\x)^2+20*abs(\x))}, {0}, {1.4-0.6*abs(\x)});
\draw (0, 0, 1.4) node[blob]{};
\draw (0.22, 0, 1.74) node{$p^{(t)}$};
\draw (\r, 0, 0) node[blob]{};
\draw (\r+0.4, 0, 0.06) node{$q^{(t)}$};
\draw (0, 0, -0.4*\r) node{$0 < t < 1$};
\end{scope}

\begin{scope}[shift={(2.7*\r, 0, -\R)}]
\draw (0, 0) circle[radius=\r];
\sphere{(0, 0, \R)}{\R}
\draw (0, 0, 2*\R) node[blob]{};
\draw ($(0, 0, 2*\R)+(0.22, 0, 0.34)$) node{$p^{(1)}$};
\draw (\r, 0, 0) node[blob]{};
\draw (\r+0.4, 0, 0.06) node{$q^{(1)}$};
\draw (0, 0, 0) node[blob]{};
\draw (0, 0, 0.23) node{$0$};
\draw (0, 0, -0.4*\r) node{$t=1$};
\end{scope}

\end{tikzpicture}
\caption{The degeneration $D^{(t)}$.\label{figDiscDegen}}
\end{figure}

We obtain a continuous family of kernels of the corresponding Cauchy--Riemann operators, which we denote by $\ker \conj{\pd}_t$.  At $t=1$ we impose the condition that the sections over $D$ and $\C\P^1$ agree at $0$, and continuity is in the sense of the above gluing result.  This comes equipped with a family of evaluation maps, giving a continuous family of isomorphisms
\[
f^{(t)} : \ker \conj{\pd}_t \xrightarrow{\sim} E^{(t)}_{p^{(t)}} \oplus F^{(t)}_{q^{(t)}}.
\]
The codomains are naturally oriented, and the definition of the orientation on $\ker \conj{\pd}_0$ is such that the orientation sign of $f^{(t)}$ is constant in $t$.  Taking $p^{(0)} = 0$ and $q^{(0)} = 1$, the statement of the lemma amounts to $f^{(0)}$ being orientation-preserving, so it suffices to show that $f^{(1)}$ is orientation-preserving.

The restriction of the bundle $E^{(1)}$ to the sphere component of $D^{(1)}$ has first Chern class given by half the index of our original Riemann--Hilbert pair, so is isomorphic to $\mathcal{O}(1)$.  We may therefore think of holomorphic sections as affine linear functions on $\C$, say $z \mapsto a+bz$ for complex numbers $a$ and $b$.  Solutions to the Riemann--Hilbert problem on the disc component, meanwhile, are all of the form $c s$, where $c$ is a real number and $s$ is an arbitrary fixed solution with $s(1)$ pointing in the positive direction in $F^{(1)}_1$.  The matching condition at $0$ forces $a=\lambda c$ for some fixed $\lambda \in \C^*$, so if $b=b_1+ib_2$ then from \eqref{eqorSES} we see that $b_1$, $b_2$, $c$ form a positively oriented set of coordinates on $\ker \conj{\pd}'$.  Choosing an appropriate basis vector for $E'_\infty$, the map
\[
f' : \ker \conj{\pd}' \rightarrow E'_\infty \oplus F'_1
\]
can be viewed as sending the point $(b_1, b_2, c)$ to $(b_1+ ib_2, c)$.  It is therefore orientation-preserving, which is exactly what we needed to show, completing the proof.
\end{proof}

%----------------------------------------------

\subsection{Splittings for axial discs}
\label{sscSplittings}

Now suppose $u$ is an axial disc with $u(0)$ in $M$, as in \cref{sscSetup}.  Our next goal is to give an explicit splitting of $(E, F)$ into rank $1$ pairs of index $0$ or $2$ and rank $2$ pairs with partial indices $(1, 1)$.  Note that in order for the boundary real subbundle to have a trivialisation it must be orientable, which is equivalent to having even total index, so there is no hope of decomposing the problem further into rank $1$ pairs of index $1$.  This splitting allows us to decompose the orientation problem using property \ref{ori1}, and then apply property \ref{ori2} and \cref{ind2sign} to orient each summand, after degenerating the $(1,1)$-summands into $(2,0)$-summands.

In order to construct the splitting, consider the infinitesimal action of $\mathfrak{g} = \mathfrak{k} \otimes \C$ at $u(0)$ in $M$, and let its kernel be $V$.  Since $M$ is $\mathfrak{g}$-invariant and of complex codimension $k$ we have $\dim_\C V \geq k$.  Pick an $\R$-basis $\xi_1, \dots, \xi_a$ for $V \cap \mathfrak{k}$, and extend to a $\C$-basis $\xi_1, \dots, \xi_a, \eta_1, \dots, \eta_b$ for $V$.  Each $\eta_j$ can be written as $\alpha_j + i \beta_j$ for unique $\alpha_j, \beta_j \in \mathfrak{k}$.

\begin{lem}
The set $\xi_1, \dots, \xi_a, \alpha_1, \dots, \alpha_b, \beta_1, \dots, \beta_b$ is $\R$-linearly independent in $\mathfrak{k}$.
\end{lem}
\begin{proof}
Let $\conj{\phantom{z}}$ denote the conjugate-linear involution of $\mathfrak{g} = \mathfrak{k} \otimes \C$ given by complex conjugation on the $\C$ factor.  The vectors $\xi_1, \dots, \xi_a, \eta_1, \dots, \eta_b, \conj{\eta}_1, \dots, \conj{\eta}_b$ span $V + \conj{V}$ over $\C$, and we claim that they form a basis.  Assuming this, $\xi_1, \dots, \xi_a, \alpha_1, \dots, \alpha_b, \beta_1, \dots, \beta_b$ also form a basis (we can transform between $\eta_j, \conj{\eta}_j$ and $\alpha_j, \beta_j$), so they are linearly independent over $\C$ and hence also over $\R$.

To prove the claim it suffices to show that $\dim_\C (V + \conj{V}) \geq a+2b$, and we know that
\[
\dim_\C (V + \conj{V}) = \dim_\C V + \dim_\C \conj{V} - \dim_\C (V\cap \conj{V}) = 2a+2b - \dim_\C (V\cap \conj{V}),
\]
so it's enough to show that $\dim_\C (V \cap \conj{V}) \leq a$.  In other words, we're done if we can find a set of size $a$ which spans the intersection $V \cap \conj{V}$.  We'll show that $\xi_1, \dots, \xi_a$ works.  Suppose then that $v$ is an element of this intersection, so both $v$ and $\conj{v}$ lie in $V$.  We can write $v$ uniquely in the form $v_\R + i v_{\mathbb{I}}$ with $v_\R$ and $v_\mathbb{I}$ in $\mathfrak{k}$, and we see that both $v_\R + i v_{\mathbb{I}}$ and $v_\R - i v_{\mathbb{I}}$ lie in $V$.  Thus
\[
v_\R = \frac{(v_\R + i v_{\mathbb{I}})+(v_\R - i v_{\mathbb{I}})}{2} \quad \text{and} \quad v_\mathbb{I} = \frac{(v_\R + i v_{\mathbb{I}})-(v_\R - i v_{\mathbb{I}})}{2i}
\]
lie in $V$ as well as in $\mathfrak{k}$.  They therefore lie in the real span of the $\xi_j$, so $v$ itself lies in the complex span, completing the proof.
\end{proof}

We can now pick $\theta_1, \dots, \theta_c$ which extend $\xi_1, \dots, \xi_a, \alpha_1, \dots, \alpha_b, \beta_1, \dots, \beta_b$ to a basis of $\mathfrak{k}$.  We obtain a trivialising frame for $F$
\[
\xi_1 \cdot u , \dots , \xi_a \cdot u, \alpha_1 \cdot u, \beta_1 \cdot u, \dots, \alpha_b \cdot u, \beta_b \cdot u, \theta_1 \cdot u, \dots, \theta_c \cdot u,
\]
where, for example, $\xi_1 \cdot u$ denotes the section $z \mapsto \xi_1 \cdot u(z) \in T_{u(z)}L$.  Since the actual choice of orientation on $L$ is irrelevant we may assume this frame is positively oriented.  The frame is compatible with the global trivialisation of $TL$ defining the standard spin structure, so its homotopy class is precisely that induced by this spin structure.

We can then split $E$ as
\begin{multline}
\label{eqsplit}
\lspan{\frac{\xi_1}{z}\cdot u}_\C \oplus \dots \oplus \lspan{\frac{\xi_a}{z}\cdot u}_\C \oplus \lspan{\frac{(1+z)\alpha_1+i(1-z)\beta_1}{z}\cdot u, \frac{(1+z)\beta_1-i(1-z)\alpha_1}{z}\cdot u}_\C
\\ \oplus \dots \oplus
\lspan{\frac{(1+z)\alpha_b+i(1-z)\beta_b}{z}\cdot u, \frac{(1+z)\beta_b-i(1-z)\alpha_b}{z}\cdot u}_\C \oplus \lspan{\theta_1 \cdot u}_\C \oplus \dots \oplus \lspan{\theta_c \cdot u}_\C,
\end{multline}
where $\lspan{\cdot}_\C$ denotes $\C$-linear span and, for example, $(\xi_1/z)\cdot u$ denotes the holomorphic section of $E$ given by $z \mapsto (\xi_1/z) \cdot u(z)$.  Note that the $a+2b+c = n$ sections of $E$ listed in \eqref{eqsplit} are holomorphic (since the $\xi_j \cdot u$ and $(\alpha_j + i \beta_j) \cdot u$ vanish at $0$), and are fibrewise $\C$-linearly independent.  To see the latter, consider the wedge product (over $\C$) of the sections: it is
\[
\frac{1}{z^{a+b}} (\xi_1 \cdot u) \wedge \dots \wedge (\xi_a \cdot u) \wedge (\alpha_1 \cdot u) \wedge (\beta_1 \cdot u) \wedge \dots \wedge (\alpha_b \cdot u) \wedge (\beta_b \cdot u) \wedge (\theta_1 \cdot u) \wedge \dots \wedge (\theta_c \cdot u),
\]
which is clearly non-zero on $D \setminus \{0\}$.  Since $u$ has index $2k$, this expression vanishes to order $k-(a+b)$ at $0$, but we know that $a+b = \dim_\C V \geq k$.  We thus have equality $a+b=k$ and the sections remain independent at $0$.  We also deduce that $T_{u(0)}M = \mathfrak{g} \cdot u(0)$.

For $j=1, \dots, a+b+c$ let $E^j$ denote the $j$th summand of \eqref{eqsplit}.  For each $j$, the intersection $F^j \coloneqq F \cap E^j|_{\pd D}$ is a subbundle of $F$ whose rank is half the real rank of $E^j$.  Moreover one can write down explicit frames for each $F^j$ in terms of the given trivialisations of the $E^j$:
\begin{gather*}
F^j = z \lspan{\frac{\xi_j}{z} \cdot u}_\R \text{ for } 1 \leq j \leq a
\\ F^{j+a} = z^{1/2} \lspan{\frac{(1+z)\alpha_j+i(1-z)\beta_j}{z}\cdot u, \frac{(1+z)\beta_j-i(1-z)\alpha_j}{z}\cdot u}_\R \text{ for } 1 \leq j \leq b
\\ F^{j+a+b} = \lspan{\theta_j \cdot u}_\R \text{ for } 1 \leq j \leq c,
\end{gather*}
where $\lspan{\cdot}_\R$ denotes $\R$-linear span.  In particular, the splitting of $E$ is in fact a splitting of the pair $(E, F)$ compatible with our trivialisation of $F$, and the first $a$ summands are rank $1$, index $2$; the next $b$ are rank $2$, partial indices $(1, 1)$; the final $c$ are rank $1$, index $0$.  We thus have the desired decomposition of $(E, F)$.

%-----------------------------------------------------

\subsection{Computing the signs}

Now return to the map \eqref{TargetSign} which we hope to prove is orientation-preserving.  Using the above splitting and property \ref{ori1}, $T_u\mathcal{M}$ decomposes as $\oplus_j \ker\conj{\pd}_{(E^j, F^j)}$.  We always order such sums from left to right.  Similarly $T_{y_\mathrm{min}}L$ decomposes as $\oplus_j F^j_1$ (the sum of the fibres over $1 \in \pd D$).  These decompositions are all as oriented vector spaces.

\begin{lem}
$T_{u(0)}M$ decomposes as $\oplus_j (E_0^j \cap T_{u(0)}M)$.
\end{lem}
\begin{proof}
We saw in \cref{sscSplittings} that $T_{u(0)}M = \mathfrak{g} \cdot u(0)$.  By definition, the space $V \subset \mathfrak{g}$ of complex codimension $k$ acts trivially, so giving a basis for $T_{u(0)}M$ is equivalent to giving a basis for $\mathfrak{g}/V$.  Consider the basis $\alpha_1-i\beta_1, \dots, \alpha_b-i\beta_b, \theta_1, \dots, \theta_c$ for the latter.  We claim that the induced splitting of $T_{u(0)}M$ is compatible with the splitting of $E_0$ into the $E_0^j$.

Well, the $\lspan{\theta_j \cdot u(0)}_\C$ summand of $T_{u(0)}M$ clearly lies in the $\lspan{\theta_j \cdot u}_\C$ summand of $E$.  Meanwhile, the
\[
\lspan{\frac{(1+z)\alpha_j+i(1-z)\beta_j}{z}\cdot u, \frac{(1+z)\beta_j-i(1-z)\alpha_j}{z}\cdot u}_\C
\]
summand of $E$ contains the section
\[
\frac{1}{2}\lb\frac{(1+z)\alpha_j+i(1-z)\beta_j}{z}\cdot u\rb - \frac{i}{2}\lb\frac{(1+z)\beta_j-i(1-z)\alpha_j}{z}\cdot u\rb = (\alpha_j - i \beta_j) \cdot u,
\]
and hence contains the $\lspan{(\alpha_j-i\beta_j)\cdot u(0)}_\C$ summand of $T_{u(0)}M$.
\end{proof}

We therefore need to understand the orientation sign of the map
\[
\bigoplus_j \ker \conj\pd_{(E^j, F^j)} \rightarrow \bigg( \bigoplus_j E^j_0\Big/\lb E^j_0 \cap T_{u(0)}Z\rb \bigg) \oplus \bigg( \bigoplus_j F^j_1 \bigg),
\]
which is the product of the orientation signs of the maps
\[
\ker \conj\pd_{(E^j, F^j)} \rightarrow \lb E^j_0 \Big/ \lb E^j_0 \cap T_{u(0)}Z \rb \rb \oplus F^j_1.
\]
Letting $\eps_j$ denote the $j$th sign, the verification of the signs in \cref{thmCO} is completed by

\begin{lem} For all $j$ we have $\eps_j = +1$
\end{lem}
\begin{proof}
Let $\mu_j$ denote the partial indices of $(E^j, F^j)$, and $\conj{\pd}_j$ the Cauchy--Riemann operator.  If $\mu_j = 2$ then the intersection $E^j_0 \cap T_{u(0)}Z$ is zero and the result follows from \cref{ind2sign}.  On the other hand, if $\mu_j = 0$ then $(E^j, F^j) = (\C \theta \cdot u, \R \theta \cdot u)$, where $\theta = \theta_{j-(a+b)} \in \mathfrak{k}$, and the space $E^j_0 / (E^j_0 \cap T_{u(0)}Z)$ is zero.  A positively oriented basis for $\ker \conj{\pd}_j$ is given by $\theta \cdot u$, whilst a positively oriented basis for $F^j_1$ is given by $\theta \cdot u(1)$, so the required map is indeed orientation-preserving.

Finally suppose $\mu_j = (1, 1)$, so $E^j$ and $F^j$ are spanned by appropriate expressions built from $\alpha = \alpha_{j-a}$ and $\beta = \beta_{j-a}$, with $(\alpha + i \beta) \cdot u(0) = 0$.  To find a positively oriented basis for $\ker \conj{\pd}_j$ we degenerate our $(1, 1)$ pair to a $(2, 0)$ as follows.  Suppose we replace the condition $(\alpha + i \beta) \cdot u(0) = 0$ by $(\alpha + ti \beta) \cdot u(0) = 0$, and deform the real parameter $t$ from $1$ down to $0$.  We obtain a family of Riemann--Hilbert pairs over $[0, 1]$, and bases for the kernels of the corresponding Cauchy--Riemann operators are given by
\begin{equation}
\label{eq11Bas}
\frac{(1+z^2)\alpha + ti(1-z^2)\beta)}{z} \cdot u, \frac{i(1-z^2)\alpha - t(1+z^2)\beta)}{z} \cdot u, \alpha \cdot u, \beta \cdot u.
\end{equation}
By \cref{ind2sign}, and the fact that $\alpha \cdot u(1)$, $\beta \cdot u(1)$ is positively oriented as a basis of $F^j_1$, this basis is positively oriented in the limit $t \rightarrow 0$.  Hence by continuity of the orientation construction it is also positively oriented at $t=1$.

Now (restricting again to the $t=1$ situation, which is what we actually care about) consider the evaluation map to
\[
\lb E^j_0 \Big/ \lb E^j_0 \cap T_{u(0)}Z \rb \rb \oplus F^j_1 = \lb\lspan{\lim_{z \rightarrow 0} \frac{\alpha + i\beta}{z}\cdot u(z), \beta\cdot u(0)} \Big/ \lspan{\beta \cdot u(0)}\rb \oplus \lspan{\alpha \cdot u(1), \beta \cdot u(1)}_\R.
\]
The basis \eqref{eq11Bas} is sent to a positively oriented basis for this space, and so $\eps_j = +1$ as claimed.
\end{proof}

\bibliography{homogbiblio}
\bibliographystyle{utcapsor2}

\end{document}